\documentclass[11pt]{article}
\usepackage[toc,page]{appendix}
\usepackage{hyperref}
\usepackage[T1]{fontenc}
\usepackage{amsfonts}
\usepackage[english]{babel}
\usepackage{a4wide,times}

\usepackage[latin1]{inputenc}
\usepackage{amssymb}
\usepackage{graphicx}
\usepackage{epsfig}
\usepackage{amsbsy}
\usepackage{verbatim}
\usepackage{color}
\usepackage{mathrsfs}
\usepackage{makeidx}
\usepackage{amsmath}
\usepackage{amsthm}

\theoremstyle{plain}
\newtheorem{thm}{Theorem}[section]

\newtheorem{lem}[thm]{Lemma}
\newtheorem{prop}[thm]{Proposition}

\theoremstyle{definition}
\newtheorem{defn}{Definition}[section]
\usepackage{dsfont}

\theoremstyle{remark}
\newtheorem{rem}{\bf Remark}[section]
\theoremstyle{remark}

\newtheorem{com*}{\bf Comment}
\usepackage{fancyhdr}

\makeatletter
\def \newequation#1#2{
 \@definecounter{#1}
 \@namedef{the#1}{\hbox{#2}}
 \@namedef{#1}{$$\refstepcounter{#1}}
 \@namedef{end#1}{
    \eqno \csname the#1\endcsname $$\global\@ignoretrue
    }
}
\makeatother
\newequation{E1}{($E_{b,\sigma,W}$)}
 
\makeatletter
\def \newequation#1#2{
 \@definecounter{#1}
 \@namedef{the#1}{\hbox{#2}}
 \@namedef{#1}{$$\refstepcounter{#1}}
 \@namedef{end#1}{
    \eqno \csname the#1\endcsname $$\global\@ignoretrue
    }
 }
\makeatother
\newequation{hyp3}{($\mathcal{H}_{b,\sigma}$)}

\makeatletter
\def \newequation#1#2{
 \@definecounter{#1}
 \@namedef{the#1}{\hbox{#2}}
 \@namedef{#1}{$$\refstepcounter{#1}}
 \@namedef{end#1}{
    \eqno \csname the#1\endcsname $$\global\@ignoretrue
    }
 }
\makeatother
\newequation{shleg}{(ShLeg)}

\makeatletter
\def \newequation#1#2{
 \@definecounter{#1}
 \@namedef{the#1}{\hbox{#2}}
 \@namedef{#1}{$$\refstepcounter{#1}}
 \@namedef{end#1}{
    \eqno \csname the#1\endcsname $$\global\@ignoretrue
    }
 }
\makeatother
\newequation{kl}{(KL)}

\makeatletter
\def \newequation#1#2{
 \@definecounter{#1}
 \@namedef{the#1}{\hbox{#2}}
 \@namedef{#1}{$$\refstepcounter{#1}}
 \@namedef{end#1}{
    \eqno \csname the#1\endcsname $$\global\@ignoretrue
    }
 }
\makeatother
\newequation{haar}{(Haar)}

\newcommand{\E}{\mathds E}
\newcommand{\PP}{\mathds P}

\title{Improved error bounds for quantization based numerical schemes for BSDE and nonlinear filtering}
\author{ 
{\sc  Gilles Pag\`es} \thanks{Laboratoire de Probabilit\'es et Mod\`eles al\'eatoires (LPMA), UPMC-Sorbonne Universit\'e, UMR 7599, case 188, 4, pl. Jussieu, F-75252 Paris Cedex 5, France.
E-mail: {\tt gilles.pages@upmc.fr} } \ \ \ 
{\sc  Abass Sagna} \thanks{Laboratoire de Math\'ematiques et Mod\'elisation d'Evry (LaMME),   UMR 8071, 23 Boulevard de France, 91037 \'Evry,  \& ENSIIE. E-mail: {\tt abass.sagna@ensiie.fr}. }\ \ \
\thanks{The  first  author  benefited from the support of the Chaire ``Risques financiers'',  a joint initiative of \'Ecole Polytechnique, ENPC-ParisTech and UPMC, under the aegis of the Fondation du Risque. The second author  has benefited from the support of the Chaire ``Markets in Transition'', under the aegis of Louis Bachelier Laboratory, a joint initiative of \'Ecole Polytechnique, Universit\'e d'\'Evry Val d'Essonne and  F\'ed\'eration Bancaire Fran\c{c}aise.} 
}

\date{}

\begin{document}

\maketitle

\begin{abstract}
We take advantage of recent  (see~\cite{GraLusPag1, PagWil}) and new results on optimal quantization theory to improve the  quadratic optimal quantization error bounds for   backward stochastic differential equations (BSDE) and nonlinear filtering problems.  For both problems, a first improvement relies on a Pythagoras like Theorem for quantized conditional expectation.  While allowing  for  some locally Lipschitz continuous conditional densities   in nonlinear filtering,  the analysis of the  error brings  into play a new robustness result about  optimal quantizers, the so-called distortion mismatch  property: the $L^s$-mean quantization error induced by $L^r$-optimal quantizers  of size $N$ 
%used to  approximate  $\mathds R^d$-valued random vectors in $L^2$ by a nearest neighbor projection (Voronoi quantization) 
converges  at the same rate $N^{-\frac 1d}$ for every $s\!\in (0,r+d)$.  
\end{abstract}

\section{Introduction}
 In this work we propose  improved  error bounds for quantization  based numerical schemes introduced in~\cite{BalPag} and~\cite{PagPha} to solve BSDEs and nonlinear filtering problems.   For  BSDE, we consider equations where the driver depends on the ``$Z$'' term (see Equation~\eqref{EqBSDEIntro} below) and for nonlinear filtering, we extend existing results  to locally Lipschitz continuous densities (see Section~\ref{SecNonLinearFilt}).  For both problems, we also improve the error bounds themselves   by using a Pythagoras like theorem for the approximation of conditional expectations introduced in~\cite{PagWil} (see also~\cite{PagCER}).    These problems  have a wide range of applications, in particular in Financial Mathematics, when modeling the price of financial derivatives or  in stochastic control, in credit risk modeling, etc.

BSDEs were first  introduced in~\cite{Bis} but raised a wide interest mostly after their extension in~\cite{ParPen}. In this latter paper,  the existence and the uniqueness of a solution have been  established for  the following backward  stochastic differential equation  with Lipschitz continuous driver $f$ (valued in $\mathds R^d$) and terminal condition $\xi$:
\begin{equation}  \label{EqBSDEIntro}
Y_t  = \xi +\int_t^T f(s,Y_s,Z_s) ds - \int_t^T Z_s   dW_s, \qquad  0 \le   t \leq T,
\end{equation} 
where $W$ is a $q$-dimensional brownian motion. We mean by  a solution a pair $(Y_t,Z_t)_{t  \leq T}$  (valued in $\mathds R^d \times \mathds R^{d \times q}$) of square integrable   progressively measurable  (with respect to   the augmented Brownian filtration  $({\cal F}_t)_{t \ge 0}$)  and  satisfying Equation~(\ref{EqBSDEIntro}).  Extensions of these existence and uniqueness results  have been investigated in  more general  situations such as: less regular drivers (locally Lipschitz driver, see~\cite{Bah, Ham}; quadratic BSDEs, see~\cite{Kob};  rough integrals instead of Lebesgue integrals,  see~\cite{DieFri}; super-linear quadratic BSDEs, see~\cite{LepSan}); randomized horizon, see~\cite{Par}; introduction of Poisson random measure component  subject to constraints on the jump component, see~\cite{KhaJinPhaZha, KhaEli};  extension to  second order  BSDEs, see~\cite{SonTouZha});   non-smooth terminal conditions, see~\cite{GeiSte, GeiGeiGob, GobMak}. 

Since the pioneering work~\cite{ElkPenQue} in which the link between BSDE and hedging portfolio of European  (and American) derivatives has been first established, various other applications have been developed, as  risk-sensitive control problems, risk measure theory, etc.  

%   One of the main fields of application of BSDE is  Financial Mathematics where  it may be related to many problems such as  pricing and hedging of European options,   risk-sensitive control problems,  risk measure theory, etc (see $e.g.$~\cite{ElkPenQue} where the  link between BDSE and Mathematical Finance has been established  for the  first time).

However, even if   it can be established  in many cases   that a BSDE has a unique solution, this solution admits no closed form  in general. This led to devise  tractable approximation schemes of the solution.   In the Markovian case (see~\eqref{EqBSDEIntro1} below) for example, where the terminal condition is of the form  $\xi = h(X_T)$ for some forward diffusion $X$, a first  numerical method has  been  proposed in~\cite{DouMaPro} for a class of  forward-backward  stochastic differential equations,  based on a four step scheme developed later on  in~\cite{MaProYon}.  

%In~\cite{Zha}, a numerical scheme for  has been investigated, see also.    
Many other  approximation methods of the solutions of some   classes of BSDEs have been proposed such as BSDEs with possible path-dependent terminal condition, see~\cite{Zha} and~\cite{BrLa},  coupled BSDE, see~\cite{ZhaFuZho}, Reflected BSDE, see~\cite{Huta}, BSDE for quasilinear PDEs, see~\cite{DelMen}, BSDE applied to control  problems or nonlinear PDEs, see~\cite{Bal}. Higher order schemes have also been considered, see~\cite{ChaCri} or~\cite{BaPaPr2}.  As for  quadratic BSDE~--~$i.e.$ when the generator $f$ has a  quadratic growth with respect to $z$~--~we refer to~\cite{ChaRic} where the authors  consider a slightly modified  dynamical programming equation to propose a numerical scheme. They  investigate the time  discretization error and use optimal quantization  to implement their algorithm. However, they do not study the  induced quantization error.

%Note in fact that in \cite{ChaRic}, approximation of~\eqref{EqBSDEIntro} when . 

   In the present work, we  consider the following  decoupled FBSDE (Forward-Backward SDE),  
   \begin{equation}   \label{EqBSDEIntro1}
 Y_t  =  \xi  + \int_t^T f(s,X_s,Y_s,Z_s) ds - \int_t^T Z_s \cdot dW_s,\quad t \in [0,T], 
 \end{equation}
 where  $W$ is a $q$-dimensional Brownian motion,  $(Z_t)_{t \in [0,T]}$ is a square integrable progressively measurable process taking values in $\mathds R^q$ and   $f : [0,T]{\small \times} \mathds R^d $  ${\small \times} \mathds R {\small \times \mathds R^q} \rightarrow \mathds R$ is a Borel function.   We suppose  a terminal condition of the form  $\xi = h(X_T)$, for a given Borel function $h:\mathds R^d \rightarrow \mathds R$,  where  $X_T$ is the value at time $T$ of a Brownian diffusion process $(X_t)_{t \geq 0}$,  strong solution to the  SDE:
 \begin{equation} \label{EqXIntro}
 X_t  = x  +  \int_0^t b(s,X_s)ds  + \int_0^t \sigma(s,X_s)  dW_s,   \qquad x \in \mathds R^d.
 \end{equation}   
In this case, the solution of the BSDE is usually  approximated at the points of a time grid  $t_0=0, \ldots,t_n=T$  and involve in particular  the approximation of conditional expectations   $\mathds E(g_{k+1}(X_{t_{k+1}}) \vert X_{t_k})$, where the functions $g_{k}$ are determined in the recursion. The sequence $(X_{t_{k+1}})_{0\le k\le n}$ is either a ``sampling" of the diffusion $X$ at times $(t_k)_{0\le k\le n}$ or, most often, a  discretization scheme  of $(X_{t})_{t  \geq 0}$, typically the discrete time Euler scheme, when the solution of~\eqref{EqXIntro} is not explicit enough to be simulated in an  exact way.  

 In this paper,  we  consider  an  {\em explicit} time discretization scheme where the conditioning is performed {\em inside} the driver $f$ (see also~\cite{IllThese}). It is recursively defined in a backward way as: 
\begin{eqnarray}
   \tilde Y_{t_n} &  = & h(\bar X_{t_n})  \label{IntroEqDiscBSDE1}\\
   \tilde Y_{t_k}  & = & \mathds E( \tilde Y_{t_{k+1}} \vert \mathcal F_{t_k}) +  \Delta_n f \big(t_k, \bar X_{t_k}, \mathds E( \tilde Y_{t_{k+1}}  \vert \mathcal F_{t_k}), \tilde  \zeta_{t_k} \big), \quad k=0, \ldots,n-1,    \qquad\label{IntroEqDiscBSDE2}\\
\mbox{with }\qquad\qquad\tilde \zeta_{t_k} &=&  \frac{1}{\Delta_n} \mathds E\big( \tilde Y_{t_{k+1}} (W_{t_{k+1}}-W_{t_k} )  \vert \mathcal F_{t_k}\big), \quad k=0, \ldots,n-1. \label{Introeq:zeta}
  \end{eqnarray}
  The process  $(\bar X_{t_k})_{k=0, \ldots, n}$  is  the discrete time Euler scheme  of  the diffusion process $(X_t)_{t \in [0,T]}$ with step $\Delta_n = \frac Tn$,  recursively defined  by 
\[
\bar X_{t_{k}}=  \bar X_{t_{k-1}} +  \Delta_nb(t_{k-1},\bar X_{t_{k-1}} )   + \sigma(t_{k-1},\bar X_{t_{k-1}}) (W_{t_{k}} - W_{t_{k-1}}), \; k=1,\ldots,n,\; \bar X_0=x. 
\]
 Under some smooth assumptions  on the coefficients of the diffusions, there exists (see Theorem~\ref{thm:bartilde}  further on for a precise statement, see also~\cite{Zha, BouTou})  that   there is a real constant $\tilde C_{b,\sigma,f,T}>0$ such that, for every $n\ge 1$,  
\[
\max_{ k \in\{0,\ldots, n \}}    \E |   Y_{t_k}-\tilde Y_{t_k} |^2 +\int_0^T   \E   |  Z_t-\tilde Z_t|^2dt   \le  \tilde C_{b,\sigma,f,T} \Delta_n,
 \]
where $\tilde Z = \tilde Z^{(n)}$ comes from the martingale representation of $\sum_{k=1}^n \tilde Y_{t_k} - \mathds E(\tilde Y_{t_k} \vert {\cal F}_{t_{k-1}}).$  
%A proof of this result is provided in  an extended version~\cite{} posted on ArXiv.

At this stage, since the scheme~\eqref{IntroEqDiscBSDE1}-\eqref{IntroEqDiscBSDE2} involves the computation of  conditional expectations  for which   no analytical   expression is available, its    solution $(\tilde Y, \tilde \zeta)$  has in turn  to be approximated.  A possible approach is to  rely on regression methods involving  the Monte Carlo simulations, see  $e.g.$~\cite{BouTou, GobLemWar,GobLopTurVaz, GobLiuZub}.  Other methods using  on line  Monte Carlo simulations   have been developed  in a Malliavin calculus framework (conditional expectations are ``regularized'' by integration by parts from which ``Malliavin'' weights come out, see~\cite{BouTou, CriManTou, GobTur2, HuNuaSon, Tur}).  Among alternative approaches let us cite the least-squares regression methods, the multistep schemes methods  (see~\cite{BenDen, GobTur}), the primal-dual approach (see~\cite{BenSchZhu}). New approaches have been proposed recently: a combination of Picard iterates and a decomposition in Wiener chaos (see~\cite{BrLa14}), a ``forward" approach in connection with the semi-linear PDE associated to the BSDE (see~\cite{HLTaTo14}), an analytic approach in~\cite{GobPag}.

 In this paper, we go back to the {\em optimal quantization tree}  approach originally  introduced in~\cite{BalPagPri0} (in fact for {\em Reflected} BSDEs) and developed in \cite{BalPag, BalPag1, BalPagPri1}. This approach  is based on an optimally fitting approximation of the Markovian dynamics of the discrete time Markov chain $(\bar X_{t_k})_{0\le k\le n}$ (or a sampling of $X$ at discrete times $(t_{k})_{k=0, \ldots,n}$) with random variables having a finite support. However, we consider a different quantization tree  (or quantized scheme)   defined recursively  by mimicking~\eqref{IntroEqDiscBSDE1}-\eqref{IntroEqDiscBSDE2}  as follows: 
 \begin{eqnarray}
  \hat Y_{t_n} &  = & h(\hat X_{t_n} ) \label{EqhatYkIntro1} \\
  \hat Y_{t_k}  & = & \hat {\mathds E}_k ( \hat Y_{t_{k+1}})  +  \Delta _n   f \big(t_k, \hat X_{t_k}, \hat{\mathds E}_k(\hat Y_{t_{k+1}}), \hat {\zeta}_{t_k} \big)    \label{EqhatYkIntro2}\\
\mbox{with } \hskip 3cm      \hat {\zeta}_{t_k} & = & \frac{1}{\Delta_n} \hat{\mathds E}_k(\hat Y_{t_{k+1}}   \Delta W_{t_{k+1}}), k=0, \ldots, n-1, \hskip 4cm \label{Introeq:zetaquant}
  \end{eqnarray}
where  $\Delta W_{t_{k+1}}= W_{t_{k+1}}-W_{t_k}$,  $\hat{\mathds E}_k = \mathds E (\,\cdot\, \vert \hat X_{t_k})$,   and  $\hat X_{t_k}$     is a quantization of $\bar X_{t_k}$ on a finite  {\em grid} $\Gamma_k\subset \mathbb R^d$,  $i.e.$, $\hat X_{t_k}  = \pi_k(\bar X_{t_k})$, where $\pi_k: \mathds R^d \rightarrow \Gamma_k$ is a  Borel ``projection" on $\Gamma_k$, $k=0, \ldots, n$. At this stage the  function  $\pi_k$ might be any $\Gamma_k$-valued Borel functions. In order to derive better theoretical rates as well as  for practical implementation, we will first consider Borel  {\em nearest neighbor projections} $\pi_k= {\rm Proj}_{\Gamma_k}$ at every time step and then search for grids optimally ``fitting'' the distribution of $X_{_k}$ $i.e.$ minimizing the resulting error $\|X_{t_k}- {\rm Proj}_{\Gamma_k}(X_{t_k})\|_{2}$ among all grids $\Gamma_k$ of a prescribed  size $N_k$, see Section~\ref{SecOptiQuant} for details. This is an {\em explicit inner} scheme in the sense that the conditioning  is performed inside the driver $f$ in contrast with  what is usually done in the literature (where implicit  or {\em outer} explicit  schemes are in force). This scheme, though quite natural,  seems not to have been extensively analyzed (see however~\cite{IllThese} where a first  analysis is carried out  in the spirit of \cite{BalPag, BalPag1}). It turns out to be  well designed  to establish our improved rates and shows  quite satisfactory numerical performances.
Our objective here is two-fold: first  include the $Z$ term in the driver and to dramatically   improve the error bounds in 
\cite{BalPag1, BalPagPri1}, especially its dependence in the size $n+1$ of the time discretization mesh.

So, the question of interest will  be to estimate   the  quadratic  quantization error $(\mathds E \vert \tilde Y_{t_k}  -  \hat Y_{t_k} \vert^2)^{1/2}$  induced by the approximation of $\tilde Y_{t_k}$ by $\hat Y_{t_k}$, for every $k=0, \ldots,n$,  where $\hat Y_{t_k}$ is the quantized version of $\tilde Y_{t_k}$ given by~\eqref{EqhatYkIntro1}-\eqref{EqhatYkIntro2}.    Under more general assumptions than~\cite{BalPag,BalPagPri0}, we show  in Theorem~\ref{TheoremPrincBsde}$(a)$ that, at every step $k$ of the procedure,
\begin{equation} \label{EqControlQuantErrorIntro}
   \big \Vert    \tilde Y_{t_k} - \hat Y_{t_k} \big \Vert _2^2     \leq     \sum_{i=k}^{n}       \widetilde K_i   \big \Vert\bar X_{t_i} -\hat X_{t_i}  \big \Vert_2^2, 
 \end{equation}
for  positive real constants $\widetilde K_i$ depending on  $t_i$ and $T$ and on the regularity of the coefficients of $b, \sigma$ and the driver $f$ which remain bonded as $n\uparrow +\infty$.     The presence of  the squared quadratic norms  on both sides of~\eqref{EqControlQuantErrorIntro} improves the control of the time discretization effect, compared with~\cite{BalPag, BalPagPri0} in which error bounds of the form $  \Vert \tilde Y_{t_k} - \hat Y_{t_k}   \Vert _p     \leq     \sum_{i=k}^{n}         K_i  \Vert\bar X_{t_i} -\hat X_{t_i}  \Vert_p$ are established for $p\!\in [1,+\infty)$. In fact,  we switch  from a global error  (at $t=0$) of order $n\times \max_{0\le k\le n}\Vert X_{t_k} - \hat X_{t_k} \Vert_2$  to   $\sqrt{n}\times\max_{0\le k\le n}\Vert X_{t_k} - \hat X_{t_k} \Vert_2$. This theoretical improvement confirms the results of numerical experiments first carried out in~\cite{BalPagPri1} though it was in a less favorable framework (with reflection) or in ~\cite{BroPagWil, BrPaPo} for American options. 
 %\textcolor{red}{They were then   extensively investigated in~\cite{BroPagWil, BrPaPo} (for American options) to devise a    
 %Romberg extrapolation combining two time discretization steps which dramatically improves  the performances of such schemes. }

  For the $Z$ part  which is approximated first by $\tilde{\zeta}$ in \eqref{Introeq:zeta} and whose quantization version $\hat{\zeta}$ is given by \eqref{Introeq:zetaquant}, we get the following approximation error (see Theorem~\ref{TheoremPrincBsde}$(b)$)
\[
\Delta_n \sum_{k=0}^{n-1} \Vert\tilde \zeta_k-\hat \zeta_k\Vert_2^2\le \sum_{k=0}^{n-1}\tilde{K}_k' \Vert \bar X_k-\hat X_k\Vert_2^2+ \sum_{k=0}^{n-1}\Vert \tilde Y_{k+1}- \hat Y_{k+1}\Vert_2^2
\]
where $\widetilde K_i'$  are positive constants depending on  $t_i$ and $T$ and on the regularity of the coefficients of $b, \sigma$ and the driver $f$. So,  we switch from  $n^{\frac 32}\times \max_{0\le k\le n}\Vert \bar X_{t_k} - \hat X_{t_k} \Vert_2$  to $n\times \max_{0\le k\le n}\Vert \bar X_{t_k} - \hat X_{t_k} \Vert_2$.

We notice here that other quantization based discretization schemes have been devised, especially  for {\em Forward-Backward} SDEs (see~\cite{DelMen1}) where  the diffusion and the BSDE are fully coupled (including the $Z$ in the driver) where the grids $\Gamma_k$ are the trace of  $\delta \mathbb{Z}^d$ ($\delta>0$) on an expanding compact as $t_k$ grows. In contrast the Brownian increments are replaced by optimal quantization of the ${\cal N}(0;I_d)$-distribution.  But the obtained resulting error bound for the  scheme are not of the improved from~\eqref{EqControlQuantErrorIntro}. A multistep approach based on two reference ODEs from the computation of conditional expectation has been developed  in a similar framework (coupled and uncoupled) in~\cite{ZhaFuZho}.

In the second part  of the paper, we first propose  (Section~\ref{SecOptiQuant}) a short background on optimal vector quantization, enriched by a new result, namely Theorem~\ref{thm:Mismatchnew}, which essentially solves the-called {\em distortion mismatch} problem.
By distortion mismatch we mean the robustness of optimal quantization  grids. An optimal (quadratic) quantization grid  $\Gamma_N$ {\em at level $N$}  for the distribution of a random vector $X$ is  such that $\|X-{\rm Proj}_{\Gamma_N}(X)\|_2=e_{N,2}(X):=\inf \big\{\|X- q(X)\|_2,\, q:\mathbb{R}^d\to \Gamma, \mbox{Borel},\,\Gamma\subset \mathbb R^d, \, {\rm card}(\Gamma_N)\le N\big\}$ where ${\rm Proj}_{\Gamma_N}$ denotes a (Borel)  nearest neighbor projection on $\Gamma_N$. It exists for every size (or level) $N\ge 1$ as soon as $X\!\in L^2$ and it follows  from  Zador's Theorem  that $e_{N,2}(X)\sim c(X)N^{-\frac 1d}$ as $N\to+\infty$ (see Section~\ref{SecOptiQuant} for  details).  The distortion mismatch property established in Theorem~\ref{thm:Mismatchnew} states that, for every $s\!\in (0,d+2)$,  $\varlimsup_N N^{\frac 1d}\|X-{\rm Proj}_{\Gamma_N}(X)\|_s<+\infty$. This result holds  whenever $X\!\in L^s$ with a distribution satisfying mild additional property. This theorem extends first results established in~\cite{GraLusPag1} for various classes of absolutely continuous distributions. Note that all the above  properties depend on the distribution $\mathbb P_{_X}$ of $X$ rather than on the random vector $X$ itself.
 This robustness property is the key of the  second kind of  improvement  proposed in this paper, this time for quantization based schemes for non-linear filtering investigated in the third part.  In Section~\ref{sec:Numeric} we propose  numerical illustrations using optimal quantization based schemes for various types  of BSDEs  which confirm that the improved rates established in the first part are the true ones.

In this third  part of the paper (Section~\ref{SecNonLinearFilt}), we consider   a (discrete time) nonlinear filtering problem and improve (in the quadratic setting)  the results obtained in~\cite{PagPha}. Firstly, we relax the Lipschitz assumption made on the conditional densities then we provide new improved error bounds for the quantization based scheme introduced in~\cite{PagPha}  to numerically solve a discrete filter by optimal quantization.  

In fact, we consider a discrete time nonlinear filtering problem where the signal process $(X_k)_{k \geq 0}$ is an $\mathds R^d$-valued discrete time Markov process  and the  observation process $(Y_k)_{k \geq 0}$ is an $\mathds R^q$-valued random vector, both  defined  on a probability space $(\Omega, \mathcal A, \mathbb P)$.  The  distribution $\mu$ of $X_0$ is given,  as well as the transition probabilities $P_k(x,dx')= \mathds P(X_{k}\in dx'|X_{k-1}=x)$ of the process $(X_k)_{k \geq 0}$.  We also suppose that  the process  $(X_k,Y_k)_{k \geq 0}$ is a Markov chain and that for every $k \geq 1$, the conditional distribution  of $Y_k,$ given  $(X_{k-1},Y_{k-1},X_k)$ has a density  $g_k(X_{k-1},Y_{k-1},X_k, \cdot)$.
Having  a fixed  observation $Y:=(Y_0,\ldots,Y_n) =(y_0,\ldots,y_n)$, for $n \geq 1$, we aim at computing the conditional  distribution $ \Pi_{y,n}$ of $X_n$ given $Y=(y_0, \ldots,y_n)$. It is well-known that for any bounded and measurable function $f$, $\Pi_{y,n} f $ is given by
%we can derive 
the celebrated Kallianpur-Striebel formula (see $e.g.$~\cite{PagPha})
\begin{equation}
\Pi_{y,n} f  =  \frac{\pi_{y,n} f }{ \pi_{y,n} \mbox{\bf 1}}
\end{equation}
where the so-called un-normalized filter $\pi_{y,n}$ is defined for every bounded or non-negative Borel function $f$ by
$$
\pi_{y,n}  f = \mathds E(f(X_n) L_{y,n})
$$
with $$L_{y,n} = \prod_{k=1}^n g_k(X_{k-1},y_{k-1},X_k,y_k).
$$
Defining the family of transition kernels $H_{y,k}$, $k=1,\ldots,n$, by 
\begin{equation} \label{eq:Hyk}
H_{y,k} f(x) = \mathds E\,\big(f(X_k) g_{k}(x,y_{k-1},X_k,y_k)  \vert X_{k-1}=x\big)
\end{equation}
for every bounded or non-negative Borel function $f:\mathds R^d \to \mathds R$ and setting  
$$  
H_{y,0} f(x) = \mathds E(f(X_0)), 
$$
one shows that the un-normalized filter may be computed by the following {\em forward} induction formula: 
\begin{equation} \label{EqForwardInduction}
\pi_{y,k}  f= \pi_{y,k-1} H_{y,k} f, \qquad k=1,\ldots, n,
\end{equation}
with $\pi_{y,0} = H_{y,0}$. A useful formulation, especially   to establish error bound for the quantization based approximate filter is its {\em backward} counterpart 
% turns out to be the  {\em backward} induction formula, 
defined by setting 
$$ \pi_{y,n} f = u_{y,-1} (f)$$
where $u_{y,-1}$ is the final value of the backward recursion:
\begin{equation} \label{EqBackwardInduction}
  u_{y,n}(f)(x)  = f(x), \quad u_{y,k-1}(f) = H_{y,k} u_{y,k} (f), \quad k=0, \ldots,n. 
\end{equation} 
In order to compute the normalized filter $\Pi_{y,n}$, we just have to compute the transition kernels $H_{y,k}$ and to use the recursive formulas~(\ref{EqForwardInduction}) or~(\ref{EqBackwardInduction}). However these kernels have no closed formula in general so that we have to approximate them.  Optimal quantization based algorithms  for non linear filtering has been introduced in~\cite{PagPha} (see also~\cite{PhaRunSel, CalSag, PagPha, Sel} for further developments and contributions).  It turned out to be an efficient alternative approach to particle methods (we refer $e.g.$ to~\cite{DelJacPro} and the references therein which rely  on Monte Carlo simulation of interacting particles)    owing to its tractability. For a survey and comparisons between optimal quantization and particle methods, we refer to~\cite{Sel}). 

 The quantization based approximate filter is designed as follows:  denoting for every $k =0, \ldots,n$ by $\hat{X}_k$ a quantization of  $X_k$ at level $N_k$ by  the grid $\Gamma_k = \{x_k^1, \ldots,x_k^{N_k} \}$,   we will formally replace  $X_k$ in ~(\ref{EqForwardInduction}) or~(\ref{EqBackwardInduction}) by  $\hat X_k$. As a consequence the (optimally) quantized approximation  $\hat{\pi}_{y,n}$ of $\pi_{y,n}$   is    defined  simply by the  quantized  counterpart of the Kallianpur-Striebel formula: we introduce for every bounded or non-negative Borel function $f:\mathds R^d\to \mathds R$ the family of  quantized transition kernels $\hat H_{y,k}$, $k=0,\ldots,n$, by $\hat H_{y,0} f(x) = \mathds E(f(\hat X_0))$ and 
\begin{eqnarray} \label{eq:Hyk}
\hat H_{y,k} f(x_{k-1}^i) &= &\mathds E\big(f(\hat X_k) g_{k}(x^i_{k-1},y_{k-1},\hat X_k,y_k)  \vert X_{k-1}=x_{k-1}^i\big),\quad k=1,\ldots,n. \\
&=&Ê\sum_{j=1}^{N_k} \hat H^{ij}_{y,k} f(x^j_k), \; i=1,\ldots,N_{k-1} \\
\mbox{with} \qquad \hat H^{ij}_{y,k} &=&  g_{k}(x^i_{k-1},y_{k-1};  x^j_k, y_{k})  \,  \hat{p}_k^{ij}\\
\mbox{and } \qquad  \hat{p}_k^{ij}&=&  \mathds P(\widehat X_k= x^j_k\,|\, \hat X_{k-1}= x_{k-1}^i)\quad i=1,\dots,N_{k-1}, j=1,\dots,N_k.
\end{eqnarray}
 Then set
\begin{equation} \label{EqForwardInductionQuant}
\hat{ \pi}_{y,k}  = \hat{\pi}_{y,k-1} \hat{H}_{y,k}, \quad k=1,\ldots, n, \quad\mbox{ and }\quad \hat{\pi}_{y,0}= \hat {H}_{y,0}
\end{equation}
or, equivalently, 
\[
\hat{ \pi}_{y,k}  = \sum_{i=1}^{N_k} \hat \pi^i_{y,k} \delta_{x^i_k}\quad \mbox{ with }\quad \hat \pi^i_{y,k} = \sum_{j=1}^{N_{k-1}}\hat \pi^j_{y, k-1}\hat H^{ij}_{y,k}, \; k=1,\ldots,n
\]
and $\hat \pi_0=\sum_{i=0}^{N_0}  \hat p_0^i\delta_{x^i_0}$ with $\hat p_0^i= \mathds P(\hat X_0= x^j_0)$, $i=1,\ldots, N_0$. 
As a final step, we  approximate the normalized filter $\Pi_{y,n}$ by   $\hat{\Pi}_{y,n}$   given by
\[
\hat \Pi_{y,n} f  =  \frac{\hat \pi_{y,n} f }{\hat  \pi_{y,n} \mbox{\bf 1}}= \sum_{i=1}^{N_n} \hat{\Pi}^i_{y,n} f(x_n^i)\quad \mbox{ with } \quad \displaystyle \hat{\Pi}^i_{y,n}  = \frac{\hat {\pi}_{y,n}^i}{\sum_{j=1}^{N_n}  \hat{\pi}_{y,n}^j}, \; i=1,\dots,N_n.
\]
One shows (see~\cite{PagPha}) that the un-normalized quantized filter may also be computed by the following    {\em backward} induction formula, 
defined by 
$$
\hat  \pi_{y,n} f = \hat u_{y,-1} (f)$$
where $\hat u_{y,-1}$ is the final value of the backward recursion:
\begin{equation} \label{EqBackwardInductionhat}
  \hat u_{y,n}(f)  = f\;\mbox{ on }\Gamma_n, \quad \hat u_{y,k-1}(f) = \hat H_{y,k} \hat u_{y,k} (f)\;\mbox{ on }\Gamma_k, \quad k=0, \ldots,n. 
\end{equation} 

Our aim is then to estimate the quantization error induced by  the approximation of  $\Pi_{y,n}$ by $\hat{\Pi}_{y,n}$. Note that this problem has been considered in~\cite{PagPha} where it has been shown that, for every bounded Borel  function $f$, the absolute error $\vert \Pi_{y,n}f - \hat{\Pi}_{y,n}f  \vert$ is  bounded  (up to a constant depending in particular on $n$) by the cumulated  sum of the $L^r$-quantization errors $\Vert  X_k - \hat {X}_k \Vert_{r}$, $k=0,\ldots,n$.  In this work, we  improve this result in the particular case of the quadratic quantization framework ($i.e.$  $r=2$) in two directions.  In fact, we first  show  that,  for every bounded Borel  function $f$, the squared-absolute error $\vert \Pi_{y,n}f - \hat{\Pi}_{y,n}f  \vert^2$ is  bounded   by the cumulated  square-quadratic quantization errors $\Vert  X_k - \hat {X}_k \Vert_{2}^2$ from $k=0$ to $n$,   similarly to what  we did for BSDEs inducing  a similar  improvement for dependence in $n$ of  the error bounds ($i.e.$ the time discretization step $1/n$ if $(X_k)_{k\ge 0}$ is a discretization step of a diffusion). Once again, this confirms numerical evidences observed in~\cite{PagPha, BalPagPri1}. Secondly, we show that  these  improved error bounds   hold 
under local Lipschitz continuity assumptions on the conditional density functions $g_k$ (instead of Lipschitz conditions in~\cite{PagPha}). The  distortion mismatch property established in Theorem~\ref{thm:Mismatchnew} is the key of this extension.

The paper is divided into three parts. The first part is devoted to the analysis of the optimal quantization error associated to the BSDE algorithm under consideration. We recall first, in Section~\ref{SecDiscretBSDE}, the discretization scheme we consider for the BSDE. Then, in Section~\ref{SecErrorAnalBSDE}, we  investigate the error analysis for  the time discretization and the quantization scheme.  In the second part, some results about optimal quantization are recalled in Section~\ref{SecOptiQuant} and a new distortion mismatch theorem is established about the robustness of $L^r$-optimal quantization in $L^s$ for $s\!\in (r, r+d)$.  Some numerical tests  confirm and illustrate  these improved error bonds in Section~\ref{sec:Numeric}.  The final part, Section~\ref{SecNonLinearFilt}, is devoted to the nonlinear filtering problem analysis when estimating the nonlinear filter by optimal quantization with new improved error bounds obtained under less than stringent~--~local~--~Lipschitz assumptions than in the existing literature. 

\noindent {\sc Notations:}  $\bullet$ $|\,.\,|$ denotes the canonical Euclidean norm on $\mathbb R^d$.

\smallskip 
\noindent $\bullet$ For every $f:\mathbb R^d \to \mathbb R$, set $\|f\|_{\infty} = \sup_{x\in \mathbb R^d} |f(x)|$ and $[f]_{\rm Lip}= \sup_{x\neq y}\frac{|f(x)-f(y)|}{|x-y|}
\le +\infty$. 

\smallskip
\noindent $\bullet$ If $A\!\in {\cal M}(d,q)$ we define the Fr\"obenius norm of $A$ by $\| A\|= \sqrt{{\rm Tr}(AA^*)}$. 

%\smallskip
%\noindent $\bullet$  For every $r\ge 0$, we 
%define $ L^{r+}(\mathbb{P}) = {\bigcup}_{\eta>0}L^{r+\eta}(\mathbb{P})$. 

\section{Discretization of  the BSDE}  \label{SecDiscretBSDE}

%\subsection{A background on discretization of quadratic BSDE}  \label{SubsectBackgra}
% \subsubsection{From continuous time BSDE quantization based numerical schemes}
 
 Let  $(W_t)_{t \geq 0}$ be a $q$-dimensional Brownian motion defined on a probability space $(\Omega, \mathcal A, \mathds P)$ and let  $(\mathcal F_t)_{t \geq 0}$ be its  augmented natural filtration.  We consider the following  stochastic differential equation:
 \begin{equation}\label{eq:diffusion}
 X_t  = x  +  \int_0^t b(s,X_s)ds  + \int_0^t \sigma(s,X_s)  dW_s, 
 \end{equation}
 where the drift coefficient  $b:[0,T] \times \mathbb R^d \to \mathbb R^d$ and the matrix diffusion coefficient  $\sigma:[0,T] \times \mathbb R^d \to{\cal M}(d,q)$ are Lipschitz continuous in $(t,x)$. For a fixed horizon (the maturity) $T>0$,  we consider the following  Markovian Backward Stochastic Differential Equation (BSDE):
 \begin{equation}   \label{EqBSDE}
 Y_t  = h(X_T)  + \int_t^T f(s,X_s,Y_s,Z_s) ds - \int_t^T Z_s \cdot dW_s,\quad t \in [0,T], 
 \end{equation}
 where the function $h: \mathds R^d \rightarrow \mathds R$ is $[h]_{\rm Lip}$-Lipschitz continuous, the   driver  $f(t,x,y,z) : [0,T]{\small \times} \mathds R^d $  ${\small \times} \mathds R {\small \times \mathds R^q} \rightarrow \mathds R$  is  Lipschitz continuous with respect to $(x, y,z)$,  uniformly in   $t\!\in[0,T]$, $i.e.$ satisfies
 \begin{eqnarray}  \label{AssLipschitzGenerator}
({\rm Lip}_f)\; \equiv\;  \vert  f(t,x,y,z)  - f(t,x',y',z')  \vert  &  \leq &   [f]_{\rm Lip} ( \vert x -x' \vert + \vert  y -y'  \vert +  \vert  z-z' \vert ).   \label{AsLipsch_f}
 \end{eqnarray}
 
Under the previous assumptions  on  $b,  \sigma, h, f$,  the  BSDE~\eqref{EqBSDE} has a unique   $\mathds R {\small \times} \mathds R^q$-valued, ${\cal F}_t$-adapted solution $(Y,Z)$  satisfying (see~\cite{ParPen}, see also~\cite{MaYon})
 \[
 \E \Big(\sup_{t\in[0,T]}|Y_t|^2 +\int_0^T |Z_s|^2ds\Big)<+\infty.
 \]

Let us consider now $(\bar X_{t_k})_{k=0, \ldots, n}$,  the discrete time Euler scheme  with step $\Delta_n=\frac Tn$ of  the diffusion process $(X_t)_{t \in [0,T]}$:
% where  
%$(t_k)_{0\le k\le n}$ is  the uniform mesh  of the interval $[0,T]$ defined by 
%$t_k : = \frac{k T}{n}$, $k=0, \ldots, n$.  It reads
\[
\bar X_{t_{k}}=  \bar X_{t_{k-1}} +  \Delta_nb(t_{k-1},\bar X_{t_{k-1}} )   + \sigma(t_{k-1},\bar X_{t_{k-1}}) (W_{t_{k}} - W_{t_{k-1}}), \; k=1,\ldots,n,\; \bar X_0=x
%\quad t \!\in [t_{k},t_{k+1})
\]
  where  
%$(t_k)_{0\le k\le n}$ is  the uniform mesh  of the interval $[0,T]$ defined by 
$t_k : = \frac{k T}{n}$, $k=0, \ldots, n$ and its continuous time counterpart, sometimes called {\em genuine Euler scheme}  (we drop the dependence in $n$ when no ambiguity)
%  We will need to extend this Euler   schemes into a continuous time process $(\bar X^n_t)_{t\in [0,T]}$by setting
defined as an It\^o process by 
%\begin{equation}\label{eq:EulerCont}
%\bar X_t  =  \bar X_{t_{k-1}} +( t - t_{k-1}) b(t_{k-1},\bar X_{t_{k-1}} ) + \sigma(t_{k-1},\bar X_{t_{k-1}}) (W_t - W_{t_{k-1}}), \; t \in [t_{k-1},t_{k}],\; k=1,\ldots,n.
%\end{equation}
%If we set $\underline t = \frac {kT}{n}$ when  $t\!\in [t_k,t_{k+1})$,  one easily checks that $(\bar X_t)_{t\in [0,T]}$ is an It\^o process satisfying
\begin{equation}\label{eq:EulerCont}
d\bar X_t =  b(\underline t, \bar X_{\underline t}) dt+\sigma(\underline t,\bar X_{\underline t})dW_t, \quad \bar X_0=x 
\end{equation}
where $\underline t = \frac {kT}{n}$ when  $t\!\in [t_k,t_{k+1})$. In particular $(\bar X_t)_{t\in [0,T]}$ is an ${\cal F}_t$-adapted It\^o  process  satisfying under the above assumptions made on $b$ and $\sigma$  (see~$e.g.$~\cite{BouLep}) :
\[
\forall\, p \in (0,+\infty),\qquad\Big\|\sup_{t\in[0,T]}|X_t|\Big\|_{p}+ \sup_{n\ge 1} \Big\|\sup_{t\in[0,T]}|\bar X^n_t|\Big\|_{p} \le C_{b,\sigma,p,T}\big(1+|x|\big)
\]
and
\[
\forall\, p\!\in (0,+\infty),\; \forall\, n\ge 1, \qquad\Big\|\sup_{t\in[0,T]}|X_t-\bar X^n_t|\Big\|_{p}  \le C_{b,\sigma,p,T} \sqrt{\Delta_n}\big(1+|x|\big)
\]
for a positive constant   $C_{b,\sigma,p,T}$. 
%On the other hand, there exists a real constant $ C_{b,\sigma,p,T}'>0$ such that, 
% for every $n \ge 1$ and for every $p \in (0,+\infty)$, 

As a consequence, general existence-uniqueness results for BSDEs entail (see~\cite{ParPen1}) the existence of  a unique solution  $(\bar Y,\bar Z)$ to the Markovian BSDE having  the genuine Euler scheme $\bar X$ instead of $X$ as a forward process.
%, namely
%\begin{equation}\label{eq:EulerBSDE}
%\bar Y_t = h(\bar X_T)  + \int_t^T f(s,\bar X_s,\bar Y_s,\bar Z_s) ds - \int_t^T \bar Z_s dW_s,\quad t \in [0,T],
%\end{equation}
%satisfying $\displaystyle  \E \Big(\sup_{t\in[0,T]}|\bar Y_t|^2 +\int_0^T |\bar Z_s|^2ds\Big)<+\infty$.
% 
Then, we can apply the classical comparison result (Proposition~2.1 from~\cite{ElkPenQue}) with 
$f^1(\omega, t, y,z)= f(t,\bar X_t(\omega), y,z)$ and $f^2(\omega, t, x,y,z)= f( t, X_t(\omega),y,z)$ which immediately yields the existence of  real constants $C^{(i)}_{b,\sigma, f,T}>0$, $i=1,2$,  such that 
\begin{eqnarray*}
%\[
  \E \Big[\sup_{t\in[0,T]}|Y_t-\bar Y_t|^2 \hskip -0.1cm+ \hskip -0.1cm \int_0^T |Z_t-\bar Z_t|^2dt\Big] &\le& C^{(1)} \Big[ \E \big(h(X_{_T})-h(\bar X_{_T})\big)^2 
 \hskip -0.1cm+ \hskip -0.1cm [f]^2_{\rm Lip}\E \int_0^T\hskip -0.25cm  |X_t-\bar X^n_t|^2 dt\Big]\\
% \\& \le &  C^{(1)}\big([h]_{\rm Lip}^2 +T  [f]_{\rm Lip}^2\big)\E\big(\sup_{t\in[0,T]}|X_t-\bar X^n_t|\big)^2
&\le&   C^{(2)}_{b,\sigma,f,T} \Delta_n.
%\]
\end{eqnarray*}

Unfortunately, at this stage,  the couple $(\bar Y_t,\bar Z_t)_{t\in [0,T]}$ is  still ``intractable"  for numerical purposes (it satisfies no Dynamic Programming Principle due to its continuous time nature and there is no possible exact simulation, etc). This is mainly due to $\bar Z$ about  which little is known. By contrast with $Z$ which is $e.g.$ closely connected to a PDE. So we will need to go deeper in the time discretization, by discretizing the $Z$ term itself. 
Consequently, we need to perform a second  time discretization on the Euler scheme based BSDE, 
%, this time on~\eqref{eq:EulerBSDE}, 
only involving discrete instants $t_k$, $k=0,\ldots,n$.

 We consider  an {\em explicit inner} scheme  recursively defined in a backward way as follows: 
\begin{eqnarray}
   \tilde Y_{t_n} &  = & h(\bar X_{t_n})  \label{EqDiscBSDE1}\\
   \tilde Y_{t_k}  & = & \mathds E( \tilde Y_{t_{k+1}} \vert \mathcal F_{t_k}) + \Delta_n f \big(t_k, \bar X_{t_k}, \mathds E( \tilde Y_{t_{k+1}}  \vert \mathcal F_{t_k}), \tilde  \zeta_{t_k} \big)    \label{EqDiscBSDE2}\\
\tilde \zeta_{t_k}& =&  \frac{1}{\Delta_n} \mathds E\big( \tilde Y_{t_{k+1}} (W_{t_{k+1}}-W_{t_k} )  \vert \mathcal F_{t_k}\big), \ k=0, \ldots,n-1 \label{eq:zeta}
  \end{eqnarray}
 It slightly differs from the other explicit schemes analyzed in the literature to our knowledge, since the conditioning is applied directly to $ \tilde Y_{t_{k+1}}$ {\em  inside the driver function} rather than outside. Note that in many situations, one uses the following  more symmetric alternative formula
\[
\tilde \zeta_{t_k} =  \frac{1}{\Delta_n} \mathds E\big( (\tilde Y_{t_{k+1}}-\tilde Y_{t_k}) (W_{t_{k+1}}-W_{t_k} )  \vert \mathcal F_{t_k}\big), 
\]
which is clearly quite  natural when thinking of a hedging term   as a derivative ($e.g.$ computed in a binomial tree). It   also has  virtues in terms of variance reduction (see~$e.g.$~\cite{BalPagPri1}). One easily shows  by a backward induction that, for every $k\!\in \{0,\ldots,n\}$,    $\tilde  Y_{t_k}\!\in L^2(\Omega,{\cal A}, \PP)$ since $\sup_{t\in[0,T]}|\bar X_{t}|\!\in L^2(\PP)$.

%\smallskip
%
%\noindent  {\em Comments.} The natural form of a time discretization scheme for~\eqref{EqBSDE} or~\eqref{eq:EulerBSDE} goes back to~\cite{BalPagPri0}. It was presented rather unexpectedly  in a reflected framework and only when the  driver $f$ does not depend on the hedging term $Z$: the motivation was American style option pricing. It has been developed and analyzed  in~\cite{BalPag1} for time discretization aspects. The approximating hedging term $\tilde \zeta_k$ above,   in~\eqref{eq:zeta}, also appears in~\cite{BalPagPri1} as the hedge  of an American option and in  \cite{Zha}, in a general numerical scheme for BSDE with possibly path-dependent terminal value. A general backward  time discretization scheme for  Markovian  BSDEs, including the $Z$ term in the driver,  can be found in~\cite{BouTou, MaZha, Zha}.  In fact many variants have been introduced, either explicit or implicit but also   depending on the location of the conditioning  inside or outside the driver. 

Our first aim is to adapt standard comparison theorems to compare the above purely  discrete  scheme $(\tilde Y_{t_k},\tilde Z_{t_k})$ with the original BSDE to derive  error bounds similar to those recalled above  between $(Y, Z)$ and $(\bar Y, \bar Z)$. To this end, like for the Euler scheme, we need  to extend $\tilde Y$ into a continuous time process by an appropriate interpolation.  We proceed as follows: let
\[
M_{_T} =\sum_{k=1}^n\tilde Y_{t_k} -\E\big(\tilde Y_{t_k}\,|\,{\cal F}_{t_{k-1}}\big).
\]
This random variable is in $L^2(\PP)$. Hence, by the  martingale  representation theorem, there exists an $({\cal F}_t)$-progressively measurable $\tilde Z\!\in L^2([0,T]\times \Omega, \PP\otimes dt)$ such that 
\[
M_{_T} = \int_0^T\tilde Z_t\,dW_t.
\]
Then $\displaystyle \tilde Y_{t_k} -\E \big(\tilde Y_{t_k}\,|\,{\cal F}_{t_{k-1}} \big)= \int_{t_{k-1}}^{t_k}\tilde Z_s\,dW_s$. In particular
\[
\tilde \zeta_{t_k} =\frac{1}{\Delta_n }  \E \big(\tilde Y_{t_{k+1}}(W_{t_{k+1}}-W_{t_{k}})\,|\, {\cal F}_{t_{k}}\big)= \frac{1}{\Delta_n }\E\Big(\int_{t_k}^{t_{k+1}}\tilde Z_s \,ds\,|\, {\cal F}_{t_k}\Big), \ k=0, \ldots,n-1,
\] 
so that we may  define a continuous extension of $(\tilde Y_{t_k})_{0\le k\le n}$ as follows:
\begin{equation} \label{EqYtildeZtildeIntro}
\tilde Y_t   =   \tilde Y_{t_{k}} - (t-t_{k})f\big(t_{k}, \bar X_{t_{k}},\mathds E( \tilde Y_{t_{k+1}}  \vert \mathcal F_{t_k}),\tilde \zeta_{t_{k}} \big)+ \int_{t_{k}}^{t} \tilde  Z_sdW_s, \quad t\in [t_{k},t_{k+1}].
\end{equation}

\section{Error analysis}  \label{SecErrorAnalBSDE}

\subsection{The time discretization error}
We provide in the theorem below the quadratic  error bound for the  inside explicit time discretization scheme $(\tilde Y,\tilde Z)$ defined by~\eqref{EqDiscBSDE1}-\eqref{eq:zeta} and~\eqref{EqYtildeZtildeIntro}.     
%The  result is post-poned to an  Appendix for self-completeness. 
Claim $(b)$ comes from~\cite{Zha}. The detailed proof  of Claim $(a)$ is given in the appendix.   Like for most  results of this type,  the  proof of $(a)$ follows the lines of that devised for comparison theorems in~\cite{ElkPenQue}. In particular, though slightly more technical at some places,  it is close to its counterpart for the standard outer explicit scheme originally established in~\cite{BalPag1}  (in $L^p$ for reflected BSDEs, but without $Z$ on the driver) or in~\cite{Zha} (in the quadratic case, see also~\cite{GobLab} for an extension  error bounds in $L^p$  or~\cite{BouTou} for implicit scheme).

%,GobLemWar}).

\begin{thm}\label{thm:bartilde}$(a)$ Assume the functions    $f:[0,T]\times \mathds R\times \mathds R^d \times \mathds R^q\to \mathds R$ is Lipschitz continuous in $(t,x,y,z)$ and that  $h:\mathbb R^d\to \mathbb R^d$ is  Lipschitz continuous.
%\begin{equation}  \label{EqAssumpTheErrorBoundCont}
%\forall t  \ge 0,  \quad  \vert f(t,x,y,z) \vert \le C(f) (1 + \vert x \vert + \vert y \vert + \vert z \vert).
%\end{equation}
 Then, there exists a real constant $C_{b,\sigma,f,T}>0$ such that, for every $n\ge 1$,  
\[
\max_{ k  =  0,\ldots, n }    \E |   Y_{t_k}-\tilde Y_{t_k} |^2 +\int_0^T   \E   |  Z_t-\tilde Z_t|^2dt \le  C_{b,\sigma,f,T}\left( \Delta_n+\int_0^T \E |Z_s- Z_{\underline s}|^2ds\right).
\]
where $\underline s = t_k$ if $s\!\in [t_k, t_{k+1})$.

\medskip
\noindent $(b)$  Assume that the functions $b, \sigma,  f$  are continuously differentiable in  their spatial variables ($x$ and $(x,y,z)$ respectively) with bounded partial derivatives  and $\frac 12$-H\"older continuous  with respect to $t$, that $d=q$ and $\sigma \sigma^{\star}$ is uniformly elliptic. Assume $h$ is  Lipschitz continuous.
Then, the process $(Z_t)_{t\in [0,T]}$ admits a c\`adl\`ag modification and 
%Assume that the functions $b, \sigma, h, f$  are   bounded  in $x$, uniformly Lipschitz  continuous in $(x,y,z)$ and H\"older continuous of parameter $1/2$ with respect %to $t$. Suppose furthermore that $h$ is of class $C_{b}^{2+\alpha}$, $\alpha \in (0,1)$ and that $\sigma \sigma^{\star}$ is uniformly elliptic. Then 
%BLAH-BLAH-BLAH  {\Large [cf. fin de la preuve en Annexe]} sur le fait que si tout va bien en terme de r\'egularit\'e
% \[
%Y_t= u(t,X_t) \quad \mbox{ and }\quad  Z_t= \nabla_xu(t,X_t)\sigma^*(t,X_t)
% \]
% so that, HOPEFULLY, 
\begin{equation}
\int_0^T \E |Z_s-Z_{\underline s}|^2ds\le C'_{b,\sigma,f,T} \Delta_n,
\end{equation}
so that there exists a real constant $\tilde C_{b,\sigma,f,T}>0$ such that, for every $n\ge 1$,  
\begin{equation}
\max_{ k = 0,\ldots, n }    \E |   Y_{t_k}-\tilde Y_{t_k} |^2 +\int_0^T   \E   |  Z_t-\tilde Z_t|^2dt   \le  \tilde C_{b,\sigma,f,T} \Delta_n .
 \end{equation}

\end{thm}
 
Claim~$(b)$ as stated comes from ~\cite{Zha} (Theorem~3.1 and Lemma~2.5$(i)$). It admits   several variants (see~\cite{BouTou}) or  improvements in the literature. Thus, 
 %in~\cite{BouTou}, the regularity of the driver $f$ can be slightly relaxed into a Lipschitz continuous regularity in $(x,y,z)$ provided $h$ is of class $C_{b}^{2+\gamma}$, $\gamma \in (0,1)$. By contrast, the above %control for $Z$ is obtained
  in~\cite{GobMak} (Theorem 20) it is obtained under a still lighter assumption on the terminal value $h$ of fractional type, provided this time $b$ and $\sigma$ are bounded, ${\cal C}^{2+\gamma}$, $\gamma\!\in (0,1]$, in space, uniformly $\frac 12$-H\" older in time (and $\sigma$ uniformly elliptic) and the driver $f$ is in ${\cal C}^1$ in its spatial variable with bounded partial derivatives and $f(., 0,0,0)\!\in L^1([0,T],dt)$.  Then, the resulting bound is $O\big(\Delta _n^{\gamma}\big)$ for a uniform mesh (but can attain $O\big(\Delta_n\big)$ for a tailored mesh). See also~\cite{DelMen} for a PDE based  extension to general forward-backward system $i.e.$ when the both forward and backward equations are  coupled.

\smallskip
\noindent {\sc Notations (change of).} The previous schemes~\eqref{EqDiscBSDE1}-\eqref{EqDiscBSDE2} involve some quantities and operators which will be the core of  what follows and are of discrete time nature. So, in order to simplify the proofs and alleviate  the notations, we will identify every time step $t_k^n$ by $k$ and we will denote $\mathds E_k = \mathds E(\,\cdot\, \vert \mathcal F_{t_k})$. Thus,   we will  switch to 
 \begin{eqnarray*}
\bar X_k:=\bar X_{t_k},& \tilde Y_k := \tilde Y_{t_k},& f_k(x,y,z) = f(t_k, x,y,z).
\end{eqnarray*}

%%%%%%%%%%%%%%%%%%%%%%%
%A GARDER POUR LA VERSION DETAILLEE
%However note that,  proof of the Appendix, we still use continuous time notations. 
%%%%%%%%%%%%%%%%%%%%%%%
%\textcolor{blue}{The specificity of our explicit inner scheme induces technical modifications in the proof of this theorem. However, it follows the standard architecture of their counterparts for explicit outer} schemes like those  originally proposed in~\cite{BalPag, BalPag1} (without  $Z$ in the driver but with a reflection), in~\cite{GobLab, GobLemWar} (among many others) or implicit schemes like~\cite{BouTou}.
% For a detailed proof we refer to the extended version of the paper (see~\cite{PagSag-ext}).

\subsection{Error bound for the quantization scheme}
In this section,  we consider  the quantization scheme ~\eqref{EqhatYkIntro1}-\eqref{EqhatYkIntro2} and  compute  the  quadratic  quantization error $(\mathds E \vert \tilde Y_{t_k}  -  \hat Y_{t_k} \vert^2)^{1/2}$  induced by the approximation of $\tilde Y_{t_k}$ by $\hat Y_{t_k}$, for every $k=0, \ldots,n$. 
This leads to the following theorem which  is the first main result of this paper.

   \begin{thm}  \label{TheoremPrincBsde} 
Assume that the drift $b$ and the diffusion coefficient    $\sigma$ of the diffusion $(X_t)_{t\in [0,T]}$ defined by~\eqref{eq:diffusion} are Lipschitz continuous, that the driver function $f$ satisfies $({\rm Lip}_f)$ (Assumption(\ref{AssLipschitzGenerator})) and that   the function $h$ is $[h]_{\rm Lip}$-Lipschitz  continuous. Assume that $n\ge n_0$ (in order to provide sharper constants depending on $n_0\ge 1$).

\smallskip 
\noindent $(a)$ For every $k =0, \ldots, n$,
 
\begin{equation} \label{Eq1TheoremPrincBsde}
   \big \Vert    \tilde Y_k - \hat Y_k \big \Vert _2^2     \leq     \sum_{i=k}^{n}  e^{(1+ [f]_{\rm Lip})(t_i  -t_k)}        K_i(b,\sigma,T,f)   \big \Vert\bar X_i -\hat X_i  \big \Vert_2^2, 
 \end{equation}
 where $K_n(b,\sigma,T,f) := [h]_{\rm Lip}^2$ and, for every $k=0,\ldots,n-1$, 
 %one can choose (provided $n\ge n_0$),
\begin{equation*}
K_k(b,\sigma,T,f) :=\kappa_1^2e^{2\kappa_0(T-t_k)}  +\big(1+\Delta_{n_0}\big)\big( C_{1,k}(b,\sigma,T,f) \Delta_{n_0} + C_{2,k}(b,\sigma,T,f) \big),
\end{equation*}
with
%\[
$\displaystyle \kappa_0=C_{b,\sigma,T}+ [f]_{\rm Lip}\Big(1+\frac{[f]_{\rm Lip}}{2}\Big)$, $\displaystyle \kappa_1= \displaystyle \frac{[f]_{\rm Lip}}{\kappa_0}+ [h]_{\rm Lip}$, \\
%\]
%\[
$\displaystyle C_{2,k}(b,\sigma,T,f) = q\kappa_1^2  [f]_{\rm Lip}^2 e^{2 \Delta_{n_0} C_{b,\sigma,T} + 2 \kappa_0(T-t_{k+1})} $
 %  \mbox{ and }\quad    
 and $\displaystyle C_{1,k}(b,\sigma,T,f)=  [f]_{\rm Lip}^2+ \frac{C_{2,k}(b,\sigma,T,f)}{q} $
 %\]
and 
\begin{equation}\label{eq:CbsigmaT}C_{b,\sigma,T}=[b]_{\rm Lip} +\frac 12\Big([\sigma]^2_{\rm Lip}+ \frac{T}{n_0} [b]_{\rm Lip}^2\Big).
\end{equation}
%Furthermore, if $n\ge n_0$, one can take $C_{b,\sigma,T}=\displaystyle [b]_{\rm Lip} +\frac 12\left([\sigma]^2_{\rm Lip}+\frac{ T}{n_0} [b]_{\rm Lip}^2\right)$.

\smallskip
\noindent $(b)$ For every $k =0, \ldots, n$,
 \[
 \Delta_n \sum_{k=0}^{n-1} \Vert\tilde \zeta_k-\hat \zeta_k\Vert_2^2\le \sum_{k=0}^{n-1} \frac{C_{2,k}(b,\sigma, T,f)}{[f]^2_{\rm Lip}} \Vert \bar X_k-\hat X_k\Vert_2^2+ \sum_{k=0}^{n-1}\Vert \tilde Y_{k+1}- \hat Y_{k+1}\Vert_2^2.
 \]
 \end{thm}
 
%Following the usual architecture of proofs of quantization based schemes like those developed in~\cite{BalPag} or~\cite{PagPha} (among others), 
The proof 
%of Theorem~\ref{TheoremPrincBsde} 
%(which will be displayed further on)  
is divided in two main steps: in the first one we establish the propagation of the Lipschitz property through the functions $\tilde y_k$ and $\tilde z_k$ involved  in the   Markov representation~\eqref{EqDiscBSDE1}-\eqref{EqDiscBSDE2} of $\tilde Y_k$ and $\tilde \zeta_k$, namely $\tilde Y_k= \tilde y_k(\bar X_k)$ and  $\tilde \zeta_k = \tilde z_k(\bar X_k)$, and  to control precisely the propagation of their  Lipschitz coefficients (an alternative to this phase can be  to consider the Lipschitz properties of the flow of the SDE like in~\cite{IllThese}). As a second step, we introduce the quantization based scheme which is the counterpart of~\eqref{EqDiscBSDE1} and~\eqref{EqDiscBSDE2} for which we  establish  a backward recursive inequality satisfied by $\| \tilde Y_k-\hat Y_k\|_2^2$.
%
%\medskip

\begin{rem} ({\em About the relationship between the temporal and the spatial partitions}) Owing to the non-asymptotic bound for the quantization (see Theorem \ref{ThemZadPierce} further), we deduce from  the upper bound of  Equation~\eqref{Eq1TheoremPrincBsde} that there exists some constants $c_i$, $i=1, \ldots,n$ (only depending on the coefficients $b$ and $\sigma$ of the diffusion $X$) such that for every $k =1, \ldots,n$,
\begin{equation}  \label{EqOptiDispatching}
 \big \Vert    \tilde Y_k - \hat Y_k \big \Vert _2^2     \leq     \sum_{i=k}^{n}       c_i  N_i^{-2/d}.
 \end{equation}
 So, a natural question   is to  determine  how to  dispatch optimally the sizes $N_1, \cdots,N_n$ (for a fixed mesh of  length  $n$, given that $X_0$ is deterministic and, as such, perfectly quantized with   $N_0=1$)  of the quantization grids  under the  total ``budget'' constraint $N_1+ \cdots + N_n\le N$  of elementary quantizers (with $N\ge n$ and $N_k \geq 1$, for every $k=1, \ldots,n$). This    amounts (at least at time $k=0$) to solving the constrained minimization problem 
 \[
 \min_{N_1+\cdots+N_n \le  N}    \sum_{i=1}^{n}  c_{i}  N_{i} ^{-2/d},
 \]
 whose solution reads  $\displaystyle N_{i} = \left\lfloor \frac{c_{i}^{ \frac{d}{d+2}}}{\sum_{k=1}^n c_{k}^{ \frac{d}{d+2}} } N \right \rfloor\vee 1$,    $i =1, \ldots, n$. Coming back to \eqref{EqOptiDispatching},  and using  the H\"{o}lder inequality (to get  the second inequality below) yields
\begin{equation}  \label{Upper1}
 \big \Vert    \tilde Y_0 - \hat Y_0 \big \Vert _2      \leq   N^{-1/d} \Big( \sum_{i=1}^n c_{i}^{ \frac{d}{d+2}}  \Big)^{1/2+1/d} \le  \Big(\frac{n}{N} \Big)^{1/d}\Big(\sum_{i=1}^n c_{i} \Big)^{\frac 12}\le\Big[\max_{i=1,\ldots,n} \!c_i^{\frac 12}\Big]\frac{n^{1/2+1/d} }{N^{\frac 1d}} .
\end{equation}
Notice that for the standard (``non-improved'') error bounds (see the introduction), the same optimal allocation procedure would yield (starting from  $ \big \Vert    \tilde Y_0 - \hat Y_0\big \Vert _2     \leq     \sum_{i=0}^{n}       c'_i  N_i^{-1/d}$), 
\[
\big \Vert    \tilde Y_0 - \hat Y_0 \big \Vert _2       \le  \Big(\frac{n}{N} \Big)^{1/d}\, \sum_{i=1}^n c'_{i}\le\Big[\max_{i=1,\ldots,n} \!c'_i\Big]\frac{n^{1+1/d} }{N^{\frac 1d}} 
\]
which emphasizes the improvement of the error bound as concerns the dependence in the time mesh size $n$.
%so that we switch  from a global error  (at $t=0$) of order $\big(\frac{n}{N} \big)^{1/d}\, \sum_{i=1}^n a_{i}$ (for the non-improved error)  to   $\Big(\frac{n}{N} \Big)^{1/d}\, \big(\sum_{i=1}^n a_{i}\big)^{1/2}$ (for the improved one).
\end{rem}

  \subsubsection{First step toward the proof of Theorem~\ref{TheoremPrincBsde}: Lipschitz operators}
  
  As a first step we  introduce  several operators which appear naturally when representing $Y_k$. We will show that these operators propagate Lipschitz continuity. It is a classical step when establishing {\em a priori error bounds} going back to~\cite{BalPag, BalPag1}, see also more recently~\cite{GobLemWar} (Proposition~3.4). However we do not skip it since it emphasizes the  technical specificities induced by our choice of an inner explicit scheme.
  
  To be more precise, we set for every $k\!\in \{0,\ldots,n-1\}$ and every Borel function  $g:\mathbb R^d \to \mathbb R$ with polynomial   growth
 \begin{eqnarray}
 {\cal E}_k(x,u) &= &x+\Delta_n b(t_k,x)+\sqrt{\Delta_n} \sigma(t_k,x) u,\; x\!\in \mathbb R^d, \; u\!\in \mathbb R^q\\
 P_{k+1}g(x) &=& \mathds E \,g\big(  {\cal E}_k(x,\varepsilon) \big) \quad \mbox{ where }\;\varepsilon \sim{\cal N}(0;I_q)\\
 Q_{k+1}g(x)&=&   \frac{1}{\sqrt{\Delta_n}} \mathds E \Big(g\big(  {\cal E}_k(x,\varepsilon) \big)\varepsilon\Big).
 \end{eqnarray}

 One immediately checks that for every $k \in \{0, \ldots,n-1  \}$,
 \[
 \mathds E_kg(\bar X_{k+1})= P_{k+1}g(\bar X_k)\quad \mbox{ and }\quad \mathds E_k\big(g(\bar X_{k+1})(W_{t^n_{k+1}}-W_{t^n_k})\big)=  \Delta_n Q_{k+1}g(\bar X_k).
 \]
 
 Note that the process  $(\bar X_k)_{0\le k\le n}$ is an $({\cal F}_k)_{0\le k\le n}$-Markov chain with transitions $P_{k}(x,dy)= \mathbb P(\bar X_k\in dy\,|\, \bar X_{k-1}=x)$, $k=1,\ldots,n$.   Moreover,  it  shares the property to propagate the Lipschitz property as established in the Lemma below.
 
 \begin{lem}  \label{LemProofProTheo} For every $k=0, \ldots,n-1$, the transition operator $P_{k+1}$ is  Lipschitz in the sense that its  Lipschitz coefficient defined by 
 $
 [P_{k+1}]_{\rm Lip}:=\displaystyle \sup_{f,\, [f]_{\rm Lip}\le 1}[P_{k+1}f]_{\rm Lip}
 $  
 is finite.  More precisely, it satisfies: 
\begin{equation}\label{eq:PLip}
 [P_{k+1}]_{\rm Lip}  \leq e^{\Delta_n C_{b,\sigma,T}}
 \end{equation}
where $C_{b,\sigma,T}$ is given by~\eqref{eq:CbsigmaT} (see also the comment that follows).
%\begin{equation}\label{eq:CbsigmaT}C_{b,\sigma,T}=[b]_{\rm Lip} +\frac 12([\sigma]^2_{\rm Lip}+ T [b]_{\rm Lip}^2).
%\end{equation}
%Furthermore, if $n\ge n_0$, one can take $C_{b,\sigma,T}=\displaystyle [b]_{\rm Lip} +\frac 12\left([\sigma]^2_{\rm Lip}+\frac{ T}{n_0} [b]_{\rm Lip}^2\right)$.
 \end{lem}

 \begin{proof}[{\bf Proof.}]  We have for every $x,x' \in \mathds R^d$, and for every Lipschitz continuous function  $g$
 % with Lipschitz coefficient $[g]_{\rm Lip}$,
  \begin{eqnarray*}
 \vert P_{k+1} g(x)  - P_{k+1}g(x') \vert^2 & \leq &   \mathds E \,\vert g\big(  {\cal E}_k(x,\varepsilon) \big)  - \mathds E \,g\big(  {\cal E}_k(x',\varepsilon) \big)  \vert ^2  \\ 
 &  \leq & [g]_{\rm Lip}^2  \,   \mathds E \vert {\cal E}_k(x,\varepsilon)  -  {\cal E}_k(x',\varepsilon)  \vert ^2
 \end{eqnarray*}
and elementary  computations, already carried out  in~\cite{BalPag}, show that 
 \begin{eqnarray*}   
 \mathds E \vert {\cal E}_k(x,\varepsilon)  -  {\cal E}_k(x',\varepsilon)  \vert ^2 &\le&  \big(1+ \Delta_n(2[b(t^n_k,.)]_{\rm Lip} +[\sigma(t^n_k,.)]^2_{\rm Lip})+ \Delta^2_n [b(t^n_k,.)]_{\rm Lip}^2  \big) \vert x-x' \vert^2 \\
 &\le& \big(1+ \Delta_n(2[b]_{\rm Lip} +[\sigma]^2_{\rm Lip})+ \Delta^2_n [b]_{\rm Lip}^2  \big) \vert x-x' \vert ^2\\
 &\le & (1+\Delta_n C_{b,\sigma,T})^2 \vert x-x' \vert^2\\
 &\le& e^{ 2 \Delta_n C_{b,\sigma,T}} \vert x-x' \vert^2
  \end{eqnarray*} 
  where $C_{b,\sigma,T}$ can be $e.g.$ taken equal to $[b]_{\rm Lip} +\frac 12([\sigma]^2_{\rm Lip}+ \frac{T}{n_0} [b]_{\rm Lip}^2)$ provided $n \ge n_0$. It follows that $P_{k+1}$ is Lipschitz with Lipschitz constant  $ [P_{k+1}]_{\rm Lip}  \leq e^{\Delta_n C_{b,\sigma,T}}.$  
%    
%    On the other hand,  for every $x,x' \in \mathds R^d$, and for every Lipschitz function  $g$ with Lipschitz coefficient $[g]_{\rm Lip}$, we have 
%     \begin{eqnarray*}
%  \vert Q_{k+1} g(x)  - Q_{k+1}g(x') \vert & \leq &\sqrt{\Delta_n} \vert  \mathds E \,g\big(  {\cal E}_k(x,\varepsilon)  \varepsilon \big)  - \mathds E \,g\big(  {\cal E}_k(x',\varepsilon) \varepsilon \big)  \vert   \\ 
%  &  \leq &\sqrt{\Delta_n}   [g]_{\rm Lip}  \,   \mathds E \vert {\cal E}_k((x,\varepsilon) \varepsilon)  -  {\cal E}_k((x',\varepsilon) \varepsilon) \vert  \\
%  & \leq &\sqrt{\Delta_n}   [g]_{\rm Lip}  \,   \mathds E \vert {\cal E}_k((x,\varepsilon) \varepsilon)  -  {\cal E}_k((x',\varepsilon) \varepsilon) \vert  \\
%  &  \leq & \sqrt{\Delta_n}  [g]_{\rm Lip}   \Vert \varepsilon \Vert_2 \,\Vert  {\cal E}_k(x,\varepsilon)  -  {\cal E}_k(x',\varepsilon)    \Vert_2,
%   \end{eqnarray*}
%so that owing to the previous result we deduce that $Q_{k+1}$ is Lipschitz with Lipschitz coefficient $[Q_{k+1}]_{\rm Lip}$ satisfying the announced inequality.
 \end{proof}

%To show the convergence of the quantized algorithm we need  the following result. 
 
  \begin{prop} \label{PropyKP1Lip} (see~\cite{BalPag})     $(a)$ The  functions $y_k$, $k=0,\ldots,n$, defined by the backward induction
  % \begin{eqnarray*}
 \[
  y_n =  h,\quad y_k =  P_{k+1}y_{k+1}+ \Delta_n f_k\big(\,.\,,P_{k+1}y_{k+1},Q_{k+1}y_{k+1}\big), \; k=0,\ldots,n-1,
\]
%  \end{eqnarray*}
  satisfies $\tilde Y_k=y_k(\bar X_k)$ for every $k\!\in \{0,\ldots,n\}$. Moreover, $\tilde \zeta_k= \frac{z_k (\bar X_k) }{ \sqrt{\Delta_n}}$  where,  for every $k\!\in \{0,\ldots,n-1\}$,  
   \[
z_k(x)  = \mathds E\big(y_{k+1}\big (\mathcal{ E}_k (x, \varepsilon) \big)\varepsilon \big),\; k=0,\ldots,n-1.
  \]
  
 % \smallskip
  \noindent $(b)$ Furthermore, assume that  the function $h$ is $[h]_{\rm Lip}$-Lipschitz continuous   and that the function $f(t, x,y,z)$ is  $[f]_{\rm Lip}$-Lipschitz continuous in $(x,y,z)$, uniformly in $t\!\in [0,T]$.   Then,  for every $k \!\in \{0, \ldots,n\}$,   the function $y_k$ is $[y_k]_{\rm Lip}$-Lipschitz continuous and there exists real constants $\kappa_0=C_{b,\sigma,T}+ [f]_{\rm Lip}(1+\frac 12[f]_{\rm Lip})$, and $\kappa_1= \displaystyle \frac{[f]_{\rm Lip}}{\kappa_0}+ [h]_{\rm Lip}$ (where $C_{b,\sigma,T}$ is given by~\eqref{eq:CbsigmaT}),
  % in Lemma~\ref{LemProofProTheo}, 
  %=\kappa_{b,\sigma,T,f}>0$ 
  such that $ [y_k]_{\rm Lip}  = [h]_{\rm Lip}$  and
   \begin{equation}     [y_k]_{\rm Lip}   \le \frac{\Delta_n}{e^{\kappa_0\Delta_n}-1}(e^{\kappa_0(T-t^n_k)}-1)[f]_{\rm Lip} + e^{\kappa_0(T-t^n_k)}[h]_{\rm Lip}e^{\kappa_0(T-t^n_k)}\kappa_1,\; k=0,\ldots,n-1.
  \label{eq:1PropyKP1Lip}
      \end{equation}
  In particular, $\displaystyle  \sup_{n\ge 1}  \max_{k=0,\ldots,n} [y_k]_{\rm Lip}   \le e^{\kappa_0 T }\kappa_1<+\infty$.   Moreover the functions $z_k$ are Lipschitz too and
   \begin{equation} \label{eq:Lipzk}  
  [z_k]_{\rm Lip} \le  \sqrt{q}\, e^{\Delta_nC_{b,\sigma,T}}\kappa_1e^{\kappa_0(T-t^n_{k+1})}, \; k=0,\ldots,n-1.
      \end{equation}
%    Setting $[P^{(n)}] = \max_{k=1, \ldots,n} [P_k]_{\rm Lip}$ and $[Q^{(n)}n] = \max_{k=1, \ldots,n} [Q_k]_{\rm Lip}$ 
%    $$
%    L_n = [P^{(n)}] (1+\Delta_n [f]_{\rm Lip})+\Delta_n [Q^{(n)}]_{\rm Lip}  [f]_{\rm Lip} -1 \in [-1,+\infty)
%    $$ 
%    we have   
%   \[
%[y_k]_{\rm Lip}  \leq     \left \{ \begin{array}{ll}
%\Delta_n \frac{(1+L_n)^{n-k}-1}{L_n}[f]_{\rm Lip} + (1+L_n)^{n-k}    [h]_{\rm Lip}  &  \textrm{ if }  L_n > 0 \\\\
% \frac{ \Delta_n}{|L_n|}[f]_{\rm Lip} + (1+L_n)^{n-k}    [h]_{\rm Lip}  &  \textrm{ if }  L_n < 0\\\\
%\ [h]_{\rm Lip} + (n-k)[f]_{\rm Lip} &  \textrm{ if }   L_n=0.
% \end{array}  
% \right.
% \]
\end{prop}

 \noindent{\em Proof}. $(a)$ We proceed by a    backward induction using~(\ref{EqDiscBSDE1}) and~(\ref{EqDiscBSDE2}), relying on  the fact that $(\bar X_k)_{k=0,\ldots,n}$ is a Markov chain which propagates Lipschitz continuity.  In fact, $\tilde Y_n = h(\bar X_n) := y_n(\bar X_n)$. Assuming that $\tilde Y_{k+1}  = y_{k+1}(\bar X_{k+1})$ and using  Equation ~(\ref{EqDiscBSDE2}) and the Markov property, we  get 
 \begin{eqnarray*}
  \tilde  Y_{k}  & = & \mathds E( y_{k+1}(\bar X_{k+1}) \vert  \bar X_k) +  \Delta_n f_k \big(\bar X_{k}, \mathds E(  y_{k+1}(\bar X_{k+1})  \vert  \bar X_k), \zeta_{t_k^n} \big)  \\
   & = & P_{k+1} y_{k+1} (\bar X_k) + \Delta_n f_k(\bar X_k, P_{k+1}y_{k+1} (\bar X_k), Q_{k+1}y_{k+1}(\bar X_k))=y_{k}(\bar X_k).
  \end{eqnarray*}
One shows likewise that $\zeta_k= Q_{k+1}(y_{k+1})(\bar X_k) = \frac{z_k(\bar X_k)}{\sqrt{\Delta_n}}$, $k=0,\ldots,n-1$.
 
\noindent $(b)$ We also show this claim by a backward induction. In fact, $\tilde Y_n = h( \bar  X_n):=y_n( \bar X_n)$ and $h$ is $[h]_{\rm Lip}$-Lipschitz.   Suppose that $y_{k+1}$ is $[y_{k+1}]_{\rm Lip}$-Lipschitz continuous. Then,  for every $x, x' \in \mathds R^d$, we can write 
  \begin{eqnarray*}
y_k(x) - y_k(x') &=&  \mathds E \big(     y_{k+1}\big({\cal E}_k(x,\varepsilon)\big)-y_{k+1}\big({\cal E}_k(x',\varepsilon)\big)\big)\\
 &&+\Delta_n \Big[A_{x,x'}(x-x')+B_{x,x'}\mathds E \big(     y_{k+1}\big({\cal E}_k(x,\varepsilon)\big)-y_{k+1}\big({\cal E}_k(x',\varepsilon)\big)\big)\\
 &&\left. \qquad \qquad + C_{x,x'}\mathds E \Big(\big(y_{k+1}\big({\cal E}_k(x,\varepsilon)\big)-y_{k+1}\big({\cal E}_k(x',\varepsilon)\big)\big)\frac{\varepsilon}{\sqrt{\Delta_n}}\Big)\right]
    \end{eqnarray*} 
    where   $\varepsilon  \sim \mathcal N(0,I_q)$ and 
   \begin{eqnarray*}   
   A_{x,x'}&=&     \frac{f_k\big(x, P_{k+1}y_{k+1}(x), Q_{k+1}y_{k+1}(x)\big)-f_k\big(x', P_{k+1}y_{k+1}(x), Q_{k+1}y_{k+1}(x)\big)}{x-x'} \mbox{\bf 1}_{\{x\neq x'\}},\\
        B_{x,x'}&=&   \frac{f_k\big(x', P_{k+1}y_{k+1}(x), Q_{k+1}y_{k+1}(x)\big)-f_k\big(x', P_{k+1}y_{k+1}(x'), Q_{k+1}y_{k+1}(x)\big)}{P_{k+1}y_{k+1}(x)- P_{k+1}y_{k+1}(x')} \mbox{\bf 1}_{P_{x,x'}},\\
             C_{x,x'}&=& \frac{f_k\big(x', P_{k+1}y_{k+1}(x'), Q_{k+1}y_{k+1}(x)\big)-f_k\big(x', P_{k+1}y_{k+1}(x'), Q_{k+1}y_{k+1}(x')\big)}{Q_{k+1}y_{k+1}(x)-Q_{k+1}y_{k+1}(x')} \mbox{\bf 1}_{Q_{x,x'}},
      \end{eqnarray*} 
with $P_{x,x'} = \{P_{k+1}y_{k+1}(x)\neq P_{k+1}y_{k+1}(x')\}$ and $Q_{x,x'} = \{Q_{k+1}y_{k+1}(x)\neq Q_{k+1}y_{k+1}(x')\}$.
      The function $f_k$ being Lipschitz continuous, one clearly has  $|A_{x,x'}|$, $|B_{x,x'}|$, $|C_{x,x'}|\le [f]_{\rm Lip}$. Now, taking advantage of the linearity of expectation, we get
\[
y_k(x) - y_k(x')    = \mathds E\left[\big(y_{k+1}\big({\cal E}(x,\varepsilon)\big)-y_{k+1}\big({\cal E}(x',\varepsilon)\big)\Big(1+\Delta_n \Big(B_{x,x'}+C_{x,x'}\frac{\varepsilon}{\sqrt{\Delta_n}} \Big)\Big)\right]   + A_{x,x'}(x-x').
\]
Then Schwarz's Inequality yields
 \[
 \vert y_k(x) - y_k(x') \vert   \le \big\|y_{k+1}\big({\cal E}(x,\varepsilon)\big)-y_{k+1}\big({\cal E}(x',\varepsilon)\big)\big\|_2 \Big\|1+\Delta_n \Big(B_{x,x'}+C_{x,x'}\frac{\varepsilon}{\sqrt{\Delta_n}}\Big)\Big\|_2+\Delta_n  [f]_{\rm Lip}\,|x-x'|.
 \]
 Now, 
 \begin{equation*}
 \big\| y_{k+1}\big({\cal E}_k(x,\varepsilon)\big)-y_{k+1}\big({\cal E}_k(x',\varepsilon)\big)\big\|_2 \le  [y_{k+1}]_{\rm Lip} \big\|{\cal E}_k(x,\varepsilon) -{\cal E}_k(x',\varepsilon)\big\|_2  \le [y_{k+1}]_{\rm Lip} e^{\Delta_n C_{b,\sigma,T}} \vert  x - x'\vert
 \end{equation*}
by Lemma~\ref{LemProofProTheo}. On the other hand, using that   $|B_{x,x'}|$, $|C_{x,x'}|\le [f]_{\rm Lip}$ and $\mathds E (\varepsilon)=0$, 
  \begin{eqnarray*}   
 \Big\|1+\Delta_n \Big(B_{x,x'}+C_{x,x'} \frac{\varepsilon}{\sqrt{\Delta_n}} \Big)\Big\|_2^2&= &(1+\Delta_nB_{x,x'})^2+\Delta_nC_{x,x'}^2\\
 &\le& 1+\Delta_n\big(2[f]_{\rm Lip}+[f]^2_{\rm Lip}\big)+\Delta_n^2[f]_{\rm Lip}^2\\
 &\le & e^{2\Delta_n [f]_{\rm Lip}(1+\frac 12[f]_{\rm Lip})} .
  \end{eqnarray*} 
 Finally, owing to the definition of $\kappa_0$, we get
 \[
 \big|y_k(x) - y_k(x')  \big| \le \big( e^{\Delta_n \kappa_{0}}[y_{k+1}]_{\rm Lip}+\Delta_n[f]_{\rm Lip} \big)|x-x'|
 \]
 $i.e.$ $y_{k}$ is Lipschitz continuous with Lipschitz coefficient $[y_k]_{\rm Lip}$ satisfying
 \[
 [y_{k}]_{\rm Lip}\le e^{ \kappa_{0}\Delta_n}[y_{k+1}]_{\rm Lip}+\Delta_n[f]_{\rm Lip}.
 \]
 %  [P_{k+1}]_{\rm Lip}= \Big(1+ \Delta_n(2[b(t^n_k,.)]_{\rm Lip} +[\sigma(t^n_k,.)]^2_{\rm Lip})+ \Delta^2_n [b(t^n_k,.)]_{\rm Lip}^2  \Big)^{\frac 12}.
%    \\
%     &  &  +  \Delta_n \vert  f_{k} (x, P_{k+1}y_{k+1}(x),  Q_{k+1}y_{k+1}(x)) -  f_{k} (x', P_{k+1}y_{k+1}(x'), Q_{k+1}y_{k+1}(x'))  \vert \\
%   & \leq & [P_{k+1}]_{\rm Lip} [y_{k+1}]_{\rm Lip} \vert  x - x' \vert  + \Delta_n[f]_{\rm Lip}\vert x-x' \vert\Big(1+ [P_{k+1}]_{\rm Lip} [y_{k+1}]_{\rm Lip} +   [Q_{k+1}]_{\rm Lip} [y_{k+1}]_{\rm Lip}   \Big)\\
%   & = & \Big( [P_{k+1}]_{\rm Lip} (1 +\Delta_n [f]_{\rm Lip}) +\Delta_n [Q_{k+1}]_{\rm Lip}  [f]_{\rm Lip}\Big)    [y_{k+1}]_{\rm Lip}          + \Delta_n[f]_{\rm Lip}  \Big) \vert x -x' \vert.
%   \end{eqnarray*} 
%   \begin{eqnarray*}
%   \vert y_k(x) - y_k(x') \vert &  \leq &  \vert P_{k+1} y_{k+1}(x) - P_{k+1} y_{k+1}(x') \vert \\
%     &  &  +  \Delta_n \vert  f_{k} (x, P_{k+1}y_{k+1}(x),  Q_{k+1}y_{k+1}(x)) -  f_{k} (x', P_{k+1}y_{k+1}(x'), Q_{k+1}y_{k+1}(x'))  \vert \\
%   & \leq & [P_{k+1}]_{\rm Lip} [y_{k+1}]_{\rm Lip} \vert  x - x' \vert  + \Delta_n[f]_{\rm Lip}\vert x-x' \vert\Big(1+ [P_{k+1}]_{\rm Lip} [y_{k+1}]_{\rm Lip} +   [Q_{k+1}]_{\rm Lip} [y_{k+1}]_{\rm Lip}   \Big)\\
%   & = & \Big( [P_{k+1}]_{\rm Lip} (1 +\Delta_n [f]_{\rm Lip}) +\Delta_n [Q_{k+1}]_{\rm Lip}  [f]_{\rm Lip}\Big)    [y_{k+1}]_{\rm Lip}          + \Delta_n[f]_{\rm Lip}  \Big) \vert x -x' \vert.
%   \end{eqnarray*} 
%   Then $y_{k}$ is Lipschitz with Lipschitz coefficient $[y_k]_{\rm Lip}$ satisfying
%   $$ [y_k]_{\rm Lip}  \leq (1+L_n) [y_{k+1}]_{\rm Lip} + \Delta_n [f]_{\rm Lip}.   $$ 
The conclusion follows by induction. As for the functions $z_k$, we get
%Let us determine $[z_k]_{\rm Lip}$. We have, 
for every $k=0,\ldots,n-1$,
\[   
z_k(x) - z_k(x')   = \mathds E \Big( \Big(y_{k+1}\big(\mathcal E_k(x,\varepsilon) \big)- y_{k+1}\big(\mathcal E_k(x',\varepsilon) \big)\Big) \varepsilon  \Big).
\]
Hence, using  that $\varepsilon  \sim \mathcal {N}(0;I_q)$ combined with  Schwartz's Inequality, we get  
 \begin{eqnarray*}  \big| z_k(x) - z_k(x') \big|& \le & [y_{k+1}]_{\rm Lip} \mathds E \Big|\big( (x-x') + \Delta_n (b(x) - b(x')) + \sqrt{\Delta_n}   (\sigma(x)  - \sigma(x')) \varepsilon  \big) \varepsilon    \Big|\\
% &  \le  & [y_{k+1}]_{\rm Lip}\mathds E \big| (x-x') + \Delta_n (b(x) - b(x')) + \sqrt{\Delta_n}   (\sigma(x)  - \sigma(x')) \varepsilon  \big| \big| \varepsilon  \big| \\
%   \sqrt{\Delta_n} \mathds E\big(( \sigma(x)  - \sigma(x')) Z^2 \big) \\
 &\le &  [y_{k+1}]_{\rm Lip}\big\|  (x-x') + \Delta_n (b(x) - b(x')) + \sqrt{\Delta_n}   (\sigma(x)  - \sigma(x')) \varepsilon  \big\|_2 \big\|  \varepsilon \big\|_2\\
 %& = &  \sqrt{q} \, [y_{k+1}]_{\rm Lip}\big((1+ \Delta_n[b]_{\rm Lip})^2 +\Delta_n[\sigma]^2_{\rm Lip})^{\frac 12}|x-x'|
 &\le&[y_{k+1}]_{\rm Lip}e^{\Delta_nC_{b,\sigma,T}} \sqrt{q}  |x-x'|\\
 &\le& \sqrt{q}\, e^{\Delta_nC_{b,\sigma,T}}\kappa_1e^{\kappa_0(T-t^n_{k+1})}  |x-x'|. \hskip 6cm \Box
 \end{eqnarray*}
% \end{proof}
 
%   \begin{rem}
%  Notice  that, from a numerical  viewpoint,  computing  $\hat{\mathds E}_k(\hat Y_{k+1}   {\varepsilon}_{k+1})$ amounts to have access to a new companion parameters
%   \begin{eqnarray*}
%    \hat{\mathds E}_k(\hat Y_{k+1}   {\varepsilon}_{k+1}) &=&  \hat{\mathds E}_k \,\hat y_{k+1}(\hat X_{k+1}  \varepsilon_{k+1}) \\
%    &=& \sum_{i=1}^{N_k} \mbox{\bf 1}_{\{\hat X_k =  x^k_i \}}  \sum_{j=1}^{N_{k+1}} \tilde \pi_{ij}^{k,k+1}\hat y_{k+1}(x^{k+1}_j)
%   \end{eqnarray*}
%   where 
%   \[
%   \tilde\pi_{ij}^{k,k+1}=  \mathds E \Big( \varepsilon_{k+1}\mbox{\bf 1}_{\{X_{k+1}= x^{k+1}_j\}} \,|\, X_k=x^k _i\Big).
%   \]
%   \end{rem}

 \subsubsection{Second step of the proof of Theorem~\ref{TheoremPrincBsde}}
 Let $(\hat X_k)_{k=0, \ldots,n}$ be the quantization of the Markov chain $\bar X$, where every quantizer $\hat X_k$ is of size $N_k$, for every $k \in \{0, \ldots, n\}$.  Recall that  the discrete time quantized BSDE process $(\hat Y_k)_{k =0, \ldots,n}$ is  defined by the following recursive algorithm:
  \begin{eqnarray*}
  \hat Y_{n} &  = & h(\hat X_{n})\\
  \hat Y_{k}  & = & \hat {\mathds E}_k ( \hat Y_{k+1})  +  \Delta_n   f_k \big(\hat X_{k}, \hat{\mathds E}_k(\hat Y_{k+1}), \hat {\zeta}_{k} \big)   \\
\mbox{with }\hskip 3,0 cm       \hat {\zeta}_{k} & = & \frac{1}{\Delta_n} \hat{\mathds E}_k(\hat Y_{k+1}   \Delta W_{t_{k+1}}), k=0, \ldots,n-1,
  % \frac{1}{\sqrt{\Delta_n}} \hat{\mathds E}_k(\hat Y_{k+1}   {\varepsilon}^{(n)}_{k+1})\hskip 3cm  
  \end{eqnarray*}
where 
%$\Delta W_{t_{k+1}}= W_{t_{k+1}}-W_{t_k}$ and  
$\hat{\mathds E}_k = \mathds E (\cdot \,\vert\, \hat X_k)$. Owing to the previous section, we are now in position to prove Theorem~\ref{TheoremPrincBsde}. 
   
 %In the following result, we show  (under more general assumptions than~\cite{BalPag, BalPagPri0})   that   $\Vert  \tilde Y_k - \hat Y_k \Vert_2^2$ is bounded by the sum of the  quantization errors $\Vert \bar  X_i  - \hat X_i  \Vert_2^2$, for $i=k, \ldots,n$. 
  
\bigskip
\noindent {\bf Proof of Theorem~\ref{TheoremPrincBsde}.}  $(a)$ Using the fact that, for every $k \in \{0, \ldots,n  \}$,  $\sigma(\hat X_k )  \subset \sigma(\bar X_k)$, we have
\begin{equation}  \label{EqFirstDecomp}
\tilde Y_k - \hat Y_k   = \tilde Y_k -  \mathds{\hat E}_k (\tilde Y_k)  + \mathds{\hat E}_k(\tilde Y_k -  \hat Y_k)
\end{equation}
where $\tilde Y_k -  \mathds{\hat E}_k (\tilde Y_k)$ and $\mathds{\hat E}_k(\tilde Y_k -  \hat Y_k)$ are square integrable and orthogonal in $L^2(\sigma(\bar X_k))$.  As a consequence,  using the Pythagoras theorem for  conditional expectation yields
\[
\Vert \tilde Y_k - \hat Y_k  \Vert_2^2 = \Vert \tilde Y_k -  \mathds{\hat E}_k (\tilde Y_k) \Vert_2^2  + \Vert \mathds{\hat E}_k(\tilde Y_k -  \hat Y_k) \Vert_2^2.
\]

On the other hand,   it follows from the definition of the conditional expectation $\hat{\mathds E}_k(\cdot)$ as the best approximation  in $L^2$ among square integrable  $\sigma(\hat X_k)$-measurable random vectors    that 
$$  
\Vert   \tilde Y_k -  \mathds{\hat E}_k (\tilde Y_k)  \Vert_2^2  = \Vert   y_k(\bar X_k) -  \mathds{\hat E}_k (y_k(\bar X_k))  \Vert_2^2  \leq  \Vert   y_k(\bar X_k) -  y_k(\hat X_k)  \Vert_2^2 \leq [y_k]_{\rm Lip}^2 \Vert  \bar X_k - \hat X_k \Vert_2^2. 
$$
 Let us consider now the last term of  the equality~(\ref{EqFirstDecomp}).   We have,
   \begin{eqnarray*}
 \mathds{\hat E}_k(\tilde Y_k -  \hat Y_k)  & = & \mathds{\hat E}_k \big[ \tilde Y_{k+1} -\hat Y_{k+1} + \Delta_n \big( f_k(\bar X_k, \mathds E_k(\tilde Y_{k+1}), \tilde \zeta_k)  - f_k(\hat X_k, \hat{\mathds E}_k(\hat Y_{k+1}),\hat \zeta_k)   \big) \big] \\
 &  =  &  \mathds{\hat E}_k \big[ \tilde Y_{k+1} -\hat Y_{k+1}  + \Delta_n \big( f_k(\bar X_k, \mathds E_k( \tilde Y_{k+1}), \tilde \zeta_k)  - f_k(\hat X_k, \hat{\mathds E}_k(\tilde Y_{k+1}), \hat{\mathds E}_k( \tilde \zeta_k))   \big) \\
 &  & + \,   \Delta_n \big( f_k(\hat X_k, \hat{\mathds E}_k(\tilde Y_{k+1}), \hat{\mathds E}_k(  \tilde \zeta_k))  - f_k(\hat X_k, \hat{\mathds E}_k(\hat Y_{k+1}),\hat \zeta_k)   \big)   \big]  \\
 & = &  \mathds{\hat E}_k \big[ \tilde Y_{k+1} -\hat Y_{k+1} + \Delta_n \hat B_k  \hat{\mathds E}_k(\tilde Y_{k+1} -\hat Y_{k+1}) + \Delta_n \hat C_k \hat{\mathds E}_k( \tilde \zeta_k  -\hat \zeta_k) \big]  \\
 & & + \,  \Delta_n\hat{\mathds E}_k \big( f_k(\bar X_k, \mathds E_k(\tilde Y_{k+1}),  \tilde \zeta_k)  - f_k(\hat X_k, \hat{\mathds E}_k(\tilde Y_{k+1}), \hat{\mathds E}_k( \tilde \zeta_k))   \big)
 \end{eqnarray*}
\begin{eqnarray*}
   &\hskip -0.5cm  \textrm{where } \quad& \hat B_k   :=  \frac{f_k(\hat X_k, \hat{\mathds E}_k(\tilde Y_{k+1}),\hat{\mathds E}_k( \tilde \zeta_k))  - f_k(\hat X_k, \hat{\mathds E}_k(\hat Y_{k+1}), \hat{\mathds E}_k( 
\tilde \zeta_k))  }{\hat{\mathds E}_k(\tilde Y_{k+1} )-  \hat{\mathds E}_k( \hat Y_{k+1})} \mbox{\bf 1}_{ \{\hat{\mathds E}_k(\tilde Y_{k+1} ) \not=  \hat{\mathds E}_k(\hat Y_{k+1} ) \}} \\
    &\hskip -0.5cm  \textrm{and } \qquad  &  \hat{C}_k :=  \frac{f_k(\hat X_k, \hat{\mathds E}_k(\hat Y_{k+1}), \hat{\mathds E}_k( \tilde \zeta_k))  - f_k(\hat X_k, \hat{\mathds E}_k(\hat Y_{k+1}), \hat {\zeta}_k))}{\hat{\mathds E}_k( \tilde \zeta_{k} )-  \hat{\mathds E}_k(\hat{\zeta}_{k} )} \mbox{\bf 1}_{ \{\hat{\mathds E}_k( \tilde \zeta_{k} ) \not=  \hat{\mathds E}_k(\hat{\zeta}_{k} ) \}}.
  \end{eqnarray*}
  As 
  $$  \hat{\mathds E}_k( \tilde \zeta_{k} )-  \hat{\mathds E}_k(\hat{\zeta}_{k} )  = \frac{1}{ \Delta_n } \hat{\mathds E}_k ((\tilde Y_{k+1} -\hat Y_{k+1}) \Delta W_{t_{k+1}}),
  $$
we deduce that 
 \begin{eqnarray}  \label{EqTermsToControl}
 \mathds{\hat E}_k(\tilde Y_k -  \hat Y_k)  & = &  \mathds{\hat E}_k \big[ \big(\tilde Y_{k+1} -\hat Y_{k+1} \big) \big(1 +\Delta_n \hat B_k + \hat C_k  \Delta W_{t_{k+1}} \big) \big]  \nonumber \\
 & &+ \Delta_n \big( f_k(\bar X_k, \mathds E_k(\tilde Y_{k+1}),  \tilde \zeta_k)  - f_k(\hat X_k, \hat{\mathds E}_k(\tilde Y_{k+1}), \hat{\mathds E}_k( \tilde \zeta_k))   \big).
 \end{eqnarray}
So, it remains to control each term of the above equality. Considering its last term, it follows from the Lipschitz assumption on the driver  $f_k$   that 
 \begin{eqnarray*}
\Vert  f_k(\bar X_k, \mathds E_k(\tilde Y_{k+1}), \tilde \zeta_k)   -  f_k(\hat X_k, \hat{\mathds E}_k(\tilde Y_{k+1}), \hat{\mathds E}_k( \tilde \zeta_k)) \Vert_2^2  & \leq &   [f]_{\rm Lip}^2 \big( \Vert \bar X_k -\hat X_k \Vert_2^2 \\
&& +  \Vert  \mathds E_k(\tilde Y_{k+1})  - \hat{\mathds E}_k (\mathds E_k (\tilde Y_{k+1})) \Vert_2^2  \\
& & + \Vert  \tilde \zeta_k  - \hat{\mathds E}_k( \tilde  \zeta_k) \Vert_2^2\big).
  \end{eqnarray*}
First, from the very definition of conditional expectation operator $\hat{\mathds E}_k$ as the best  quadratic approximation by a Borel function of $\hat X_k$ (or, equivalently, the orthogonal projection on $L^2(\sigma(\hat X_k), \PP)$), we derive that 
%the contraction  property of  the conditional expectation  yields 
    \begin{eqnarray*}
   \Vert  \mathds E_k(\tilde Y_{k+1})  - \hat{\mathds E}_k (\mathds E_k (\tilde Y_{k+1})) \Vert_2^2  & \leq & \Vert  P_{k+1} y_{k+1}(\bar X_k) - P_{k+1} y_{k+1}(\hat X_k)  \Vert_2^2  \\
   &  \leq & [P_{k+1}]_{\rm Lip}^2 [y_{k+1}]_{\rm Lip}^2 \Vert  \bar X_k -  \hat X_k \Vert_2^2.  
       \end{eqnarray*}
  On the other hand, starting from  $\tilde \zeta_k = \frac{1}{\Delta_n} \mathds E_k(\tilde Y_{k+1}  \Delta W_{t_{k+1}}) =  \frac{z_k(\bar X_k)}{\sqrt{\Delta_n}}$, $k=0,\ldots,n-1$ (see Proposition~\ref{PropyKP1Lip}$(a)$), we get, using again the above characterization of the  conditional expectation operator~$\hat{\mathds E}_k$, 
%  set
%  $$
%  z_k(x)  = \mathds E\big(y_{k+1}\big (\mathcal{ E}_k (x, \varepsilon) \big)\varepsilon \big),\; k=0,\ldots,n-1,
%  $$ 
%  where $\varepsilon  \sim \mathcal {N}(0;I_q)$. We get
   \begin{eqnarray}
 \nonumber \Vert  \tilde  \zeta_k  -  \hat{ \mathds E}_k  \tilde \zeta_k \Vert_2^2  & =  &
 % \frac{1}{\Delta_n^2}  \Vert  \mathds E_k(\tilde Y_{k+1}  \Delta W_{t_{k+1}}) - \hat{\mathds E}_k \big(\mathds E_k(\tilde Y_{k+1}  \Delta W_{t_{k+1}}) \big) \Vert_2^2  %\\ \nonumber    &  = & 
 \frac{1}{\Delta_n} \Vert  z_k(\bar X_k) -  \hat{ \mathds E}_k ( z_k( \bar X_k)) \Vert_2^2 \\
    \label{eq:zeta}     &  \leq &  \frac{1}{\Delta_n} \Vert  z_k(\bar X_k) -  z_k(\hat X_k) \Vert_2^2 \leq \frac{1 }{ \Delta_n}  [z_k]_{\rm Lip}^2 \Vert  \bar X_k - \hat X_k  \Vert_2^2.  \label{eq:1PropyKP1LipZ}
      \end{eqnarray}
%Let us determine $[z_k]_{\rm Lip}$. We have, for every $k=0,\ldots,n-1$,
%\[   
%z_k(x) - z_k(x')   = \mathds E \Big( \Big(y_{k+1}\big(\mathcal E_k(x,\varepsilon) \big)- y_{k+1}\big(\mathcal E_k(x',\varepsilon) \big)\Big) \varepsilon  \Big)  \\
%\]
%so that 
% \begin{eqnarray*}  \big| z_k(x) - z_k(x') \big|& \le & [y_{k+1}]_{\rm Lip} \mathds E \Big|\big( (x-x') + \Delta_n (b(x) - b(x')) + \sqrt{\Delta_n}   (\sigma(x)  - \sigma(x')) \varepsilon  \big) \varepsilon   \big] \Big|\\
% &  \le  & [y_{k+1}]_{\rm Lip}\mathds E \big| (x-x') + \Delta_n (b(x) - b(x')) + \sqrt{\Delta_n}   (\sigma(x)  - \sigma(x')) \varepsilon  \big| \big| \varepsilon  \big| \\
%%   \sqrt{\Delta_n} \mathds E\big(( \sigma(x)  - \sigma(x')) Z^2 \big) \\
% &\le &  [y_{k+1}]_{\rm Lip}\big\|  (x-x') + \Delta_n (b(x) - b(x')) + \sqrt{\Delta_n}   (\sigma(x)  - \sigma(x')) \varepsilon  \big\|_2 \big\|  \varepsilon \big\|_2\\
% %& = &  \sqrt{q} \, [y_{k+1}]_{\rm Lip}\big((1+ \Delta_n[b]_{\rm Lip})^2 +\Delta_n[\sigma]^2_{\rm Lip})^{\frac 12}|x-x'|
% &\le&[y_{k+1}]_{\rm Lip}e^{\Delta_nC_{b,\sigma,T}} \sqrt{q}  |x-x'|\\
% &\le& \sqrt{q}\, e^{\Delta_nC_{b,\sigma,T}}\kappa_1e^{\kappa_0(T-t^n_{k+1})}  |x-x'|.
% \end{eqnarray*}
% so that 
% $$ 
% [z_k]_{\rm Lip}  \le  C_{0,q,b,\sigma}[y_{k+1}]_{\rm Lip}.
% $$  
% 
% \textcolor{red}{Il faudra retravailler les cstes}
Finally, using the upper-bound for $[z_k]_{\rm Lip}$ established in  Proposition~\ref{PropyKP1Lip}$(b)$, we  deduce that
\setlength\arraycolsep{1pt}
 \begin{eqnarray} 
\nonumber \Vert  f_k(\bar X_k, \mathds E_k(\tilde Y_{k+1}),  \tilde \zeta_k) &   - & f_k(\hat X_k, \hat{\mathds E}_k(\tilde Y_{k+1}), \hat{\mathds E}_k( \tilde \zeta_k)) \Vert_2\\
  &   \leq  &  \Big(C_{1,k}(b,\sigma,T,f)   + \frac{C_{2,k}(b,\sigma,T,f)}{\Delta_n}\Big)^{\frac 12} \Vert  \bar X_k  - \hat X_k \Vert_2  \label{ControlFirstTerm}
  \end{eqnarray}
   \setlength\arraycolsep{3pt}
since, owing to~\eqref{eq:1PropyKP1Lip} and~\eqref{eq:Lipzk}, we have   
\begin{equation*} 
 [f]_{\rm Lip}^2 \big(1 + [P_{k+1}]_{\rm Lip}^2[y_{k+1}]_{\rm Lip}^2   \big)\le  C_{1,k}(b,\sigma,T,f)  \  \mbox{ and } \   [f]_{\rm Lip}^2 [z_k]^2_{\rm Lip}\le C_{2,k}(b,\sigma,T,f),
 \end{equation*}
$k=0,\ldots,n-1$, where
\begin{equation}\label{eq:C1C2}
 C_{2,k}(b,\sigma,T,f) = q\kappa_1^2  [f]_{\rm Lip}^2 e^{2 \Delta_n C_{b,\sigma,T} + 2 \kappa_0(T-t_{k+1})}    \  \mbox{ and } \    C_{1,k}(b,\sigma,T,f)=  [f]_{\rm Lip}^2+ \frac{C_{2,k}(b,\sigma,T,f)}{q} .
\end{equation}
To complete the proof, it  suffices to control the remaining  terms in  Equation~\eqref{EqTermsToControl}. Using the (conditional)  Schwarz  inequality  yields
    \begin{equation*}
\Big| \mathds{\hat E}_k \big[ \big(\tilde Y_{k+1} -\hat Y_{k+1} \big) \big(1 - \Delta_n \hat B_k -  \hat C_k  \Delta W_{t_{k+1}} \big) \big] \Big|   \leq    \big[\hat{\mathds E}_k(\tilde Y_{k+1} - \hat Y_{k+1})^2 \big]^{\frac 12}   \big[\hat{\mathds E}_k(1- \Delta_n \hat B_k -  \hat C_k  \Delta W_{t_{k+1}} )^2 \big]^{\frac 12} .
  \end{equation*}
Furthermore, using the fact that  $\hat{\mathds E}_k(\Delta W_{t_{k+1}}) = \hat{\mathds E}_k( \mathds E_k(\Delta W_{t_{k+1}})) = 0$  and owing to the measurability of $\hat B_k$ and $\hat C_k$ with  respect to $\sigma(\hat X_k)$, we get
\begin{eqnarray*}  
  \hat{\mathds E}_k \big[(1- \Delta_n \hat B_k - \hat C_k  \Delta W_{t_{k+1}} )^2 \big] & = &  (1-\Delta_n \hat B_k)^2 +  \hat C_k^2  \hat{\mathds E}_k ((\Delta W_{t_{k+1}})^2) \\
  & =  &  (1-\Delta_n \hat B_k)^2 + \hat C_k^2 \Delta_n \\
  & \leq & (1 + \Delta_n [f]_{\rm Lip})^2 + (\Delta_n [f]_{\rm Lip})^2   \le   e^{2\Delta_n  [f]_{\rm Lip} }.
\end{eqnarray*}
Then, using the conditional Schwarz inequality and again  the contraction property of conditional expectation, we get
 \begin{equation}  \label{ControlSecondTerm}
\Big\|\mathds{\hat E}_k \big[ \big(\tilde Y_{k+1} -\hat Y_{k+1} \big) \big(1 - \Delta_n \hat B_k -  \hat C_k  \Delta W_{t_{k+1}} \big) \big] \Big\|_2\le e^{\Delta_n  [f]_{\rm Lip}}\| \tilde Y_{k+1} - \hat Y_{k+1}\|_2.
 \end{equation}
%[Using  the orthogonality property of  the conditional expectation and putting all the previous  estimates   together  leads to \textcolor{red}{L\`a tu vas trop vite}]
Using   Schwarz's Inequality  for the $L^2$-norm, we derive from~\eqref{EqFirstDecomp},~\eqref{EqTermsToControl},~\eqref{ControlFirstTerm} and~\eqref{ControlSecondTerm}  that
\begin{eqnarray}  \label{EqKeyProperty}
 \Vert  \tilde  Y_k - \hat Y_k \Vert_2^2   & = &     \Vert   \tilde Y_k - \hat {\mathds E}_k(\tilde Y_k)  \Vert_2^2    + \Vert  \hat{ \mathds E}_k( \tilde Y_k - \hat Y_k) \Vert_2^2 \\
\nonumber  & \leq  &  [y_k]_{\rm Lip}^2  \Vert  \bar X_k - \hat X_k \Vert_2^2+ \Big( e^{\Delta_n  [f]_{\rm Lip} }\| \tilde Y_{k+1} - \hat Y_{k+1}\|_2+ \Delta_n \Vert  f_k(\bar X_k, \mathds E_k(\tilde Y_{k+1}), \tilde \zeta_k)   \\
 \nonumber  & & -  f_k(\hat X_k, \hat{\mathds E}_k( \tilde Y_{k+1}), \hat{\mathds E}_k( \tilde \zeta_k)) \Vert_2\Big)^2\\
%K_k(b,\sigma,T,f)  \Vert X_k - \hat X_k \Vert_2^2 +   \gamma_n  \Vert  Y_{k+1}  - \hat Y_{k+1} \Vert_2^2 \nonumber,\\
 \nonumber& \leq  &  [y_k]_{\rm Lip}^2  \Vert  \bar X_k - \hat X_k \Vert_2^2 + \Big( e^{\Delta_n  [f]_{\rm Lip} }\| \tilde Y_{k+1}  - \hat Y_{k+1}\|_2\\
\nonumber   && \hskip 4 cm + \Delta_n   \Big(C_{1,k}(b,\sigma,T,f) +\frac{C_{2,k}(b,\sigma,T,f)}{\Delta_n}\Big)^{\frac 12} \Vert  \bar X_k  - \hat X_k \Vert_2     \Big)^2.
\end{eqnarray}
We first deal with the second term on the right hand side of the above inequality. Using  the classical inequality
\[
(a+b)^2  \le a^2(1+\Delta_n)+b^2\big(1+\Delta_n^{-1}\big),
\]
we derive that
%\textcolor{red}{il faudra re-arranger cela proprement\dots}
\begin{eqnarray*} 
&& \Big( e^{\Delta_n  [f]_{\rm Lip} }\| \tilde Y_{k+1} - \hat Y_{k+1}\|_2 +  \Delta_n   \Big(C_{1,k}(b,\sigma,T,f) +\frac{C_{2,k}(b,\sigma,T,f)}{\Delta_n}\Big)^{\frac 12} \Vert  \bar X_k  - \hat X_k \Vert_2     \Big)^2 \\
&  & \le e^{\Delta_n  [f]_{\rm Lip} }(1+\Delta_n)\| \tilde Y_{k+1} - \hat Y_{k+1}\|^2_2 \\
&& \quad + \Big(1+\frac{1}{\Delta_n}\Big)\Delta^2_n \left(C_{1,k}(b,\sigma,T,f) +\frac{C_{2,k}(b,\sigma,T,f)}{\Delta_n}\right) \Vert  \bar X_k  - \hat X_k \Vert^2_2\\
&& \le e^{\Delta_n (1+ [f]_{\rm Lip}) }\| \tilde Y_{k+1} - \hat Y_{k+1}\|^2_2 + \Big(1+ \Delta_n\Big) \Big(C_{1,k}(b,\sigma,T,f) \Delta_n+ C_{2,k}(b,\sigma,T,f) \Big) \Vert \bar X_k  - \hat X_k \Vert^2_2 .
\end{eqnarray*}
Hence (using that  $\Delta_n\le T/n_0$, if  $n \ge n_0$), we obtain, for every $k\!\in \{0,\ldots,n-1\}$,   
\begin{eqnarray}  \label{EqKeyProperty2}
\Vert   \tilde  Y_k - \hat Y_k \Vert_2^2   & \le  &  e^{\Delta_n (1+ [f]_{\rm Lip}) }\| \tilde Y_{k+1} - \hat Y_{k+1}\|^2_2  + \widetilde K_k(b,\sigma,T,f) \Vert  \bar  X_k  - \hat X_k \Vert^2_2
  \end{eqnarray} 
  where 
\begin{eqnarray*}
\widetilde K_k(b,\sigma,T,f) &:=  &[y_{k+1}]^2_{\rm Lip}       +\Big(1+\frac Tn\Big)\Big(C_{1,k}(b,\sigma,T,f)\frac{T}{n} +C_{2,k}(b,\sigma,T,f) \Big), \; k=0,\ldots,n-1,\\
&\le& K_k(b,\sigma,T,f)
 \end{eqnarray*}
It follows that, for every $k\!\in \{0,\ldots,n-1\}$,  
%\begin{eqnarray*}
\[
e^{\Delta_n k(1+[f]_{\rm Lip})} \Vert  \tilde Y_k - \hat Y_k \Vert_2^2\le e^{\Delta_n(k+1)(1+[f]_{\rm Lip})}\| \tilde Y_{k+1} - \hat Y_{k+1}\|^2_2 
% \\
%&& 
+  e^{\Delta_n k(1+[f]_{\rm Lip})} K_k(b,\sigma,T,f) \Vert \bar X_k  - \hat X_k \Vert^2_2.
\]
%\\\
%&\le &\sum_{\ell=k}^{n-1} e^{t^n_\ell(1+[f]_{\rm Lip})} K_\ell(b,\sigma,T,f) \Vert X_\ell  - \hat X_\ell \Vert^2_2 + e^{t^n_n(1+[f]_{\rm %Lip})} K_n(b,\sigma,T,f) \Vert X_n  - \hat X_n\Vert^2_2 
%\end{eqnarray*}
%00\textcolor{red}{Je me suis arr\^e te l\`a dans la correction du fichier} 
Keeping in mind that  $\Vert  \tilde Y_{n}  - \hat Y_{n}  \Vert_{2}^2  \leq [h]_{\rm Lip}^2  \Vert  \bar X_n - \hat X_n \Vert_2^2$,  we  finally derive by a  backward induction~that  
 \begin{equation*}
   \big \Vert  \tilde Y_k - \hat Y_k \big \Vert _2^2     \leq     \sum_{i=k}^{n}  e^{\Delta_n(i-k)(1+[f]_{\rm Lip})}  K_i(b,\sigma,T,f)  \big \Vert  \bar X_i -\hat X_i  \big \Vert_2^2.
   \end{equation*}
%where  $K_n(b,\sigma,T,f) =[h]^2_{\rm Lip}$.
%We may remark that    for every $n \geq 1$,   $[h]^{2}_{\rm Lip}  \leq  [y_{n}]_{\rm Lip}^2$ (see  Proposition~\ref{PropyKP1Lip})  and that for every $n \geq 0$, $\gamma_n  \leq e^{ [f]_{\rm Lip} (1+T)}$.  This leads to the announced statement.  

\smallskip
\noindent $(b)$ We derive from the very  definition  of $\tilde \zeta_k$ and $\hat \zeta_k$  that 
\[
\tilde \zeta_k -  \hat \zeta_k = \big(\tilde \zeta_k -\hat{\mathds E}_k(\tilde \zeta_k) \big) \stackrel{\perp}{+} \big(\hat{\mathds E}_k(\tilde \zeta_k)- \hat \zeta_k\big)
\]
where  $\stackrel{\perp}{+} $ means that both random variables are $L^2$-orthogonal. We know from~\eqref{eq:zeta} that 
\[
\big\Vert  \hat{ \mathds E}_k(\tilde \zeta_k -\hat{\mathds E}_k(\tilde \zeta_k)) \big\Vert_2^2\le \frac{ [z_k]^2_{\rm Lip}}{\Delta_n} \big\Vert \bar X_k-\hat X_k\big\Vert_2^2.
\]
On the  other hand, as $\sigma (\hat X_k)\subset \sigma(\bar X_k)\subset {\cal F}_k$, it is clear that $\hat{\mathds E}_k(\tilde \zeta_k)= \frac{1}{\Delta_n}\hat{\mathds E}_k(\tilde Y_{k+1}\Delta W_{t_{k+1}})$ so that
\[
\big\Vert \hat{\mathds E}_k(\tilde \zeta_k)- \hat \zeta_k\big\Vert_2^2  =  \frac{1}{\Delta^2_n}\big\Vert \hat{\mathds E}_k\big((\tilde Y_{k+1}- \hat Y_{k+1})\Delta W_{t_{k+1}}\big)\big\Vert_2^2.
\]
Conditional Schwarz's Inequality applied with $\hat \E_k$ implies that 
$$
\hat{\mathds E}_k\big((\tilde Y_{k+1}- \hat Y_{k+1})\Delta W_{t_{k+1}}\big)^2\le \big(\hat{\mathds E}_k(\tilde Y_{k+1}- \hat Y_{k+1})^2\big)\Delta_n
$$
which in turn implies that 
\[
\big\Vert \hat{\mathds E}_k(\tilde \zeta_k)- \hat \zeta_k\big\Vert_2^2  =  \frac{1}{\Delta_n}\big\Vert \tilde Y_{k+1}- \hat Y_{k+1}\big\Vert_2^2
\]
so that, finally,
\[
\hskip 2cm \Delta_n\big \Vert \hat{\mathds E}_k(\tilde \zeta_k)- \hat \zeta_k\big\Vert_2^2  \le \frac{C_{2,k}(b,\sigma, T,f)}{[f]^2_{\rm Lip}}\big \Vert \bar X_k-\hat X_k\Vert_2^2+ \Vert \tilde Y_{k+1}- \hat Y_{k+1}\big\Vert_2^2.\hskip 2cm \Box
\]
 %\end{proof}

\begin{rem}
The key property leading to  Theorem~\ref{TheoremPrincBsde}  and allowing to improve the  existing results for similar problems (see $e.g.$~\cite{BalPag}) is the Pythagoras like  equality~(\ref{EqKeyProperty}) which is true only for the quadratic norm.  This  equality  is the key to get  the sharp constant equal to $1$ before  the term   $ \Vert  \hat{ \mathds E}_k(\tilde Y_k - \hat Y_k) \Vert_2^2$.  
\end{rem}

\subsection{Computing  the $\hat \zeta_k$ terms}

Recall that  for every $k\!\in \{0,\ldots,n-1\}$,  the $\mathbb R^q$-valued random vector $\hat \zeta_k= (\hat \zeta_k^1,\ldots, \hat \zeta_k^q)$ reads
\[
\hat \zeta_k= \frac{1}{\Delta_n}  \hat z_k (\hat X_k) \quad\mbox{ where }\quad \hat z_k (\hat X_k) =  \hat{\mathds E}_k(\hat Y_{k+1}   \Delta W_{t_{k+1}})
\]
with $\hat z_k : \Gamma_k\to \mathbb R^q$ is a Borel function ($\Gamma_k$ is  the grid used to quantize $\bar X_k$). As $\hat Y_{k+1}=\hat   y_{k+1}(\hat X_{k+1})$ we easily derive that the function $\hat z_k$ is defined on $\Gamma_k= \{x^k_1,\ldots,x_k^{N_k}\}$ by
the ($\mathbb R^q$-valued) weighted  sum 
\[
\hat z_k(x^k_i) = \sum_{j=1}^{N_{k+1}}    \hat y_{k+1}(x^{k+1}_j) \, \pi^{W,k}_{ij}
\]  
where, for every $(i,j)\!\in \{1,\ldots,N_k\}\times \{1,\ldots,N_{k+1}\}$, $\pi^{W,k}_{ij}$ is an $\mathbb R^q$-valued vector given by 
\[
\pi^{W,k}_{ij} = \frac{1}{ \mathds P(\hat X_k=x^k_i)} \times  \mathbb{E}\Big(\Delta W_{t_{k+1}}\mbox{\bf 1}_{\{ \hat X_{k+1}=x^{k+1}_j,\, \hat X_k=x^k_i\}}\Big).
\]
 
These vector valued ``weights"    appear as new companion parameters (as well as the original weights $\pi^k_{ij}$ of the quantized transition matrices) which can be computed {\em on line} when simulating the Euler scheme  of the diffusion by a Monte Carlo simulation.

Note that, for every $k\!\in \{0,\ldots,n-1\}$ and  every $i\!\in \{1,\ldots,N_k\}$, 
\begin{eqnarray*}
\sum_{j=1}^{N_{k+1}} \pi^{W,k}_j = \hat{\mathds E}_k\big(\Delta W_{t^n_{k+1}}\mbox{\bf 1}_{\{\hat X_k = x^k_i\}}\big) &=&   \hat{\mathds E}_k\Big({\mathds E}_k \big(\Delta W_{t^n_{k+1}}\mbox{\bf 1}_{\{\hat X_k = x^k_i\}}\big)\Big)\\
&=& \hat{\mathds E}_k\Big({\mathds E}_k \big(\Delta W_{t^n_{k+1}}\big){\mathds P}\big(\hat X_k = x^k_i\big)\Big) =  \hat{\mathds E}_k\, 0=0.
\end{eqnarray*}
As a consequence, an alternative formula for $\hat z_k$  can be
\[
\hat z_k(x^k_i) = \sum_{j=1}^{N_{k+1}} \pi^{W,k}_{ij}\big(\hat y_{k+1}(x^{k+1}_j)-\hat y_k(x^k_i)\big).
\]

\section{Background and new results on   optimal vector quantization}  \label{SecOptiQuant}
 
 It is important to have in mind that all what precedes holds true for any quantizations $\hat X_{t_k}$ of the Euler scheme $\bar X_{t_k}$ $i.e.$ for any sequence of the form $\hat X_{t_k} = \pi_k( \bar X_{t_k})$ where $\pi_k:\mathbb R^d \to \mathbb R$ is Borel and $\pi_k(\mathbb R^d)$ is finite. In fact the theory of optimal vector quantization starts when tackling the problem of minimizing the $L^2$ (and more generally the $L^r$)-mean quantization error induced by this substitution, namely $\Vert \bar X_{t_k} -  \hat X_{t_k} \Vert_2$, which in turn will provide the lowest possible  error bounds for quantization based numerical schemes. This question is in fact a very old question that goes back to the1940's,  motivated by Signal transmission and processing. These techniques have been imported in Numerical Probability, originally for numerical integration by cubature formulas,  in the 1990's (see~\cite{Pag} or~\cite{Cher2}). 
 
\subsection{Short background} 
 Let      $X:(\Omega,\mathcal{A},\mathbb{P})\to \mathbb{R}^{d}$ be a    random vector lying in $L^r(\mathbb{P})$,  $r\!\in (0,+\infty)$.  The $L^r$-optimal quantization  problem of size $N$ for  $X$ (or equivalently  for  its  distribution $\mathbb P_X$) consists in  finding  the best $L^r(\mathbb{P})$-approximation of  $X$   by a random variable   $\pi(X)$  taking  at most  $N$ values.  The integer $N$ is called the {\em quantization  level}.
 
  First, we associate  to every Borel function $\pi :\mathbb R^d\to \mathbb R$ taking at most $N$ values  the induced $L^r(\mathbb P)$-mean   error  $\Vert  X - \pi(X)\Vert_r$ (where ${\Vert X \Vert}_{r} := {\left(\mathbb E \vert X \vert ^r \right)}^{1/r}$ is the usual $L^r(\mathbb{P})$-norm  on $(\Omega,\mathcal{A})$ induced by the norm $|\,.\,|$ on $\mathbb R^d$ ({\em a  priori} any norm, but always  the canonical Euclidean norm in this paper and in most applications). Note that when $r\!\in(0,1)$, the terms ``norm" is an abuse of language  since $L^r(\mathbb{P})$ is only a metric space metrized by $\|X-Y\|_{_r}^r$. 
 % and  where $ \vert \cdot \vert $ denotes an arbitrary  norm on $\mathbb{R}^d$.  
 As a consequence,  finding the best approximation of $X$ in the earlier described sense   boils down   to solve the following minimization problem:
 $$ 
e_{N,r}(X) = \inf{\Big\{ \Vert X - \pi(X) \Vert_{_r}, v: \mathbb{R}^d \rightarrow \Gamma,\, \Gamma \subset \mathbb{R}^d,  \textrm{ card}(\Gamma) \leq N \Big\}}
$$
where $\textrm{ card}(\Gamma)$ denotes the cardinality of the set $\Gamma$ (commonly called {\em grid} or  {\em codebook} depending on the field of application.  It is clear that for every  grid  $\Gamma =\{x_1, \ldots, x_N  \} \subset \mathds R^d$,  for any Borel function $\pi:\mathds R^d \rightarrow \Gamma$,
\[
\vert   \xi -\pi(\xi)  \vert   \geq {\rm dist}(\xi,\Gamma) =\min_{1 \le i \le N} \vert  \xi - x_i \vert.
\]
Equality holds if and only if  $\pi$ is a  Borel nearest neighbor projection $\pi_{\Gamma}$ defined by 
\[
\pi_{\Gamma}(\xi)  = \sum_{i=1}^N x_i \mbox{\bf 1}_{C_i(\Gamma)}(\xi)
\]
where $(C_i(\Gamma))_{i=1, \ldots,N}$ is a Borel partition of $\mathds R^d$ satisfying 
%
%This  problem  reads (see $e.g.$~\cite{GraLus})
%\begin{eqnarray}
%e_{N,r}(X)  & = & \inf{\Big\{ \Vert X - \hat{X}^{\Gamma} \Vert_{_r}, \Gamma  \subset \mathbb{R}^d,  \textrm{ card}(\Gamma) \leq N\Big \}}  \nonumber \\
% & = &  \inf_{ \substack{\Gamma   \subset \mathbb{R}^d \\ \textrm{card}(\Gamma ) \leq N}} \left(\int_{\mathbb{R}^d} d(\xi,\Gamma )^r d\mathbb P_X(\xi) \right)^{1/2} \label{er.quant}
%\end{eqnarray}
% where $\hat{X}^{\Gamma}$ is  a Voronoi quantization  of $X$ (or simply called quantization of $X$) induced  the {\em grid}  $\Gamma$. Such a Voronoi quantization is attached to a Borel nearest neighbor projection on $\Gamma$ or, equivalently to a 
% Borel    partition  $C_i(\Gamma)_{ i=1,\ldots,N}$ of $\mathbb{R}^d$  (called a Voronoi partition)  satisfying  
 $$
 \forall\, i \in \{1,\ldots,N\},\quad C_i(\Gamma) \subset \big\{ \xi \in \mathbb{R}^d : \vert \xi-x_i \vert = \min_{j=1,\ldots,N}\vert \xi-x_j \vert \big\}.
 $$ 
 Such a Borel partition is called a {\em Voronoi partition} (induced by $\Gamma$). The random variable $\hat X^{\Gamma}$ is called a  {\em Voronoi quantization}  of $X$ induced by $\Gamma$.  It follows that  for every $r >0$, $\Vert X - \hat X^{\Gamma} \Vert_r = \Vert {\rm dist}(X,\Gamma) \Vert_r$ does not depend on the choice of the Voronoi projection. Thus, we may denote $e_r(X,\Gamma)=\Vert X - \hat X^{\Gamma} \Vert_r$ the {\em $L^r$-mean quantization error} induced by the grid $\Gamma$ (under ${\mathbb P}_{_X}$). As a consequence, the optimal $L^r$-mean quantization error finally reads 
 \begin{equation} \label{er.quant}
e_{N,r}(X)  =  \inf{\big\{ \Vert X - \hat{X}^{\Gamma} \Vert_{_r}, \Gamma  \subset \mathbb{R}^d,  \textrm{ card}(\Gamma) \leq N\big \}}  .
\end{equation}
 
% The related Borel nearest neighbor projection $v_{\Gamma}$ on $\Gamma$ is defined by
% \[
% v_{\Gamma}(\xi)= \sum_{i=1}^N  x^N_i  \mbox{\bf{1}}_{C_i(\Gamma)} (\xi)
% \]
% so that $\hat{X}^{\Gamma} $ reads
% $$ 
%\hat{X}^{\Gamma} = \sum_{i=1}^N  x^N_i  \mbox{\bf{1}}_{\{X \in C_i(\Gamma)\}} .
%$$ 

Note that for every  level $N \geq 1$, the infimum in $(\ref{er.quant})$ is in fact a minimum, $i.e.$, it is  attained   at  least at one  grid  $\Gamma_{_N}$, see $e.g.$~\cite{GraLus} or~\cite{Pag}. Any such grid  is called an  {\em $L^r$-optimal   $N$-quantizer} and the resulting  Borel nearest neighbor projection is called an  {\em $L^r$-optimal   $N$-quantization}. It should be noticed as well that $e_{N,r}(X)$ is 
entirely characterized by the distribution $\mathbb P_{_X}$ of $X$, hence will be often denoted by $e_{N,r}(\mathbb P_{_X})$.

One shows that if  $\textrm{card(supp}(\mathbb P_X)) \geq N$ then  any  optimal $N$-quantizer is of full size  $N$. Furthermore  (see again~\cite{GraLus} or~\cite{Pag}),  the   optimal $L^r$-mean quantization error $e_{N,r}(X)$ at level $N$ 
% ($=e_{r}(\Gamma_N,X)$) 
decreases to $0$ as $N$ goes to infinity. Its rate of convergence is ruled by the so-called Zador Theorem  recalled below, in which, $|\,.\,|$ {\em temporarily} may denote {\em any} norm on $\mathbb R^d$.  

\begin{thm}{\bf Zador's Theorem}  \label{ThemZadPierce} $(a)$ Sharp asymptotic rate (see~\cite{GraLus}): Let $X:(\Omega, {\cal A}, \mathbb P)\to \mathbb R^d$ be a  random  vector  such that  $ X \in L^{r+\delta}(\mathbb P)$  for some real number $\delta>0$ and  let $\mathbb P_X=\varphi.\lambda_d+ \mathbb P_{_X}^s$ 
denote the canonical Lebesgue decomposition of  $\mathbb P_X$
% with respect to the Lebesgue measure $\lambda_d$, 
where 
%$\varphi$ denotes the Lebesgue density of $\mathbb P_X$ and 
$\mathbb P_{_X}^s$ stands for  the singular part of  $\mathbb P_X$. Then 
\begin{equation}\label{eq:Zador1}
 \lim_{N \rightarrow +\infty} N^{r/d} (e_{N,r}(\mathbb P_{_X}))^r = J_{r,d} \, \Vert \varphi \Vert_{\frac{d}{d+r}} \!\in [0,+\infty)
  \end{equation} 
  
  \vskip -0.75cm
\begin{equation}\label{eq:Zador2}
\hskip -0.1cm\mbox{with }\quad  \Vert \varphi \Vert_{\frac{d}{d+r}} = \left( \int_{\mathbb{R}^d} \varphi^{\frac{d}{d+r}} d\lambda_d \right)^{\frac{d+r}{d}}
 \; \mbox{  and }\; J_{r,d, |.|} = \inf_{N \geq 1} N^{r/d} e_{N,r}^r(U([0,1]^d)) \in (0,+\infty)  
 \end{equation}
($ U([0,1]^d)$  denotes the uniform distribution on the hypercube  $[0,1]^d$).  

\smallskip
\noindent $(b)$ Non-asymptotic bound (see~\cite{GraLus,LUPAaap}).   Let $r'>r$. There exists a universal real constant $C_{r,r',d}\!\in (0,+\infty)$ 
%and $N_{p,p',d}$ 
such that, for every $\mathds R^d$-valued   random vector $X$,
\[
\forall\, N\ge 1,\quad e_{N,r}(\mathbb P_{_X}) \le C_{r,r',d}\,\sigma_{r'}(X) N^{-\frac 1d}
\]
where $\sigma_{r'}(X):= \inf_{a\in \mathds R^d}\|X-a\|_{r'}\leq +\infty$ is the $L^{r'}$-(pseudo-)standard deviation of $X$.
\end{thm}

\noindent {\bf Numerical aspects (few words about).} From the numerical probability viewpoint, finding an optimal $N$-quantizer  $\Gamma$ is a challenging task, especially in higher dimension ($d\ge 2$). In this paper as in many applications we will mainly focus on the quadratic case $r=2$.  Note that, in practice, $|\,.\,|$ will be  the canonical Euclidean norm on $\mathbb R^d$ for numerical implementations.

\smallskip The key property to devise    procedures  to search for optimal quantizers rely on the following differentiability property of the squared quadratic quantization error (also known as quadratic distortion function) for a fixed level $N$ (and with respect to the canonical Euclidean norm). First, we define  the distortion function $ D_{N,2}$ (which is defined on $(\mathbb R^d)^N$ and not on the set of grids of size at most $N$) by:
\begin{eqnarray}\label{EqDistor}
\forall\, x=(x_1, \ldots,x_{_N})\!\in (\mathbb R^d)^N,\quad D^X_{N,2} (x)  &= &\int_{\mathbb{R}^d} \min_{1 \le i \le N}  \vert \xi - x_i\vert^2 d\mathbb P_X(\xi).   
%&= &\sum_{i=1}^N \int_{C_i(\Gamma^{x})} \vert \xi - x_i^{N} \vert^2 d\mathbb P_X(\xi),
\end{eqnarray}
To  an $N$-tuple $x=(x_1, \ldots,x_N) \!\in (\mathbb R^d)^N$ we associate its   grid  of values $\Gamma^x= \{x_1,\ldots,x_{_N}\}$ so that  $D^X_{N,2} (x)  = \Vert  X - \hat X^{\Gamma^x}\Vert_2^2$. In particular, it is clear that
\[
e_{N,2} (\mathbb P_{_X}) = \inf_{x \in (\mathds R^d)^N} D_{N,2} (x)
\]
since an $N$-tuples  can contain repeated values. 
\begin{prop}[see Theorem~4.2 in~\cite{GraLus}] \label{PropDifferentiability}$(a)$ 
The function $D_{N,2}$ is differentiable  at any $N$-tuple $x \in (\mathbb R^{d})^N$ having pairwise distinct components  and satisfying the following boundary negligibility assumption:
$$ 
\mathbb{P}_{_X}\big(\cup_{1\le i\le N} \partial  C_i(\Gamma^{x} )\big)=0.  
$$ 
Its gradient is given by 
\begin{equation}\label{EqStation}
 \nabla D^X_{N,2}(x)  =  2  \Big(  \int_{C_i(\Gamma^x) }  (x_i - \xi ) d \mathbb P_X (\xi)  \Big)_{i=1,\ldots,N}.
\end{equation}

\noindent $(b)$ The above negligibility assumption on the Voronoi partition boundaries does not depend on the selected partition. It holds in particular  when the distribution of $X$ is {\em strongly  continuous} $i.e.$ assigns no mass to hyperplanes and, for any distribution $\mathbb P_{_X}$ such that $\textrm{card(supp}(\mathbb P_X)) \geq N$,  when $x\!\in {\rm argmin }  D_{N,2}$. 
\end{prop}

The result  is a consequence of the interchange of  the differentiation and the integral leading to~\eqref{EqStation} when formally differentiating~\eqref{EqDistor} (see~\cite{GraLus, Pag}). Consequently,  any $N$-tuple $x \!\in {\rm argmin }   D_{N,2}$ satisfies 
\[
\nabla D_{N,2}(x)=0. 
\]
Note that this equality also reads, still under the assumption $\textrm{card(supp}(\mathbb P_X)) \geq N$,
\begin{equation*}
%\label{EqStation}
\mathbb{E} \left(X  \vert  \hat{X}^{\Gamma} \right) =\hat{X}^{\Gamma}.
\end{equation*}
%where $\Gamma= \{x_1,\ldots,x_{N}\}$. Except when $d=1$ and $\mathbb P_{_X}$ has a log-concave density where the minimal quantizer is unique, the converse is not true. This leads naturally to the following definition.
%
%\begin{defn}An $N$-quantizer  $\Gamma = \{ x_1,\ldots,x_N \}$  inducing the quantization $\hat{X}^{\Gamma}$  of $X$  is  stationary  if  
%\begin{equation*}
%\forall \ i \neq j, \quad \ \ x_{i}\neq x_{j} \ \ \mbox{and} \ \ \mathbb{P}_{_X}\left(\cup_{i} \partial  C_i(\Gamma) \right)=0
%\label{eq11}
%\end{equation*}
%and satisfies the above stationary equation~\eqref{EqStation}.
%%\begin{equation}\label{EqStation}
%%\mathbb{E} \left[X  \vert  \hat{X}^{\Gamma_{_N}} \right] =\hat{X}^{\Gamma_{_N}}.
%%\end{equation}
%\end{defn}
%

%T
%we are sometimes led to find some ``good'' quantizations  $\hat X^{\Gamma_{_N}}$ which are close to $X$ in distribution, so that  for every Borel function $f: \mathbb R^d \mapsto \mathbb R$, we can approximate $\mathbb E[ f(X)]$ by  
%\begin{equation} \label{QuantProcedureEstim}
%\mathbb{E}  f \big(\hat{X}^{\Gamma_{_N}} \big) = \sum_{i=1}^N f(x_i^N)  \ p_i,
%\end{equation}
%where $p_i =\mathbb{P}( \hat{X}^{\Gamma_{_N}} = x_i^N ).$  Among ``good'' quantizations of $X$  we have stationary quantizers defined as follows.  

%
% a  stationary  quantizer  $\Gamma_{_N} = \{ x^N_1,\ldots,x^N_N \}$ is in fact an $N$-quantizer satisfying the stationary equality: $\nabla D_{N,2} (x^N) = 0.$ 
% 
All numerical methods to compute optimal quadratic quantizers  are based on this result:
%For numerical  implementations, stationary quantizers search procedures amount to zero search 
recursive procedures like Newton's algorithm (when $d=1$), randomized fixed point procedures  like Lloyd's~I algorithms (see $e.g.$~\cite{GerGra, PagYu})     or recursive stochastic gradient descent like   the Competitive Learning Vector Quantization (CLVQ) algorithm (see~\cite{GerGra, Pag} or~\cite{PagPri03}) in the multidimensional framework.  However note that in higher dimension this equation has several  solutions (called {\em stationary quantizers}) possibly sub-optimal.  Optimal  quantization grids  associated to   the multivariate Gaussian random vector can be downloaded from  the website {\tt www.quantize.math-fi.com}. For more details about numerical methods we refer to the recent survey~\cite{PagCER} and the references therein.

\subsection{Distortion mismatch: $L^s$-robustness of $L^r$-optimal quantizers}
The distortion mismatch problem is the following: when does an $L^r$-optimal sequence of quantizers $(\Gamma_N)_{N\ge 1}$ for a random variable $X$ remain $L^s$-rate optimal for some $s>r$ (if $X\!\in L^s$)~? Or in more mathematical terms, if $X\!\in L^s$, $s>r$, when do we have for such a sequence of $L^r$-optimal quantizers
\[
\limsup_N  N^{\frac{1}{d}}e_s( \Gamma_N,X) <+\infty~?
\]
This problem has obvious applications in numerical probability since, for algorithmic reasons, one usually has access to optimal quadratic quantizers (see $e.g.$ the website \url{www.quantize.maths-fi.com}) whereas they are currently used in  a non quadratic framework. What  will be done  in Section~\ref{SecNonLinearFilt}  for nonlinear filtering is precisely to take advantage of this result to strongly relax some growth assumptions on the conditional densities involved in the Kallianpur-Striebel formula.

%By mismatch we mean the behavior of $e_s(\Gamma_N,X)$ where $(\Gamma_N)_{N\ge 1}$ is a sequence of $L^r$-optimal %quantization grids and $s>r$. 

The distortion mismatch problem was first addressed in~\cite{GraLusPag1} for various classes of distributions on $\mathds R^d$, in particular  for distributions having a  {\em radial density} satisfying (an almost necessary) moment assumption   of  order higher than $s$. In the theorem below we  extend this result to  {\em all random vectors satisfying   this moment condition}.
% (see Appendix~\ref{ap:B}  for a proof).
%(for more results in this direction,  see the monography in progress co-authored by the first author and H. Luschgy, see~%\cite{FQbook}), Chapter~2).  

\begin{thm}[$L^r$-$L^s$-distortion mismatch]\label{thm:Mismatchnew}  Let $X:(\Omega,{\cal A}, \mathds P)\to \mathds R^d$ be a random vector and let $r\!\in (0 +\infty)$. Assume that the distribution $\mathds P_{_X}$ of $X$ has a non-zero absolutely continuous component with density $\varphi$. Let  $(\Gamma_N)_{N\ge 1}$ be a sequence of  $L^r$-optimal grids and let $s\!\in (r,r+d)$. If 
 
 \begin{equation} \label{EqLrLsProblem}
 X \!\in L^{\frac{sd}{d+r-s}+\delta}(\Omega,{\cal A}, \mathds P) 
 \end{equation}
for some $\delta>0$, then
 \begin{equation} \label{eq:LrLsrobust}
\limsup_N  N^{\frac{1}{d}}e_s( \Gamma_N,X) <+\infty.
%\le  C_{r,d,\delta} \big\|X-a_1\big\|_{\frac{(d+\delta)s}{d+r-s}}^{\frac{d}{r+d}}\|X-a_1\|_{r(1+\frac{\delta}{d})}^{\frac{r}{r+d}} \times N^{-\frac 1d}
  \end{equation}
\end{thm}
\begin{defn}\label{def:(r,s)} Let $r\!\in (0,+\infty)$ and $s\!\in (r,r+d)$.  A random vector {\em $X\!\in L^r(\Omega, {\cal A}, \PP)$ has an $(r,s)$-distribution (or its distribution $\mathbb P_{_X}$ is an $(r,s)$-distribution)} if~\eqref{eq:LrLsrobust} is satisfied. 
\end{defn}

Note that  if $X$ has an $(r,s)$-distribution, then it has an $(r,s')$ distribution for any $s' \!\in (r,s)$ since the $L^s$-norm is increasing in $s$.

Thus, the integrability condition~\eqref{EqLrLsProblem} appears as a criterion to have an $(r,s)$-distribution (see also the first remark after the proof of the theorem).  Note that, as expected,  $\frac{sd}{d+r-s}>s$ so that  the preservation of the $(r,s)$-property  for $s>r$ requires more than $L^s$-integrability.  Finally, if $X$ has polynomial moment at any order, then the $(r,s)$-property holds for every $s\!\in (r,r+d)$.
%A proof of this theorem is provided in Appendix~\ref{ap:B}.

\bigskip
\noindent {\em Proof of Theorem~\ref{thm:Mismatchnew}.}
%\begin{proof}
 \noindent  {\sc Step~1} ({\em Control of the distance to the quantizers}): Let $(\Gamma_N)_{N\ge 1}$ be a sequence of  $L^r$-optimal   quantizers. It is clear that, for every $\xi \!\in \mathds R^d$, 
 \[
 d(\xi, \Gamma_N) \le |\xi| + d(0, \Gamma_N).
 \]
 The sequence $\big(d(0, \Gamma_N)\big)_{N\ge 1}$ is bounded since $d(\Gamma_N, {\rm supp}(\mathbb P_{_X})^c)\to 0$ as $N\to +\infty$ and $d(0, {\rm supp}(\mathbb P_{_X})^c)<+\infty$. Then there exists a real constant $A_{_X}\ge 0$ such that 
 for every $\xi \!\in \mathds R^d$, 
 \[
 d(\xi, \Gamma_N) \le |\xi| + A_{_X}.
 \]

 \medskip
\noindent {\sc Step~2} ({\em Micro-macro inequality}): The optimality of the grids $\Gamma_N$, $N\ge 1$, allow to apply to the micro-macro inequality (see Equation~(3.2)  in the proof of Theorem~2 in~\cite{GraLusPag1}), namely~: for every real constant $c\!\in (0,\frac 12)$ and every $y\!\in \mathds R^d$, 
\begin{equation}\label{eq:micro-macro}
e_r(\Gamma_N,X)^r -e_r(\Gamma_{N+1},X)^r \ge \big((1-c)^r-c^r\big) \mathbb P_{_X}\big(B\big(y;cd(y,\Gamma_N)\big)\big)d(y,\Gamma_N)^r.
\end{equation}
Let $\nu$ be an auxiliary  Borel probability measure on $\mathds R^d$ to be specified further on. Set $C(r) =(1-c)^r-c^r $. Integrating the above inequality with respect to $\nu(dy)$ yields,  owing to Fubini's Theorem,
\begin{eqnarray*}
e_r(\Gamma_N,X)^r -e_r(\Gamma_{N+1},X)^r &\ge& C(r) \int\!\!\int \!\!\big(B\big(y;cd(y,\Gamma_N)\big)\big)d(y,\Gamma_N)^r \mathbb P_{_X}(d\xi)\nu(dy)\\
&=& C(r) \int\!\!\int\!\! {\bf 1}_{\{|y-\xi|\le cd(y,\Gamma_N)\}}d(y,\Gamma_N)^r\nu(dy) \mathbb P_{_X}(d\xi)\\
&\ge&C(r)  \int\!\!\int \!\!{\bf 1}_{\{|y-\xi|\le cd(y,\Gamma_N),\, d(y,\Gamma_N) \ge \frac{1}{c+1}d(\xi,\Gamma_N) \}}d(y,\Gamma_N)^r\nu(dy) \mathbb P_{_X}(d\xi).
\end{eqnarray*}
Now using that $\xi\mapsto d(\xi, \Gamma_N)$ is Lipschitz continuous with coefficient $1$, one derives that 
\[
\big\{(\xi,y)\,:\,  |y-\xi|\le \frac{c}{c+1}d(\xi,\Gamma_N)\big\} \subset \Big\{(\xi,y)\,:\, |y-\xi|\le cd(y,\Gamma_N),\, d(y,\Gamma_N) \ge \frac{1}{c+1}d(\xi,\Gamma_N) \Big\}
\]
and, still by Fubini's Theorem, 
\begin{equation}\label{eq:micro-macrofinal}
e_r(\Gamma_N,X)^r -e_r(\Gamma_{N+1},X)^r \ge \frac{C(r)}{(1+c)^r}  \int \nu\Big(B\big(\xi ;cd(y,\Gamma_N)\big)\Big) d(\xi, \Gamma_N)^r \mathbb P_{_X}(d\xi).
\end{equation}
Let $\varepsilon\!\in (0,1/2)$. We set  $\nu = f_{\varepsilon,\delta}.\lambda_d$ where $f_{\varepsilon, \delta}$ is a probability density given by 
\[
f_{\varepsilon, \delta}(\xi)= \frac{\kappa_{\varepsilon, \delta}}{(|x|+1+\varepsilon)^{d+\delta}} \;\mbox{ with }\; \delta>0.
\]
%($|\,.\,|$ is any norm). 
The density $f_{\varepsilon, \delta}$ shares the following property on balls:  let $\xi\!\in \mathds R^d$ and $t\in \mathds R_+$. If $t\le \varepsilon (|\xi|+1)$, then 
\[
\nu\big(B(\xi,t)\big)\ge g_{\varepsilon,\delta}(\xi)t^d\quad \mbox{ with }\quad g_{\varepsilon,\delta}(\xi)=\frac{1}{(1+\varepsilon)^{d+\delta}} \frac{\kappa_{\varepsilon, \delta}}{(|\xi|+1)^{d+\delta}}V_d
\]and $V_d = \lambda_d\big(B(0;1)\big)$. Now let $c=c(\varepsilon)\!\in (0,1)$
%= \frac{\varepsilon}{1-\varepsilon}$ so that 
such that $\frac{c}{c+1}= \varepsilon(A^{-1}_{_X}\wedge1)$. As   $d(\xi,\Gamma_N)\le |\xi|+A_{_X}$, this in turn implies that $\frac{c}{c+1}d(\xi,\Gamma_N)\le \varepsilon( |\xi|+1)$.   As a consequence  
\[
e_r(\Gamma_N,X)^r -e_r(\Gamma_{N+1},X)^r \ge \frac{C(r)}{(c+1)^r}  \int g_{\varepsilon,\delta}(\xi) d(\xi, \Gamma_N)^{r+d} \mathbb P_{_X}(d\xi).
\]

  Let $s\!\in [r,r+d)$. It follows from Equation~\eqref{eq:micro-macrofinal} and the reverse H\"{o}lder inequality applied with $p=\frac{s}{r+d} \!\in (0,1)$ and $q=-\frac{s}{d+r-s}\!\in (-\infty, 0) $  that 
\[
 \int g_{\varepsilon,\delta}(\xi) d(\xi, \Gamma_N)^{r+d} \mathbb P_{_X}(d\xi)   \ge  \left[ \int_{\mathds R^d}d(\xi,\Gamma_N)^s \mathbb P_{_X}(d\xi) \right]^{\frac{r+d}{s}}\left[\int g_{\varepsilon,\delta,a}(\xi)^{-\frac{s}{d+r-s}} \mathbb P_{_X} (d\xi)\right]^{-\frac{d+r-s}{s}}.
\]

It follows  from the assumption made on $X$ (or $\mathbb P_{_X}$) that, for small enough $\delta>0$, 
$$
\left[ \int_{\mathds R^d} g_{\varepsilon, \delta}^{-\frac{s}{d+r-s}}(\xi) \mathbb P_{_X}(d\xi) \right] ^{-\frac{d+r-s}{s}} = \frac{\kappa_{\varepsilon,\delta V_d}}{(1+\varepsilon)^{d+\delta}} \left[ \mathds E\Big[\big(1+|X|)^{\frac{(d+\delta)s}{d+r-s}}\Big]\right] ^{-\frac{d+r-s}{s}}
%\le    C_{r,s,\varepsilon, \delta} \mathds E\,|X-a|^{\frac{ds}{d+r-s}+\eta}
<+\infty.
$$
%where $ C_{r,s,\varepsilon, \delta}= \left[\frac{(1+\varepsilon)^{d+\delta}}{\kappa_{\varepsilon,\delta}} \right]^{\frac{(d+\delta)s}{d+r-s}}$. 
As a consequence
\begin{equation}\label{eq:micro-macrofinalrs}
e_r(\Gamma_N,X)^r -e_r(\Gamma_{N+1},X)^r \ge C_{X,r,s,\varepsilon, \delta} \, 
e_s(\Gamma_N,X)^{r+d}
\end{equation}
 where $ C_{X,r,s,\varepsilon, \delta}=   \frac{(1-c)^r -c^r}{(1+c)^r(1+\varepsilon)^{d+\delta}} \kappa_{\varepsilon, \delta}\big\|1+|X|\big\|^{-(d+\delta)}_{\frac{(d+\delta)s}{d+r-s}}$.

\medskip
\noindent {\sc Step~3} ({\em Upper-bound for the quantization error increments}): Since the distribution of $X$ is absolutely continuous $X$ ($i.e.$ admits  a density), one  derives following the lines of the proof of Theorem~2 in~\cite{GraLusPag1}  this upper-bound for the increments of the $L^r$-quantization error: there exists a  real constant $\kappa_{_{X,r}}>0$ such that 
\[
e_r(\Gamma_N,X)^r -e_r(\Gamma_{N+1},X)^r\le \kappa_{_{X,r}} N^{-1-\frac rd}.
\]
Combining this inequality with~\eqref{eq:micro-macrofinalrs} yields
\[
 \Big[
e_s(\Gamma_N,X)^s\Big]^{\frac{r+d}{s}}\le \tilde C_{X,r,s,\varepsilon, \delta} N^{-\frac{r+d}{d}}
\]
where $\displaystyle  \tilde C_{X,r,s,\varepsilon, \delta} = \frac{\kappa_{_{X,r}} }{C_{X,r,s,\varepsilon, \delta}}$. This completes the proof by considering the $(d+r)^{th}$ root of the inequality.~$\hfill \Box$
%\end{proof}

\bigskip

 \noindent {\bf Remarks.} $\bullet$ If $\varphi$ is radial, more precisely if  $\varphi= \tilde\varphi(|x|_0)$ where $\tilde \varphi:\mathds R_+\to \mathds R_+$ is bounded and non-increasing on $[R_0,+\infty)$ and $|\,.\,|_0$ denotes any norm on $\mathds R^d$ the above result holds true even if $\delta=0$ (see~\cite{FQbook}). 
 
 \smallskip 
 \noindent $\bullet$ Criterion~\eqref{EqLrLsProblem}  is  close to optimality  for the following reason. It has been established in~\cite{GraLusPag1} (Theorem~1) that if $X\!\in L^{r+}(\Omega,{\cal A}, \mathds P)$, and if $(\Gamma_N)_{N\ge 1}$ is a sequence of $L^r$-asymptotically optimal quantization grids, then 
 \[
 \varliminf_N N^{\frac 1d} e_s(\Gamma_N,X)\ge J^{\frac 1s}_{r,d,|.|}\left[\int_{\mathds R^d}\varphi^{\frac{d}{r+d}}d\lambda_d\right]^{\frac{1}{d}}\left[ \int_{\mathds R^d}\varphi^{\frac{d+r-s}{r+d}}d\lambda_d\right]^{\frac 1s}
 \]
where  $J^{\frac 1s}_{r,d,|.|}$ is given by~\eqref{eq:Zador2}. Since $X\in L^{r+}(\Omega,{\cal A}, \mathds P)$, $\int_{\mathds R^d} \varphi^{\frac{d}{r+d}}d\lambda_d <+\infty$ by  an elementary application of the reverse H\"{o}lder inequality  (see Equation~(2.11) from~\cite{GraLusPag1}). On the other hand, 
 \[
  \int _{\mathds R^d} \varphi^{\frac{d+r-s}{r+d}}d\lambda_d = +\infty \Longrightarrow  X\notin L^{\frac{ds}{d+r-s}}(\Omega,{\cal A}, \mathds P).
 \]

\section{Numerical experiments for the BSDE scheme}\label{sec:Numeric}
To illustrate empirically the improved theoretical rate obtained in the previous section,  we deal here with  two toy examples: a bull-call spread option  (in a market where the risk free returns for the borrower and the lender are different) and a multidimensional example with the Brownian motion.  Note that our aim is  not to make an extensive numerical test with  a complete description (or a complexity analysis)  of several used algorithms for the optimal grid search. These subjects  have extensively been considered in the past and we refer for example  to~\cite{PagPri03} for more details. 

%We deal now with numerical experiments using two examples: the pricing  of
% %a call option (in a market where the risk free return for the borrower and the lender are the same), 
% a bull-call spread option  (in a market where the risk free returns for the borrower and the lender are different) and a multidimensional example with the Brownian motion. 

Numerical tests are performed using our  quantized BSDE algorithm.  At each discretization instant $t_k$, we associate a quantization grid $\Gamma_k=\{x^k_i, \, i=1,\ldots,N_k\}$ of size $N_k$,  possibly not optimal {\em a	 priori}, and $\hat X_k = {\rm Proj}_{\Gamma_k}(\bar X_k)$ the resulting Voronoi quantization of $\bar X_{t_k}$.  Then, we set  for every  $k= 0,\ldots,n-1$, $i=1, \ldots, N_k$, $j=1, \ldots,N_{k+1}$, the {\em transition weights} (or probabilities)
\[
p_{ij}^k = \mathds P(\hat X_{k+1} = x_j^{k+1} \vert \,\hat X_{k} = x_i^k),  \ k=0, \ldots, n-1.
\]
and, for $k= 0,\ldots,n$, $i=1, \ldots, N_k$, the {\em marginal weights} $p_i^k  = \mathds P(\hat X_k  = x_i^k)$,  $k=0, \ldots,n $.

%\begin{eqnarray*}
%& & p_i^k  = \mathds P(X_k  = x_i^k),  k=0, \ldots,n \\
%& \textrm{ and } \quad  & p_{ij}^k = \mathds P(\hat X_{k+1} = x_j^{k+1} \vert \hat X_{k} = x_i^k),  \ k=0, \ldots, n-1.
%\end{eqnarray*}
Setting $\hat Y_k = \hat y_k(\hat X_k)$, for every $k \in \{ 0, \ldots,n\}$, the quantized BSDE scheme reads as
\begin{equation*}
   \left \{ \begin{array}{l}
   \hat y_{n}(x_i^n)  =     h( x_i^{n}) \hspace{6cm } i=1, \ldots,N_n  \\
   \hat y_{k}(x_i^k)   = \hat {\alpha}_k(x_i^k)   + \Delta_n f \big(t_k, x_i^{k}, \hat {\alpha}_k(x_i^k), \hat {\beta}_k(x_i^k)  \big) \qquad  i=1, \ldots,N_k   
    \end{array}  
 \right.
 \end{equation*}
%\begin{eqnarray*}
% \end{eqnarray*}
 where for $\ k=0, \ldots,n-1$,
\begin{equation}
\hat{\alpha}_k(x_i^k) = \sum_{j=1}^{N_{k+1}}  \hat y_{k+1}(x_j^{k+1}) \, p_{ij}^{k} \quad \textrm{ and } \quad \hat {\beta}_k(x_i^k)   =  \frac{1}{\Delta_n} \sum_{j=1}^{N_{k+1}}    \hat y_{k+1}(x^{k+1}_j) \, \pi^{W,k}_{ij}
\end{equation}
with 
\[
\pi^{W,k}_{ij}    =\frac{1}{p_i^k } \times  \mathbb{E}\Big(\Delta W_{t_{k+1}}\mbox{\bf 1}_{\{ \hat X_{k+1}=x^{k+1}_j,\, \hat X_k=x^k_i\}}\Big).
\]
We use a time discretization mesh of length $n=20$ for the first example and of length $n=10$ for all dimensions in the second example. In both examples below, the  quantizers $\hat X_k$, $k=1, \ldots, n$ (with $\hat X_0 =X_0$) are  computed from a scaling of the optimal grid of  ${\cal N}(0,I_d)$ Gaussian distributions available on the website devoted to quantization

\smallskip
\centerline{
\url{www.quantize.maths-fi.com}. 
}

\smallskip
%followed (except for $d=1, 2$, where the scaling already lead to good results) by a Lloyd's algorithm (see $e.g.$~\cite{GerGra,PagPri03} for details on Lloyd's algorithm).  
The transition probabilities  are approximated using a Monte Carlo simulation of  size   $ 10^7$ for all examples (keep in mind that we may have the same precision with a smaller size of Monte Carlo trials but our aim is not to optimize these sizes of the trials).   For simplicity reasons, we use a uniform dispatching  across the time layers for the quantizers by assigning  the same grid size $N_k$  to all $\hat X_k$ at  every discretization step $t_k$, $k=1, \ldots, n$. Once the optimal quantization grids are computed offline, the complexity of our procedure depends on the grid sizes and varies from less than $1$ second in lower dimension up to a few minutes in dimension 5, almost entirely devoted to the computation by Monte Carlo simulation of the transition weights. By contrast, the (quantized) dynamic programming descent itself is  instantaneous.

\subsection{Bid-ask spread for interest rate}
Let us consider  a model  with two interest rates introduced  in \cite{Ber}:  a borrowing rate $R$ and a lending rate $r\le R$ where the 
stock price $(X_t)_{t \in [0,T]}$  evolves following the Black-Scholes dynamics
\[
dX_t = \mu X_t dt + \sigma X_t dW_t, \;X_0= x_0>0.
\] 
 Let $\varphi_t$  be the amount of assets held at time $t$.  Then, the dynamics of the  replicating portfolio is given by
\begin{equation}  \label{EqBSDENumNonLin}
Y_t  = Y_T  + \int_t^T f(Y_s,Z_s) ds -  \int_{t}^T Z_s dW_s
\end{equation}
where   $Z_t =  \sigma \varphi_t  X_t$ and the driver function $f$ is given by 
\[
f(y,z)  = -ry- \frac{\mu -r}{\sigma} z - (R-r) \min \big(y - \frac{z}{\sigma},0 \big).
\]
As in \cite{BenSte}, we consider  a  bull-call  spread comprising  a long call with strike $K_1=95$ and two short call with strike $K_2=105$, with payoff function
\[
 (X_T -K_1)^{+} - 2 (X_T-K_2)^{+} =  Y_T.
\]
Furthermore, we consider the set of parameters:
\[
X_0 = 100, \quad R=0.06, \quad r=0.01, \quad \mu=0.05, \quad  \sigma = 0.2, \quad T=0.25.  
\]
The BSDE~\eqref{EqBSDENumNonLin} has no analytical  solution. We refer to the reference prices given in  \cite{BenSte, RuiOos} where $(Y_0,Z_0)$ is approximated by $(2.96,0.55)$. We put $n =20$ and, for every $k=1, \ldots,n$, the grid sizes $N_k = \bar N = \frac{N}{n}$ is constant (keep in mind that $N = N_1+ \ldots + N_n$).  The quantizers $\hat X_{t_k}$ have been obtained by  using dilatations of optimal Gaussian quantization grids that we substitute to $W_{t_k}$ into  the formula $X_{t_k}= x_0e^{(\mu-\frac{\sigma^2}{2})t +\sigma W_{t_k}}$. 
% We checked  that further optimizations using Lloyd's algorithm starting from such sub-quantizers (for the $X_{t_k}$) has no significant influence on the numerical performances. Such an optimization of the whole chain $(X_{t_k})_{k=0,\ldots,n}$ takes less than one minute thanks to a GPU implementation. The above quantized BSDE scheme is almost instantaneous (less than $\frac{1}{10\,000}$s).  }

The numerical convergence rate of the error $\bar N\mapsto \vert  Y_0 - \hat Y_0^{\bar N} \vert$,  ${\bar N}= 5 \ell, \ell=1, \ldots, 30$,   is depicted in Figure~\ref{figCompRegMar}, including a  polynomial regression which  emphasizes the empirical order  of the convergence rate, namely  $\bar N^{-1}$. 
%The stabilization occurs very soon (with $\bar N \approx 35$ points per time layer).  
% The regression error   $0.005$  includes the time discretization error  \`a l'erreur de discr\'etisation en temps, \`a ton avis?}
%The estimated solution of the BSDE by optimal quantization is  $(\hat Y_0^N, \hat Z_0^N) = (2.97, 0.63)$ for an uniform dispatching  with  the same grid size $N_k=N=40$.

\subsection{Multidimensional example}
%\textcolor{red}{Habituellement $N$ designe le nombre total de points utilis\'e, pas le nombre de points par couche que l'on dessine souvent par $\bar N$ lorsque $N_k$ est constant. Il faudrait homogen\'eiser les notations.}

We consider the following non linear BSDE (example due to J.-F. Chassagneux):
\[
dX_t = dW_t, \qquad -dY_t  =  f(t,Y_t,Z_t) dt - Z_t \cdot dW_t, \quad  Y_{_T}= \frac{e_T}{1 +e_T}
\]
where 
\[
e_t = \exp(t+ W_t^1+ \ldots +W_t^d),    \quad t \in [0,T],
\]
 $f(t,y,z)  = (z_1+ \ldots +z_d) \big(y - \frac{2+d}{2d} \big)$ and  $W$ is a $d$-dimensional Brownian motion.  The solution of this BSDE is 
\begin{equation}
Y_t  = \frac{e_t}{1 +e_t}, \qquad Z_t = \frac{e_t}{(1+e_t)^2}.
\end{equation}
For the numerical experiments, we use the (regular) time discretization mesh with   $n=10$. 
% We  use the uniform dispatching grid allocation and define the quantization $(\hat W_{t_k})_{ 0 \le k \le n}$ of the Brownian trajectories  $(W_{t_k})_{ 0 \le k \le n}$  from the following recursive procedure
%\begin{equation}
%\hat W_{t_{k+1}} = \hat W_{t_k}  + \sqrt{\Delta} \,  \hat {\varepsilon},
%\end{equation}
%$\hat W_0 = 0$ and  where  $\hat {\varepsilon}$ is the optimal quantization of the $d$-dimensional standard Gaussian random variable. 
We choose $t=0.5$, $d=2,3,4,5$, so that $Y_0 =0.5$ and $Z_0^{i} = 0.24$, for every $i=1, \ldots, d$.   We  depict in Figures~\ref{figCompRate0} and \ref{figCompRate1}, the rates of convergence of $\vert \hat Y_0^{\bar N} - 0.5 \vert$ towards $0$, for the (constant) layer grid sizes $N_k=\bar N = \frac Nn= 5, \ldots, 150$.  The graphics in Figures~\ref{figCompRate0} and~\ref{figCompRate1} confirm  a rate of convergence of order $N^{-1/d}$. 
%In particular, when $d=2$, we get
% \begin{eqnarray*}
%& &  (\hat Y_0, \hat Z_0^1, \hat Z_0^2) = (0.50, 0.28, 0.28)  \textrm{ for }  N_k =40  \\
%& \textrm{ and }  &   (\hat Y_0, \hat Z_0^1, \hat Z_0^2) = (0.50, 0.23, 0.23) \textrm{ for } N_k =100.
% \end{eqnarray*}
%  When $d=3$, we get  
%  \begin{eqnarray*}
% & &  (\hat Y_0, \hat Z_0^1, \hat Z_0^2, \hat Z_0^3) = (0.51, 0.08, 0.06, 0.06) \textrm{ for }  N_k =40 \\
%  &  \textrm{ and } & (\hat Y_0, \hat Z_0^1, \hat Z_0^2, \hat Z_0^3) = (0.51, 0.18, 0.16,0.11)  \textrm{ for }  N_k =100.
%\end{eqnarray*}

 \begin{figure}[htpb]
 \begin{center}
  \!\includegraphics[width=10.5cm,height=8.0cm]{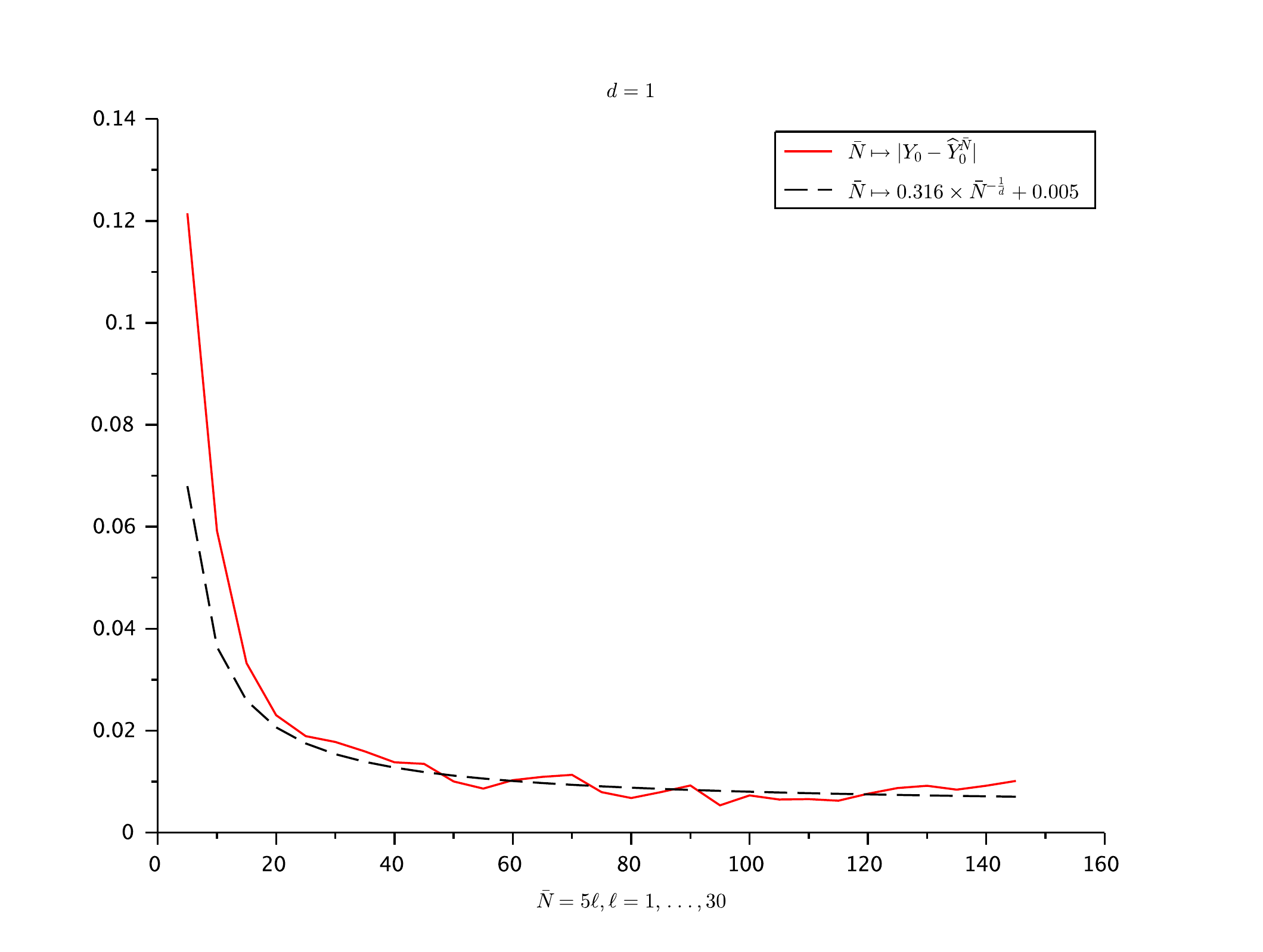}	
%\hfill   \!\includegraphics[width=8.5cm,height=7.0cm]{Y0Z0mu02.pdf}
  \caption{\footnotesize  Convergence rate of the quantization error for the Bid-ask spread in the Black-Scholes model.  Abscissa axis: the size $\bar N=5 \ell, \ell=1,  \ldots, 30$ of the quantization.  Ordinate axis:  The error $\vert Y_0 - \hat Y_0^{\bar N} \vert$ and the graph ${\bar N} \mapsto \hat a /\bar {N}  + \hat b$, where $\hat a$ and $\hat b$ are the regression coefficients.} 
  \label{figCompRegMar}
  \end{center}
\end{figure}

%\vskip -1 cm 
\begin{figure}[htpb]
% \begin{center}
  \!\includegraphics[width=8.5cm,height=6.0cm]{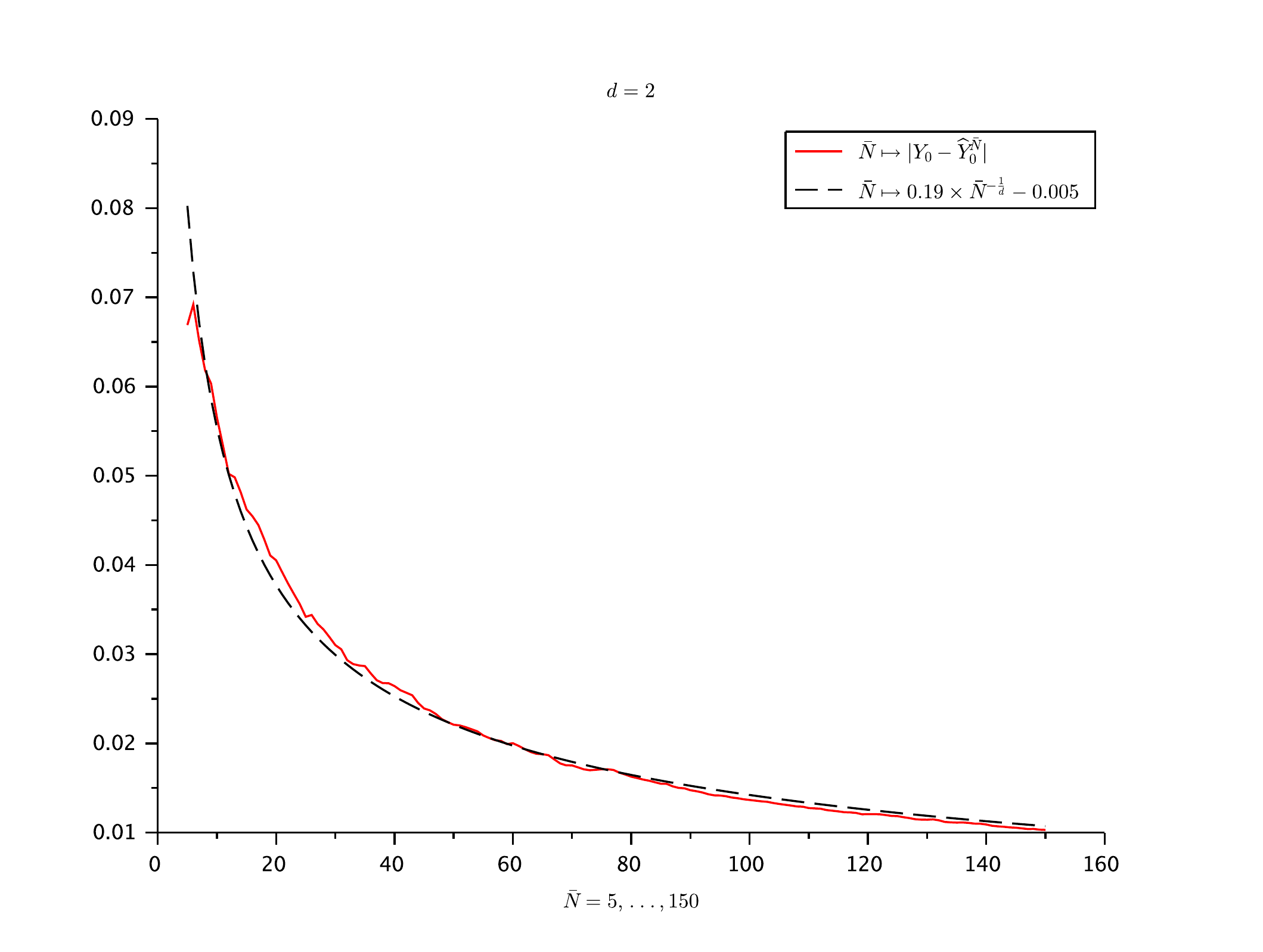}	
\hfill   \!\includegraphics[width=8.5cm,height=6.0cm]{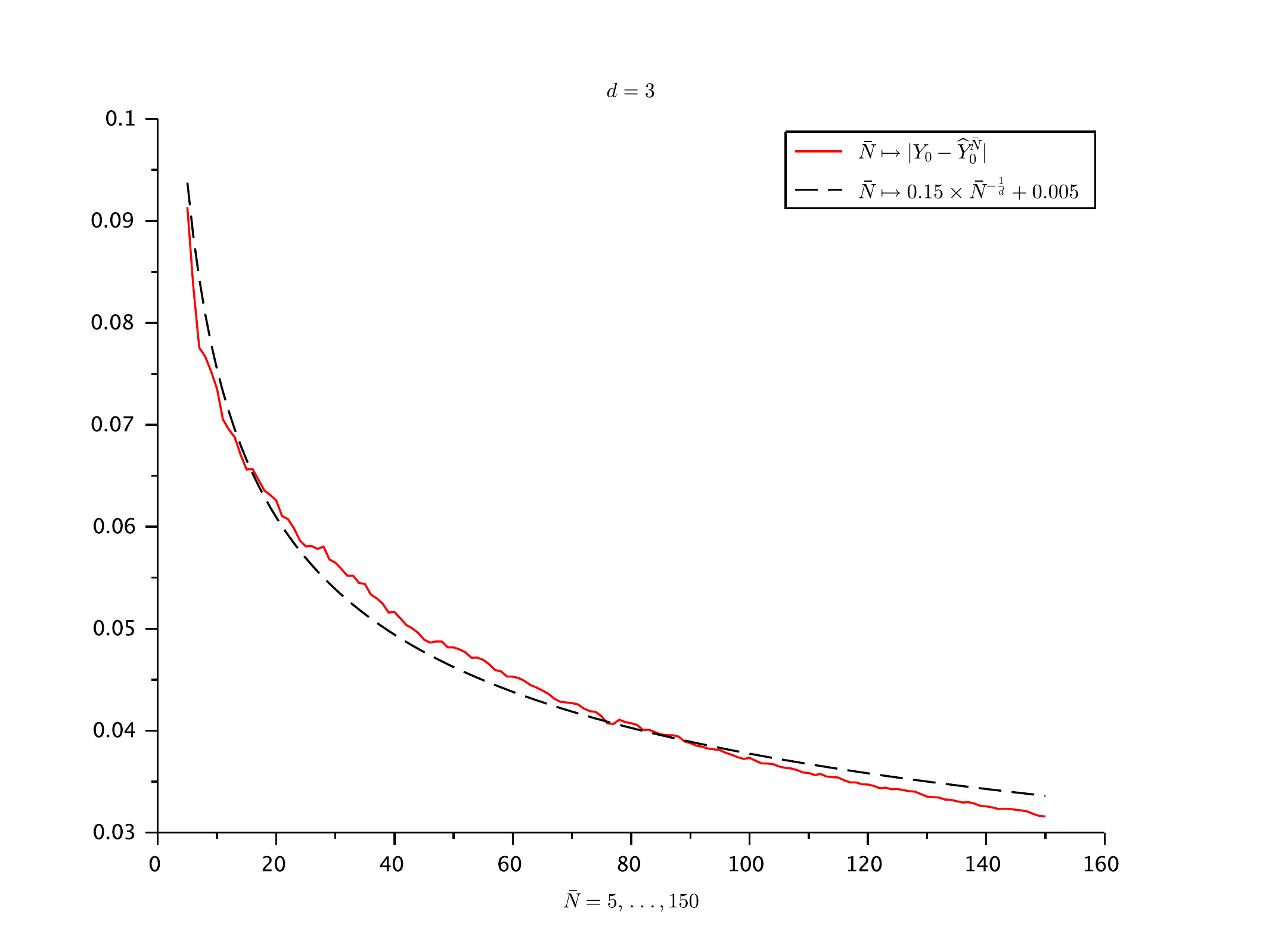}
  \caption{\footnotesize  Convergence rate of the quantization error for the multidimensional example.  Abscissa axis: the size $\bar N=5, \ldots, 150$ of the quantization.  Ordinate axis:  The error $\vert Y_0 - \hat Y_0^{\bar N} \vert$ and the graph $\bar N \mapsto \hat a  \bar{N}^{-1/d}  + \hat b$, where $\hat a$ and $\hat b$ are the regression coefficients. The left hand side graphic corresponds to  the dimension $d=2$ and  the right hand side  to  $d=3$.}
  \label{figCompRate0}
%  \end{center}
\end{figure}

\begin{figure}[htpb]
% \begin{center}
  \!\includegraphics[width=8.5cm,height=6.0cm]{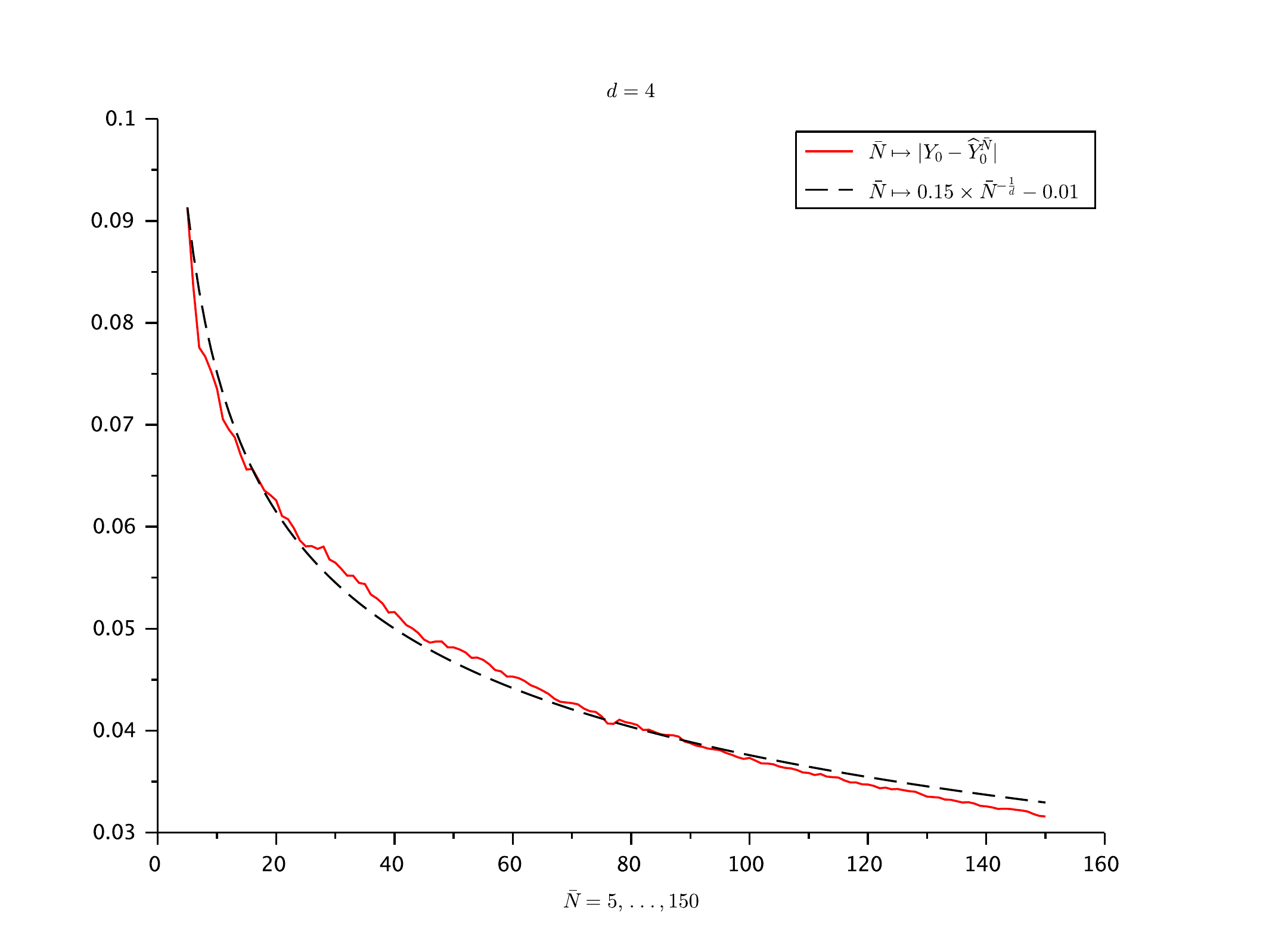}	
\hfill   \!\includegraphics[width=8.5cm,height=6.0cm]{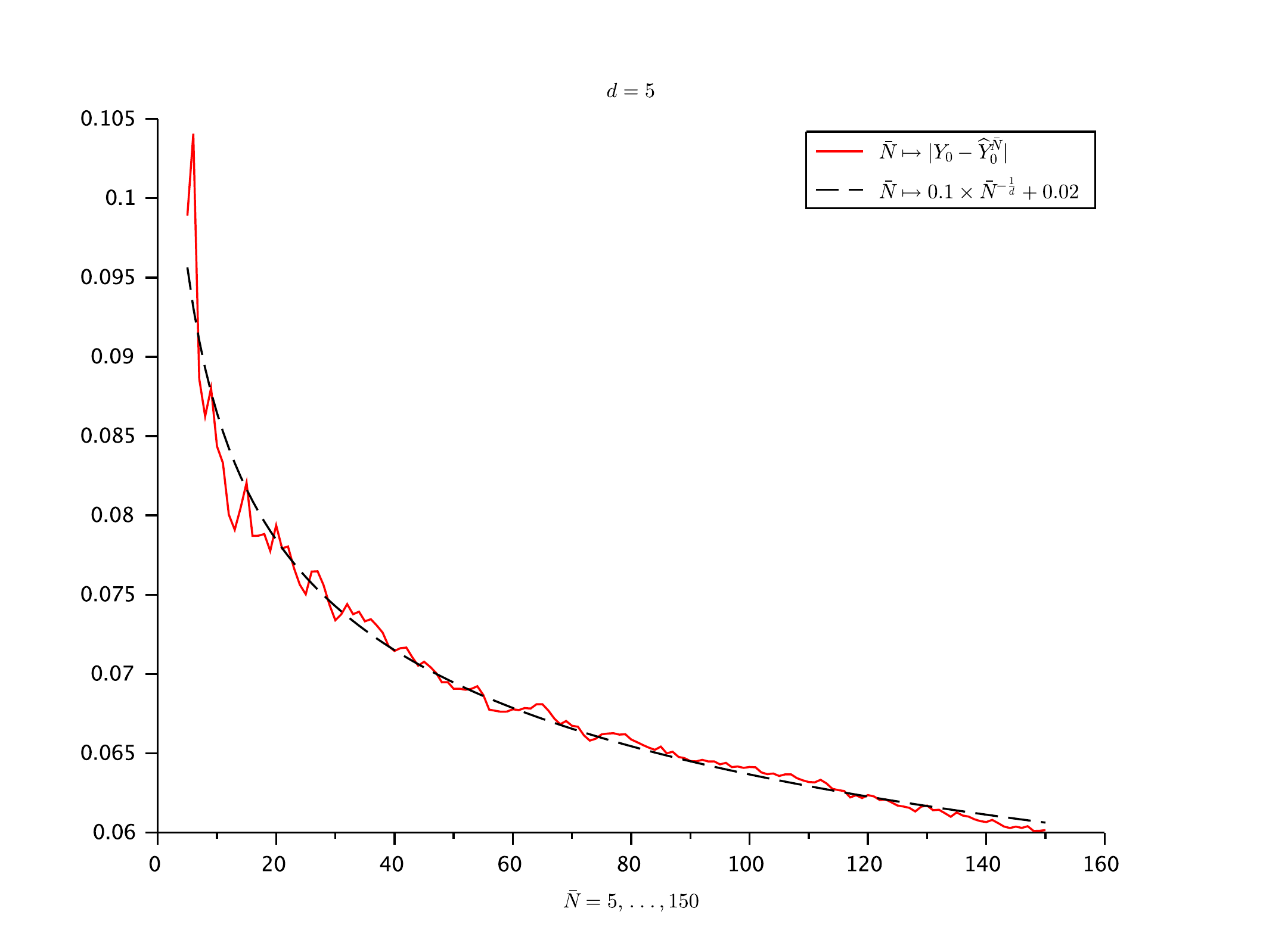}
  \caption{\footnotesize  Convergence rate of the quantization error for the multidimensional example.  Abscissa axis: the size $\bar N=5, \ldots, 150$ of the quantization.  Ordinate axis:  The error $\vert Y_0 - \hat Y_0^{\bar N} \vert$ and the graph ${\bar N} \mapsto \hat a  \bar{N}^{-1/d}  + \hat b$, where $\hat a$ and $\hat b$ are the regression coefficients. The left hand side graphic corresponds to  the dimension $d=4$ and  the right hand side  to  $d=5$.}
  \label{figCompRate1}
%  \end{center}
\end{figure}

%-----------------------------------------------------------------------------------------------------------------------------------------------------------------------------------
%-----------------------------------------------------------------------------------------------------------------------------------------------------------------------------------
%-----------------------------------------------------------------------------------------------------------------------------------------------------------------------------------
%---------------------------------------------------------------------------NONLINEAR FILTERING----------------------------------------------------------------------
%-----------------------------------------------------------------------------------------------------------------------------------------------------------------------------------
%-----------------------------------------------------------------------------------------------------------------------------------------------------------------------------------
%-----------------------------------------------------------------------------------------------------------------------------------------------------------------------------------

\section{Nonlinear filtering problem}  \label{SecNonLinearFilt}
We consider in this section the  discrete time nonlinear filtering model and the quantization based  numerical scheme presented in the introduction. Our aim is two-fold: improving the  error bounds like for BSDE on the one hand and, on the other hand,  relaxing the Lipschitz continuity on the conditional densities $g_k$ (in favor of a local Lipschitz continuity assumption with polynomial growth). In particular, these new error bounds confirm  the results obtained in the survey~\cite{Sel} devoted to a comparison between quantization and particle based numerical methods for non-linear filtering. 

\subsection{Error analysis}

Let us   first recall the   assumptions   made in~\cite{PagPha} on the conditional transition density  functions $g_k$ and the Markov transitions $P_k$:
\medskip 

 \noindent   $ (\mathscr H_{0}) \  $ For  every $k \in \{1,\ldots,n\}$ there exists  $ [g_k^1]_{\rm Lip},  [g_k^2]_{\rm Lip}: \mathds R^q {\small \times} \mathds R^q  \mapsto \mathds R_{+}$ such that
   \begin{eqnarray*}
    \vert  g_k(x,y,x',y') -  g_k(\hat x,y, \hat x',y')\vert &  \leq &  [g_{k}^1]_{\rm Lip}(y,y')   \vert x - \hat x \vert   +    [g_{k}^2]_{\rm Lip}(y,y')   \vert x' - \hat x' \vert.
 \end{eqnarray*}

\noindent   $ (\mathscr A_{1}) \  $ $(i)$ The Markov transition operators $P_k(x,dx'), k=1, \ldots, n$  propagate Lipschitz continuity  (in the sense of Lemma~\ref{LemProofProTheo}) and  
$$ 
[P]_{\rm Lip}:= \max_{k=1,\ldots,n}  [P_k]_{\rm Lip} <+\infty. 
$$
\hskip 1.cm $(ii)$  For every $k=1, \ldots, n$, the functions $g_k$ are bounded on $\mathds R^d {\small \times} \mathds R^q {\small \times} \mathds R^d {\small \times} \mathds R^q $ and we set 

$$ 
K_g:= \max_{k=1, \ldots,n} \Vert  g_k  \Vert_{\infty}<+\infty.
$$
We will relax these Lipschitz assumptions into controlled Lipschitz assumptions. Let us consider, for a fixed {\em non-negative} function $\theta: \mathds R^{d} \mapsto \mathds R_{+}$ satisfying,

 \noindent $ ({\mathscr I}_{\theta})\ $ \hskip 2.0 cm  $\forall\, k \in \{1,\ldots,n\}$,    $\quad\mathds E(\theta(X_k)) < + \infty$.

 \medskip
 We make the following $\theta$-local Lipschitz continuity  assumption (which is weaker than   $ (\mathscr H_{0}) $) on the growth of the conditional transition density  functions $g_k$:    

 \medskip
  \noindent   $ (\mathscr H_{\rm Liploc}) \ $ There exists   $ [g_k^1]_{\rm Liploc},  [g_k^2]_{\rm Liploc}: \mathds R^q {\small \times} \mathds R^q  \mapsto \mathds R_{+}$ such that, for  every $k \in \{1,\ldots,n\}$, 
   \begin{eqnarray*}
    \vert  g_k(x,y,x',y') -  g_k(\hat x,y, \hat x',y')\vert &  \leq &  [g_{k}^1]_{\rm Liploc}(y,y') (1+\theta( x ) + \theta(\hat x) +  \theta(x' ) + \theta( \hat x' ) ) \vert x - \hat x \vert  \\
 &  &+   [g_{k}^2]_{\rm Liploc}(y,y') (1+\theta( x ) + \theta(\hat x) +  \theta(x' ) + \theta( \hat x' ) )  \vert x' - \hat x' \vert.
 \end{eqnarray*}

\smallskip 
A standard situation  is the sometimes called ${\rm Li}(1,\alpha)$ framework when the $g_k$ satisfy $ (\mathscr H_{\rm Liploc})$  with  the   function $\theta: x \mapsto \theta(x) = \vert  x \vert^{\alpha}$ for an $\alpha\ge 0$, namely
 
 \noindent   $ (\mathscr H_{\alpha})  \ $ For  every $k \in \{1,\ldots,n\}$ there exists  $ [g_k^1]_{\rm pol},  [g_k^2]_{\rm pol}: \mathds R^q {\small \times} \mathds R^q  \mapsto \mathds R_{+}$ such that
   \begin{eqnarray*}
    \vert  g_k(x,y,x',y') -  g_k(\hat x,y, \hat x',y')\vert &  \leq &  [g_{k}^1]_{\rm pol}(y,y') (1+\vert x \vert^{\alpha} + \vert  \hat x \vert^{\alpha} + \vert x' \vert^{\alpha} + \vert \hat x' \vert^{\alpha}) \vert x - \hat x \vert  \\
 &  & +  [g_{k}^2]_{\rm pol}(y,y') (1+\vert x \vert^{\alpha} + \vert  \hat x \vert^{\alpha} + \vert x' \vert^{\alpha} + \vert \hat x' \vert^{\alpha}) \vert x' - \hat x' \vert.
 \end{eqnarray*}
%for $\alpha \ge  1$, then $(\mathscr H_{\rm Liploc})$ holds  . 

When $\theta\equiv 0$, this framework coincides with the Lipschitz one.  To simplify some statement we will introduce
\begin{equation}\label{eq:gLiplocGlob}
[g^i]_{\rm Liploc}(y,y') = \max _{k=1,\ldots,n} [g_{k}^i]_{\rm Liploc}(y,y'),\quad i=1,2.
\end{equation}

\noindent {\bf Example.} We may consider a model in which  the  signal process $X$ is a discrete time real valued (Markov) process and that the observation process evolves as
\[
Y_k =  Y_{k-1}+\varphi( X_{k-1})  + \sigma\varepsilon_{k}, \;\sigma>0,\quad \varepsilon_{k} \stackrel{i.i.d}{ \sim } {\cal N}(0,1), \quad k=1, \ldots, n.
\]
where the function $\varphi$ is  bounded, locally Lipschitz continuous but possibly not Lipschitz. In this case, the conditional density functions $g_k$ do not depend on $k$ and read  (where $\Phi_0$  denotes the c.d.f.  of the ${\cal N}(0;1)$) 
\[
g_k(x,y,x',y')   = \Phi_0 \Big(\frac{ y -y'- \varphi(x)}{\sigma}  \Big).
\]
Thus, if $\varphi(x)= \sin (x^3)$, then $g_k$  will be of ${\rm Li}(1,2)$-type. This situation extend to any locally Lipschitz bounded  function $\varphi $ whose derivative has a ``$\theta$-growth" at infinity.

Then we ask the transitions $P_k(x,dy)$   to propagate this $\theta$-local Lipschitz property as a counterpart of $({\mathscr A}_1)$.  Let $f:\mathds R^d \rightarrow \mathds R$ be $\theta$-locally Lipschitz with a local Lipschitz coefficient $[f]_{\rm Liploc}$ defined by
% :  there exists  a non-negative  convex function  $\theta: \mathds R^{d} \mapsto \mathds R_{+}$ such that 
 \begin{equation}
[f]_{\rm Liploc}=\sup_{x\neq x'}\frac{\vert f(x) - f(x') \vert}{ \big(1 + \theta(x) + \theta(x') \big) \, \vert x - x' \vert}<+\infty.
%  \quad x,x' \in \mathds R^d,
\end{equation}

\[
\hskip -1.5cm ({\mathscr A}_{1,\rm  loc}) \quad  [P]_{\rm Liploc} = \max_{k=1,\ldots,n}[P_k]_{\rm Liploc} <+\infty \;\mbox{ where }\; [P_k]_{\rm Liploc}=Ê\sup_{[f]_{\rm Liploc}\le 1} [P_kf]_{\rm Liploc}.
\]  

\begin{rem}
One easily checks  that the transition kernels of the Euler scheme with  step $\frac Tn$ (and Brownian increments) of a diffusion with Lipschitz continuous  drift $b$ and diffusion coefficient  $\sigma$ have the $\theta$-local Lipschitz property when $\theta_{\alpha}(x) = |x|^{\alpha}$, $\alpha>0$.
\end{rem}

The following classical lemma is borrowed (and straightforwardly  adapted) from~\cite{PagPha} (Lemma 3.1). 
\begin{lem}  \label{LemNLFilter} Let  $\mu_y$ and $\vartheta_y$ be two families of finite and positive measures on a measurable space $(E,\mathcal E)$. Suppose that there exist two symmetric functions $R$ and $S$ defined on the set of positive finite measures  such that, for every bounded $\theta$-Lipschitz function $f$,
\begin{equation}
\left \vert  \int   f d\mu_y  - \int f d\vartheta_y    \right\vert ^2  \leq \Vert f \Vert^2_{\infty}  R(\mu_y,\vartheta_y)  + [f]^2_{\rm Liploc} S(\mu_y,\vartheta_y). 
\end{equation}
Then,
 \begin{equation}
\left \vert  \frac{\int   f d\mu_y}{\mu_y(E)}  - \frac{ \int f d\vartheta_y  }{\vartheta_y(E)}  \right \vert   \leq  4 \frac{\Vert f \Vert^2_{\infty}  R(\mu_y,\vartheta_y)  +  \frac 12 [f]_{\rm Liploc}^2 S(\mu_y,\vartheta_y)}{\mu_y(E)\vee \vartheta_y(E)}  . 
\end{equation}
\end{lem}

In Theorem~\ref{ThmLocalLips} below we will consider  Assumption  $(\mathscr H_{\rm Liploc})$ in place of Assumption $(\mathscr H_{0})$ (considered in~\cite{PagPha})  to derive an error bound.   This less stringent assumption is compensated  by taking advantage of the   {\em distortion mismatch}  property established in Theorem~\ref{thm:Mismatchnew}.  More precisely,  we  will  need that the  $L^{s}$-mean quantization error, for an $s\!\in (2,2+d)$,  associated to any sequence of optimal quadratic quantizers at level $N$   still goes to zero at the optimal rate $N^{-\frac 1d}$. 

%We first need to   control the $\theta$-local Lipschitz constants of $[u_{k}]_{\rm Liploc}$, for every $k \ge 0$ (where $\theta:\mathds R^d\to \mathds R_+$ is defined as above).  If we suppose that $f:\mathds R^d \rightarrow \mathds R$ is $\theta$-locally Lipschitz with a local Lipschitz coefficient $[f]_{\rm Liploc}$ defined by
%% :  there exists  a non-negative  convex function  $\theta: \mathds R^{d} \mapsto \mathds R_{+}$ such that 
% 
% \begin{equation}
%[f]_{\rm Liploc}=\sup_{x\neq x'}\frac{\vert f(x) - f(x') \vert}{ \big(1 + \theta(x) + \theta(x') \big) \, \vert x - x' \vert,}<+\infty
%%  \quad x,x' \in \mathds R^d,
%\end{equation}
%
%\noindent then for every $k=0, \ldots,n,$ $P_k f$ is  $\theta$-locally Lipschitz with $[P_kf]_{\rm Liploc} \le [P_k]_{\rm Liploc}[f]_{\rm Liploc}$.  

The following  lemma  provides  a control of the $\theta$-local Lipschitz coefficients of the functions $u_{y,k}(f)$ defined  recursively  by~\eqref{EqBackwardInduction}. 

Note that we drop the subscript  related to the observations $y$ in the conditional densities $g_k$ as well as the function $f$ in $u_{y,k}(f)$ to alleviate notations in what follows.
\begin{prop}  \label{LemAssMainTheorem}
$(a)$ Assume that  $({\mathscr I}_{\theta})$, $ (\mathscr H_{\rm Liploc})$ and $({\mathscr A}_{1,\rm  loc})$ hold and that, for every $k =1, \ldots, n$, 
\begin{equation}  \label{EqAssumConvexgrowth}
\mathds E \big(\theta(X_{k}) \vert X_{k-1}=x \big) \le C_{\theta,X} (1 + \theta(x)).
\end{equation}
Let $f : \mathds R^d\to \mathds R^d$ be $\theta$-locally Lipschitz function.  Then, the functions $u_k$ defined by~\eqref{EqBackwardInduction} satisfy  
%$[u_n]_{\rm Liploc}= [f]_{\rm Liploc}$ and, for every $k=0, \ldots,n-1$,
%\begin{equation} 
%[u_k]_{\rm Liploc}  \le \big(2 [g_{k+1}^1]_{\rm Liploc} C_{\theta,X}  + [P_{k+1}]_{\rm Liploc}  [g_{k+1}^2]_{\rm Liploc} \big) K_g^{n-k}\Vert f \Vert_{\infty} + K_g  [P_{k+1}]_{\rm Liploc} [u_{k+1}]_{\rm Liploc}.
%\end{equation}
\begin{eqnarray} 
%\nonumber [u_k]_{\rm Liploc} 
% &\le  & K^{n-k}_g \|f\|_{\infty}\sum_{\ell=k}^n   [P_{k+1}]^{\ell-k}_{\rm Liploc} \big( 2[g_{\ell+1}^1]_{\rm Liploc} C_{\theta,X}  
%+ [P_{\ell+1}]_{\rm Liploc}  [g_{\ell+1}^2]_{\rm Liploc} \big), \quad  k=0,\ldots,n-1,\\
\label{eq:Lesaell} [u_k]_{\rm Liploc} &\le& K^{n-k}_g \left[\kappa_{g,X}\frac{[P]_{\rm Liploc}^{n-k}-1}{[P]_{\rm Liploc}-1}\|f\|_{\infty}+[P]_{\rm Liploc}^{n-k}[f]_{\rm Liploc}\right]  \,\mbox{(Convention: $\frac{1^n-1}{1-1}=n$)}
\end{eqnarray}
where $\kappa_{g,X}=  2 C_{\theta,X} [g^1]_{\rm Liploc}  + [P]_{\rm Liploc}  [g^2]_{\rm Liploc}$ and $\|u\|_{\infty}\le (K_g)^n\|f\|_{\infty}$. 
% \begin{equation}\label{eq:Lesaell}
%\mbox{with }\quad a_k= 2 [g_{k+1}^1]_{\rm Liploc} C_{\theta,X}  + [P_{k+1}]_{\rm Liploc}  [g_{k+1}^2]_{\rm Liploc}, k=0,\ldots,n-1 \mbox{ and }  a_n = \frac{}{\|f\|_{\infty}}.
%\end{equation}

\noindent $(b)$   Let  $(X_k)_{k=0,\ldots,n}$ be  the Markov chain defined  as an iterated random map of the form
\begin{equation}\label{eq:randItermap}
X_{k} = F_{k}(X_{k-1},\varepsilon_{k}), \qquad k=1, \ldots,n
\end{equation}
where $(\varepsilon_k)_{k=1,\ldots,n}$ is an i.i.d sequence of  random variables independent of  $X_0$.

 $(i)$ If there exists some $p\!\in (0,+\infty)$ such that 
\begin{equation} \label{EqAssumSquareConvexgrowth}
\theta(X_0)\!\in L^p\quad\mbox{ and }\quad \big\|\theta(F_{k}(x, \varepsilon_{1}) )\big\|_p \le C_{\theta,X}' (1 + \theta(x))
\end{equation}
then $\max_{k=0,\ldots,n}\|\theta(X_k)\|_p<+\infty$. \\
In particular, for $p=1$, the chain satisfies the integrability assumption $({\mathscr I}_{\theta})$ and~\eqref{EqAssumConvexgrowth}.

 $(ii)$ If $F_k(0, \varepsilon_1), \theta(X_0) \!\in L^2$ and,  for every $k\!\in \{1,\ldots,n\}$, $x,\,x'\! \in \mathds R^d$,
\begin{equation} \label{EqFkLip2}
\big\|\theta(F_{k}(x, \varepsilon_{1}) )\big\|_2 \le C_{\theta,X}' (1 + \theta(x)) \quad\mbox{and}\quad 
 \| F_{k}(x,\varepsilon_{1})-F_{k}(x', \varepsilon_{1})\|_2 \le  [F_{k}]_{2,\rm Lip} |x-x'|,
   \end{equation} 
then both $({\mathscr I}_{\theta})$ and  $({\mathscr H}_{\rm Liploc})$ are satisfied. To be more precise
\begin{equation*}
\forall\, k=1, \ldots,n,\quad [P_{k}]_{\rm Liploc}  \le  \max(1,C_{\theta,X}')[F_{k}]_{2,\rm Lip}.
\end{equation*}
\end{prop}

\begin{rem} Once again the Euler scheme with step $\frac Tn$ satisfies Assumption~\eqref{EqAssumConvexgrowth} with functions $\theta_{\alpha} (x) = \vert x \vert^{\alpha}$, $\alpha>0$.
\end{rem}

\begin{proof} $(a)$ By the Markov property, we have for every $k=0, \ldots,n-1$ and every $x\!\in \mathds R^d$,  
\begin{equation}\label{eq:ukBDDP}
 u_k(x) = \mathds{E}\big(u_{k+1}(X_{k+1})g_{k+1}(X_k,X_{k+1})\big |X_k=x\big)  = (P_{k+1} u_{k+1} g_{k+1}(x, \cdot)) (x).
\end{equation}
It follows that, for every $k=0, \ldots,n-1$, $\Vert  u_{k} \Vert_{ \infty} \le K_g \Vert  u_{k+1} \Vert_{\infty}$, so that,
 \[
 \Vert u_k \Vert_{\infty}  \le  K_g^{n-k}\, \Vert  f \Vert_{ \infty}
 \]
 since $\Vert u_n \Vert_{\infty} = \Vert f \Vert_{\infty}$. Let $k\!\in \{0, \ldots,n-1\}$;    for every  $x,x' \in \mathds R^d$,
\begin{eqnarray*}
\vert u_k(x) - u_k(x') \vert   &\le&   [g_{k+1}^1]_{\rm Liploc}  \Vert u_{k+1} \Vert_{\infty} \big(1+ \theta (x) + \theta(x')   + \mathds E(\theta(X_{k+1})\vert X_k=x)\big) \vert x-x' \vert \\
&   & + \, [P_{k+1}]_{\rm Liploc}[u_{k+1}g_{k+1}(x',\cdot)]_{\rm Liploc} \big(1+\theta(x) + \theta(x') \big) \vert x -x'\vert.
\end{eqnarray*} 
Now, still for every $k=0, \ldots,n-1$,
\begin{eqnarray*}
\vert u_{k+1}(z) g_{k+1}(x',z) - u_{k+1}(z') g_{k+1}(x',z') \vert   &\le&  \vert u_{k+1}(z) -u_{k+1}(z') \vert  g_{k+1}(x',z) \\
& & + \, \vert g_{k+1}(x',z) -g_{k+1}(x',z')\vert \, \vert u_{k+1}(z') \vert \\
& \le  &  K_g[u_{k+1}]_{\rm Liploc} (1 +\theta(z) +\theta(z')) \vert z-z' \vert  \\
& &   +\, \Vert u_{k+1} \Vert_{\infty}[g_{k+1}^2]_{\rm Liploc} (1 +\theta(z) +\theta(z')) \vert z-z' \vert
\end{eqnarray*} 
so that
\[
[u_{k+1}g_{k+1}(x',\cdot)]_{\rm Liploc}  \le K_g[u_{k+1}]_{\rm Liploc} + \Vert u_{k+1} \Vert_{\infty}[g_{k+1}^2]_{\rm Liploc}.
\]
Finally, collecting these inequalities, we deduce from Assumption~\eqref{EqAssumConvexgrowth} that, for every $k=0, \ldots,n-1$,
\begin{eqnarray*}
[u_k]_{\rm Liploc} & \le & \big( 2 C_{\theta,X} [g_{k+1}^1]_{\rm Liploc} + [P_{k+1}]_{\rm Liploc}  [g_{k+1}^2]_{\rm Liploc} \big) \Vert u_{k+1} \Vert_{\infty} + K_g  [P_{k+1}]_{\rm Liploc} [u_{k+1}]_{\rm Liploc}\\
&\le& \kappa_{g,X}\Vert u_{k+1} \Vert_{\infty} + K_g  [P]_{\rm Liploc} [u_{k+1}]_{\rm Liploc}.
\end{eqnarray*}
The conclusion follows by a backward   induction (discrete time Gronwall's Lemma) having in mind that $u_n = f$.

\smallskip
\noindent $(b)$ Claim $(i)$ is obvious. As for claim~$(ii)$, let  $f$ be a  $\theta$-locally Lipschitz with  constant $[f]_{\rm Liploc}$. Then, for every $x, x' \!\in \mathds R^d$ and every $k=1, \ldots,n$,
\begin{eqnarray*}
\vert P_{k} f(x) & - &P_{k} f(x') \vert =  \vert  \mathds E f(F_{k}(x, \varepsilon_{1})) - \mathds E f(F_{k}(x', \varepsilon_{1})) \vert \\
&\le& [f]_{\rm Liploc} \mathds E\Big( \big|  F_{k}(x, \varepsilon_{1}) - \mathds E F_{k}(x', \varepsilon_{1}) \big|\big(1+\theta(F_{k}(x, \varepsilon_{1}) )+\theta(F_{k}(x', \varepsilon_{1}) )\big)\Big)\\
&\le &[f]_{\rm Liploc} \big\| F_{k}(x, \varepsilon_{1}) - F_{k}(x', \varepsilon_{1}) \big\|_2\big(1+ \|\theta(F_{k}(x, \varepsilon_{1}) )\|_2+\|\theta(F_{k}(x, \varepsilon_{1}) )\|_2\big)\\
&\le&  [f]_{\rm Liploc} [F_{k}]_{2,\rm Lip} \vert x-x' \vert \big(1+ C'_{\theta,X} \theta(x) + C'_{\theta,X} \theta(x') \big)
\end{eqnarray*}
where we used Schwarz's Inequality in the third line, and~\eqref{EqFkLip2} and~\eqref{EqAssumSquareConvexgrowth} in the last line. We  deduce that  $[P_{k}]_{\rm Liploc} \le  (1\vee C_{\theta,X}')$, for every $k=1, \ldots,n$.
\end{proof}

Notice that assumptions~\eqref{EqAssumConvexgrowth} and~\eqref{EqAssumSquareConvexgrowth} hold when $\theta$ is a polynomial convex function  and when  $(X_k)_{ 0 \le k \le n}$ is the Euler scheme   (with step $\frac Tn$ and horizon $T$) associated with a  stochastic differential equation of the form~\eqref{eq:diffusion}.
 \begin{thm}   \label{ThmLocalLips}
 Let   $(\mathscr H_{\rm Liploc})$ hold   and assume that    $(\mathscr A_{1,loc})$   is fulfilled, as well as assumptions of Proposition~\ref{LemAssMainTheorem}.    Suppose that for every $k=0,\ldots,n$,   $X_k$ has an $(2, 2+\nu_k)$-distribution (in the sense of Definition~\ref{def:(r,s)}) for some  $\nu_k \in (0,d)$, and set $\bar {\nu}_n = \min_{k=0, \ldots,n} \nu_k/2$.   Then  for every $\nu \in (0,  \bar{\nu}_n)$, 
\begin{equation}  \label{EqUpperbound22PlusNu}
\vert \Pi_{y,n}f -\hat{\Pi}_{y,n} f \vert ^2 \leq  \frac{4 ( M_{n,\nu} K_g^{n})^2}{\phi_n^2(y) \vee \hat{\phi}_n^2(y)} \sum_{k=0}^n  B_k^n(f,y) \Vert X_k - \hat X_k \Vert_{2(1+\nu)}^2
\end{equation}
where  
$$ 
\phi_{n}(y) = \pi_{y,n} \mbox{\bf 1} \quad \textrm{and } \  \hat {\phi}_{n}(y) = \hat{\pi}_{y,n} \mbox{\bf 1}
$$
and
%  \textcolor{red}{A actualiser avec les nouveau calculs}
\begin{equation*}  
 B_k^n(f,y) : =  
 2^{k}\left(2\kappa_{g,X}^2\left(\frac{[P]_{\rm Liploc}^{n-k}-1}{[P]_{\rm Liploc} -1}\right)^2+   \frac{2[g^1]_{\rm Liploc}^2 +   [g^2]_{\rm Liploc}^2}{K_g^2}  +[P]_{\rm Liploc}^{2(n-k)}\right)     
 %2 [P]_{\rm Liploc}^{2(n-k)}[f]_{\rm Liploc}^2  +   2  \Vert f \Vert_{\infty}^2 R_{n,k} +   \Vert f \Vert_{\infty} R_{n,k}^2, 
 \end{equation*}
 with $\kappa_{g,X}=  2 C_{\theta,X} [g^1]_{\rm Liploc}  + [P]_{\rm Liploc}  [g^2]_{\rm Liploc}$ and 
%  $$ 
%  R_{n,k} =
%  %  \frac{8^{q_{\nu}} }{K_g^2}  \Big[[g_{k+1}^1]_{\rm Liploc}^2 + [g_{k}^2]_{\rm Liploc}^2    +      \Big(\sum_{m=1}^{n-k}[P]_{\rm Liploc}^{m-1}([g_{k+m}^1]_{\rm Liploc} + [P]_{\rm Liploc}[g_{k+m}^2]_{\rm Lip}) \Big)^2  \Big], 
%  $$
%$q_{\nu} = 1+1/ \nu$ and 
$$ 
 M_{n,\nu}: =   1+  \max_{k=0, \ldots,n-1}\Big(\|\theta(X_k)\|_{ 2(1+\frac{1}{\nu})}    +  \|\theta(\hat X_k)\|_{2(1+\frac{1}{\nu})}   +   \|\theta(X_{k+1}) \|_{2(1+\frac{1}{\nu})}   +    \| \theta( \hat X_{k+1})\|_{2(1+\frac{1}{\nu})}\Big).
%( \mathds E \big( \theta(X_k )^{2q_{\nu}}\big)   +   \mathds E \big( \theta( \hat X_{k})^{ 2q_{\nu}} \big)  + \mathds E \big( \theta(X_{k+1} )^{2q_{\nu}}\big)   +   \mathds E \big( \theta( \hat X_{k+1})^{ 2 q_{\nu}}  \big). 
$$
 \end{thm}
 
 Let us make few remarks about the assumptions of the theorem before dealing with the proof. 
 \begin{rem}
$(a)$ If  $\theta$ is convex  and   if all $\hat X_k$ are quadratic  optimal quantizers,  then it is   stationary  $i.e.$ satisfies $\hat X_k = \mathds E(X_k\,|\, \hat X_k)$ so that,  for every $k=0, \ldots,n$,  we have, owing to the convexity of  $\theta^{2(1+\frac{1}{\nu})}$ and Jensen's Inequality, 
 $$ 
 \|\theta ( \hat X_k ) \|_{2(1+\frac{1}{\nu})}   \leq \mathds \| \theta( X_k )\|_{2(1+\frac{1}{\nu})} <+\infty.
 $$
 
 \noindent $(b)$ Suppose that  $(X_k)_{k=0,\ldots,n}$ is   a Markov chain  of iterated random maps
 %~\eqref{eq:randItermap}
  \[
 X_{k} = F_{k}(X_{k-1}, \varepsilon_{k}), \quad k=1, \ldots,n,
 \]
 under the assumptions of Proposition~\ref{LemAssMainTheorem}. Assume it satisfies the $\theta$-local Lipschitz assumption with a function $\theta(y)\ge  C|y|^a$ for some real constants $C, a>0$. If $a>\frac 12 \frac{  (2+\nu) d}{d- \nu}$ for some $\nu\!\in (0,d)$, and  the distributions of  $X_k$  are absolutely continuous, then, all $X_k$ have an $(2,2+\nu)$-distribution.
%
% and  satisfying~\eqref{}\dots and the control function $\theta: \mathds R^d\to \mathds R_+$. Assume furthermore that there exists $p\!\in (1,+\infty)$,  such that $\theta(X_0)\!\in L^p$ and 
% %where $(\varepsilon_k)$ is an i.i.d sequence of random variables, $X_0$, and, 
%  for every $x \in \mathds R^d$,
%\begin{equation}  \label{EqAssumSquareConvexgrowthRm}
%\|\theta(F_{k}(x, \varepsilon_{1}) \|_p \le C_{\theta,X}' (1 + \theta(x)), \quad  k=1, \ldots,n
%\end{equation}
%for a positive real valued and convex function $\theta$. Then it is straightforward that $\theta(X_k)\!\in L^p$ for every $k\!\in \{0,\ldots,n\}$.
%
% If  $\mathbb E \vert X_0 \vert^a<+\infty$, then for every $k \ge 1$ and for any $\nu \in (0,d)$,  $X_k$  has an $(2,2+\nu)$-distribution, so that $\bar{\nu}_n  = d/2$. In fact, for  any $a\ge 1$, the function  $x \mapsto \vert x \vert^a$ is convex and  we have
%\[
%\mathbb E \vert  X_{k+1} \vert^a  = \mathbb E \vert  \mathbb E (F_k(X_k,\varepsilon_{k+1}) \vert X_k)\vert^a \le  \mathbb E \big(\mathbb E  \vert F_k(X_k,\varepsilon_{k+1})\vert^a \vert X_k) \big) \le  C_{\theta,X}' (1 + \mathbb E \vert X_k \vert^{a}).
%\]
%We deduce  by induction that $\mathbb E \vert  X_{k} \vert^a <+\infty$, for any $a>1$ and $k=0, \ldots, n$, as soon as $\mathbb E \vert X_0 \vert^a <+\infty$. This shows that   for any  $\nu_k \in (0,d)$, the moment condition \eqref{EqLrLsProblem} is satisfied with $r=2$, $s =2+\nu_k$ and $a = \frac{(2+\nu_k)d}{d-\nu_k}+\delta>1$.
 \end{rem}

 \begin{proof} 
 % To alleviate  the notations, we will omit  the dependence of the used functions on the observation parameter $y$. 
  Like in~\cite{PagPha}, the proof relies on the backward formulas~\eqref{EqBackwardInduction} and~\eqref{EqBackwardInductionhat} involving the functions $u_{y,k}(f)$ and their quantized counterpart $\hat u_{y,k}(f)$ whose final values $u_{-1}$ and $\hat u_{-1}$ define the un-normalized filter $\pi_{y,n}(f)$  (applied to the function $f$) and its quantized counterpart, respectively. 
%For every $k\!\in \{0,\ldots,n-1\}$, we define the function $\varphi_{k+1}$ by  

%\smallskip
%\noindent {\em Step~1}:  
Following the lines of the proof  of Theorem~3.1 in~\cite{PagPha}, one shows by a backward induction taking advantage of the Markov property that the functions $u_k:\mathbb R^d\to \mathbb R$, $k=0,\ldots, n$, defined recursively by~\eqref{EqBackwardInduction} satisfy $u_n = f$ and 
 \begin{equation}\label{eq:ukfilter}
 u_k(X_k) = \mathds{E}_k \big(\varphi_{k+1}(X_{k}, X_{k+1})\big)
= \mathds{E}  \big(\varphi_{k+1}(X_{k}, X_{k+1})\,|\, X_k\big),\, k=0,\ldots,n-1
 \end{equation}
% \begin{equation}\label{eq:ukfilter}
% u_k(X_k) = \mathds{E}_k \big(g_{k+1}(X_{k}, X_{k+1})u_{k+1}(X_{k+1})\big)
%= \mathds{E}  \big(g_{k+1}(X_{k}, X_{k+1})u_{k+1}(X_{k+1})\,|\, X_k\big),\, k=0,\ldots,n-1
% \end{equation}
 where
  $$
 \varphi_{k+1}(x_k,x_{k+1}):= g_{k+1}(x_k,x_{k+1}) u_{k+1}(x_{k+1}),\; x_k,\, x_{k+1}\!\in \mathbb R^d.
 $$
Finally $u_{-1}= \mathds E(u_0(X_0))= \pi_{y,n}(f)$ (un-normalized filter applied to $f$).  One shows likewise that the functions $\hat u_k$ defined by~\eqref{EqBackwardInductionhat} satisfy $\hat u_n =f $ (on the grid $\Gamma_n$) and  
 \begin{equation}\label{eq:ukfilterhat}
 \hat u_k(\hat X_k) = \hat{\mathds{E}}_k \big(\hat u_{k+1}(\hat X_{k+1}) g_{k+1} (\hat X_{k},\hat X_{k+1})\big),\, k=0,\ldots,n-1,
 \end{equation}
so that finally $\hat u_{-1}(f) = \E\, \hat u_0(\hat X_0)) = \hat \pi_{y,n}(f)$ (quantized un-normalized filter). One shows like for the functions $u_k$ in Proposition~\ref{LemAssMainTheorem} that $\|\hat u_k\|_{\infty}\le K_g^{n-k}\|f\|_{\infty}$. 
Now, using the definition of conditional expectation $\hat{\mathds E}_k$ as an orthogonal projection (hence an $L^2$-contraction as well),   we have 
 \begin{eqnarray}  
\nonumber  \Vert  u_k(X_k) - \hat{u}_k(\hat X_k) \Vert_2^2 & = &   \Vert  u_k(X_k)  -  \hat{\mathds E}_k(u_k( X_k) )\Vert_2^2+\| \hat{\mathds E}_k(u_k( X_k) )-\hat{u}_k(\hat X_k) \Vert_2^2  \\
% &  & +  \Vert  \hat{\mathds E}_k(  u_k( X_k) )   -  \hat{\mathds E}_k(\varphi_{k+1}(\hat X_k, \hat X_{k+1},   \hat X_{k+1}) ) \Vert_2^2 \\
 \nonumber  & \leq &  \Vert  u_k(X_k)  -  \hat{\mathds E}_k(u_k( X_k) )\Vert_2^2\\
 \label{eq:IneqFiltrageQ}&&+ \Vert   \varphi_{k+1}( X_k,   X_{k+1})-  \hat u_{k+1}(\hat X_{k+1}) g_{k+1} (\hat X_{k},\hat X_{k+1}) \Vert_2^2
 %  \Vert  \mathds E_k( u_k(X_k) -  \hat{\mathds E}_k(\varphi_{k+1}(\hat X_k, \hat X_{k+1},  X_{k+1}) )\Vert_2^2 \\
% &&   + K_g^2  \Vert  u_{k+1} (X_{k+1}) - \hat{u}_{k+1}(\hat X_{k+1})   \Vert_2^2.
   \end{eqnarray}
where we used in the second line the tower property  for conditional expectation to show that  $\hat{\mathds E}_k(u_k( X_k) ) =  \hat{\mathds E}_k(  \varphi_{k+1}( X_k,   X_{k+1}))$ (the $\sigma$-field $\sigma(\widehat X_k)\subset {\cal F}_k$) and the contraction property of $ \hat{\mathds E}_k$.

 It follows now from the definition of the conditional expectation $\hat{\mathds E}_k(\cdot)$ as the best approximation in $L^2$  among   square integrable $\sigma(\hat X_k)$-measurable random vectors that
$$ \Vert  u_k(X_k) - \hat{\mathds E}_k (u_k (X_k)) \Vert_2^2    \leq    \Vert  u_k(X_k) -  u_k (\hat X_k) \Vert_2^2 \leq [u_k]_{\rm Liploc}^2  \Vert  (1 + \theta(X_k) + \theta(\hat X_k)) (X_k - \hat X_k) \Vert_2^2. 
 $$ 
    Let $\nu \!\in (0, \bar{\nu}_n)$, so that for every $k=0, \ldots,n$,  $2(1+\nu)  \le 2 +\nu_k$. H\"{o}lder's inequality with conjugate exponents $p_{\nu}=1+\nu$ and $q_{\nu} = 1 + \frac{1}{\nu}$ gives
$$ \Vert  u_k(X_k) - \hat{\mathds E}_k (u_k (X_k)) \Vert_2^2  \le [u_k]_{\rm Liploc}^2 \big \|1 +   \theta(X_k)  + \theta(\hat X_k) \big\|^2_{2 q_{\nu}}  \Vert X_k - \hat X_k \Vert_{2(1+\nu)}^2. 
$$

Let us deal now with the second  term on the right hand side of~\eqref{eq:IneqFiltrageQ} and set for convenience 
\[
\Delta_k:=   \varphi_{k+1}( X_k,   X_{k+1})-  \hat u_{k+1}(\hat X_{k+1}) g_{k+1} (\hat X_{k},\hat X_{k+1}).
 \]
 By the triangle  inequality and the boundedness of $g_{k+1}$, we get
  \begin{eqnarray*}  
|\Delta_k|&\le& \big|(u_{k+1}(X_{k+1})- \hat u_{k+1}(\hat X_{k+1})) g_{k+1}(X_k,X_{k+1}) \big| \\
 &&+ \big|  \hat u_{k+1}(\hat X_{k+1})) (g_{k+1}(X_k,X_{k+1})-g_{k+1} (\hat X_{k},\hat X_{k+1})) \big|\\
 &\le& K_g \big|u_{k+1}(X_{k+1})-\hat u_{k+1}(\hat X_{k+1}) \big|   + \|\hat u_{k+1}\|_{\infty}\big|g_{k+1}(X_k,X_{k+1})-g_{k+1} (\hat X_{k},\hat X_{k+1}) \big|
   \end{eqnarray*}
 so that 
\[
 \Vert \Delta_k \Vert_2^2 \le 2 K_g^2   \Vert  u_{k+1}(X_{k+1})-\hat u_{k+1}(\hat X_{k+1})   \Vert _2^2+ 2 \|\hat u_{k+1}\|^2_{\infty}\Vert g_{k+1}(X_k,X_{k+1})-g_{k+1} (\hat X_{k},\hat X_{k+1}) \Vert_2^2.
\]
 % \begin{equation}   \label{EqExistenceDens}
%\mathbb P_ {  X_{k}}(  dx_{k})   =  dx_{k}  \int_{\mathbb R^d}  \ldots \int_{\mathbb R^d}  \prod_{i=0}^{k-1} \Phi_{m_{i}(x_i),\Sigma_{i}(x_i)} (x_{i+1})    \mu(dx_0) dx_1 \ldots dx_{k-1}, 
%\end{equation}
% where the functions $m_k$ and $\Sigma_k$ are defined by   
% $$
% m_k(x) = x + \Delta_n b(t_k, x)     \qquad  \textrm{and } \quad \Sigma_k(x) = \Delta_n \, \sigma(t_k, x) \sigma(t_k,x)^T
% $$
% and  the function     $\Phi_{m_k(x),\Sigma_{k}(x)}$denotes the probability density of the distribution $ \mathcal N \big(m_k(x);\Sigma_k(x) \big)$.  Then, using the Jensen inequality,  it follows from~(\ref{asscor1}) that $X_k$ has a   $(2,2+ \varepsilon)$-distribution for every $\varepsilon \in (0,d)$.
 %
%Let us temporarily denote $R_k =  \hat{\mathds E}_k (u_k(X_k))     - \hat{\mathds E}_k (\varphi_{k+1}(\hat X_k, \hat X_{k+1}, X_{k+1}))$. Noting that $\hat{\mathds E}_k (u_k(X_k)) = \hat{\mathds E}_k (\varphi_{k+1}( X_k,   X_{k+1}, X_{k+1}))$ since, G
It follows  from   $(\mathscr H_{\rm Liploc})$,     H\"{o}lder  (still  with $p_{\nu}$ and $q_{\nu}$) and Minkowski inequalities  that
  \begin{eqnarray*}\label{EqControlRk} 
\Vert g_{k+1}(X_k,X_{k+1})&-&g_{k+1} (\hat X_{k},\hat X_{k+1}) \Vert_2^2\\
  &\leq &    [g_{k+1}^2]^2_{\rm Liploc}   \mathds E \Big[ \big(1+ \theta(X_k)   + \theta(\hat X_k)   +  \theta(X_{k+1} )   +   \theta( \hat X_{k+1}) \big) ^2 \vert  X_{k+1} - \hat X_{k+1} \vert^2 \Big]\\
 &&+ [g_{k+1}^1]^2_{\rm Liploc}  \mathds E \Big[ \big(1+ \theta(X_k)   + \theta(\hat X_k)   +  \theta(X_{k+1} )   +   \theta( \hat X_{k+1}) \big) ^2 \vert  X_{k} - \hat X_{k} \vert^2 \Big].\\
%\end{eqnarray*}
%Applying once again  Schwarz's inequality to both expectations in the above inequality yields
%%with exponent $p_{\nu}=1+\nu$ and $q_{\nu} = 1 + \frac{1}{\nu}$  yields
% \begin{eqnarray*}    
% \Vert R_k \Vert_2^2 
 & \leq &  \big( M_{k,\nu}\big)^2 \Big( [g_{k+1}^2]^2_{\rm Liploc} \Vert X_{k+1 } - \hat X_{k+1}  \Vert_{2(1+\nu)}^2 + [g_{k+1}^1]^2_{\rm Liploc}    \Vert X_k  - \hat X_k  \Vert_{2(1+\nu)}^2\Big)  \end{eqnarray*}
where $M_{k,\nu}:=   1+ \|\theta(X_k)\|_{ 2 q_{\nu}}    +  \|\theta(\hat X_k)\|_{2 q_{\nu}}   +   \|\theta(X_{k+1}) \|_{2 q_{\nu}}   +    \| \theta( \hat X_{k+1})\|_{2 q_{\nu}}   $.
% \begin{eqnarray}    
% \Vert  \hat{\mathds E}_k u_k(X_k)     - \hat{\mathds E}_k (\varphi_{k+1}(\hat X_k, \hat X_{k+1}, X_{k+1})) \Vert_2^2 & \leq &  \Vert  \varphi_{k+1}( X_k,  X_{k+1}, X_{k+1})     - \varphi_{k+1}(\hat X_k, \hat X_{k+1}, X_{k+1}) \Vert_2^2 \nonumber  \\
% & \leq & 2 \Vert  u_{k+1} \Vert_{\sup}^2 \Big([g_{k+1}^1]^2_{\rm Lip} \Vert  X_k - \hat X_k \Vert_2^2 
% \nonumber \\ 
% &   & 
% + [g_{k+1}^2]^2_{\rm Lip} \Vert  X_{k+1} - \hat X_{k+1} \Vert_2^2 \Big).  \label{EqBoundG_k}
%    \end{eqnarray}

Plugging these bounds in~\eqref{eq:IneqFiltrageQ}, we finally get that, for every $k=0,\ldots,n-1$,
$$  
\Vert  u_k(X_k) - \hat{u}_k(\hat X_k) \Vert_2^2      \leq  \widetilde K  \Vert  u_{k+1}(X_{k+1}) - \hat{u}_{k+1}(\hat X_{k+1}) \Vert_2^2 + \alpha_k  \Vert  X_k - \hat X_k \Vert_{2(1+\nu)}^2 + \beta_{k+1}  \Vert X_{k+1} - \hat X_{k+1}\Vert_{2(1+\nu)}^2 
$$
where $ \widetilde K= 2(K_g)^2$
% 
%we set $\Delta_k = X_k - \hat X_k$, 
 \begin{eqnarray*}    
     \alpha_{k}& :=  &  (M_{k,\nu})^2\big([u_k]^2_{\rm Liploc}   + 2    \Vert  \hat u_{k+1} \Vert_{\infty}^2 [g_{k+1}^1]^2_{\rm Liploc}\big), \quad 0 \leq k \leq n,\hskip 2,1cm \\
 \textrm{and  }  \hskip 3,6cm \beta_{k} &:=&  2  \,   \big(M_{k,\nu} \Vert \hat  u_{k} \Vert_{\infty} [g_{k}^2]_{\rm Liploc}\big)^2 , \qquad  \qquad  \qquad \quad \quad \; 1 \leq k \leq n,\hskip 2,2cm
 \end{eqnarray*}
(we set $u_{n+1}=0$ by convention so that  $\alpha_{n}:=  \big([f]_{\rm Liploc}M_{n,\nu})^2$). It follows by induction that
% for every $k\!\in\{0,\ldots,n\}$, 
  \[
   \|u_k(X_k)- \hat u_k(\hat X_k)\|^2_2  \leq \frac{1}{ \widetilde K^{k}} \sum_{\ell=k}^nC_{\ell,n}(f,y)  \Vert  X_\ell - \hat X_\ell \Vert_{2(1+\nu)}^2,\;k=0,\ldots,n,\\  
 \] 
  where, using the upper-bound for $[u_\ell]_{\rm Liploc}$ given by~\eqref{eq:Lesaell} (and the definition of $\kappa_{g,X}$ that follows),  
  \begin{eqnarray*}
 C_{\ell,n}(f,y)& : = &  \widetilde K^{\ell-1} (\alpha_{\ell}  \widetilde K + \beta_{\ell}), \; \ell=0,\ldots,n  \\
  & = &   2^{\ell} (M_{\ell,\nu})^2 \left( (K_g)^{2\ell}  [u_{\ell}]^2_{\rm Liploc}  +  (K_g)^{2(n-1)} \Vert f \Vert_{\infty}^2  (2[g_{\ell+1}^1]_{\rm Liploc}^2 + [g_{\ell}^2]_{\rm Liploc}^2) \right)  \\
  & \leq & 2^{\ell+1} (M_{\ell,\nu} (K_g)^{n})^2  \Bigg[\left(\kappa_{g,X}^2\Bigg(\frac{[P]_{\rm Liploc}^{n-\ell}-1}{[P]_{\rm Liploc} -1}\right)^2 +   \frac{2[g^1]_{\rm Liploc}^2 +  [g^2]_{\rm Liploc}^2}{2(K_g)^2} \Bigg)\|f\|_{\infty}^2 \\
  && \hskip 9,5 cm +\,[P]^{2(n-\ell)}_{\rm Liploc} [f]^2_{\rm Liploc}   \Bigg]
%  \\ 
% &  &\left. +  \Big(\sum_{m=1}^{n-\ell}[P]_{\rm Liploc}^{m-1}([g_{\ell+m}^1]_{\rm Liploc} C_{\theta,X} + [P]_{\rm Liploc}[g_{\ell+m}^2]_{\rm Liploc}) \Big)^2  \Big)\right].\hskip 2cm 
     \end{eqnarray*}
(we also used the elementary inequality $ab\le \frac 12(a^2+b^2)$, $a,b\ge 0$ in the third line).    
% Using that $A^2(f)\le 2\big( [f]^2_{\rm Liploc}+ \kappa_{P,g}^2\big)$with $\kappa_{P,g}= \dots$.
Finally 
\begin{eqnarray*}
 \vert \pi_{y,n}f -\hat{\pi}_{y,n}f \vert^2&=& | \mathds E\, u_0(X_0) - \mathds E\,\hat u_0(\hat X_0)|^2\\
  &\le&\| u_0(X_0)-  \hat u_0(\hat X_0)\|^2_2   \\
  &\le& \big((K_g)^n   M_{n,\nu}\big)^2\big(R_{y,n} \|f\|^2_{\infty} +S_{y,n} [f]^2_{\rm Liploc}\big)
\end{eqnarray*}
where
\[
R_{y,n}=  \sum_{\ell=0}^n \left[2^{\ell+1}\left(\kappa_{g,X}^2\left(\frac{[P]_{\rm Liploc}^{n-\ell}-1}{[P]_{\rm Liploc} -1}\right)^2+   \frac{2[g^1]_{\rm Liploc}^2 +   [g^2]_{\rm Liploc}^2}{2(K_g)^2} \right)\right]     \Vert  X_\ell - \hat X_\ell \Vert_{2(1+\nu)}^2 
\]
and
\[
S_{y,n}=  \sum_{\ell=0}^n 2^{\ell+1}    [P]_{\rm Liploc}^{2(n-\ell)}  \Vert  X_\ell - \hat X_\ell \Vert_{2(1+\nu)}^2.
\]
 We conclude by Lemma~\ref{LemNLFilter}.
 \end{proof}

% \subsection{The locally Lipschitz case}
%We suppose here that $(X_k)_{0\le k\le n}$ is the   Euler-Maruyama scheme with step $\Delta_n =\frac Tn$.  We show that the previous rate still holds under  the following (weaker)  

%%\noindent   $(\mathscr{A}_2) \equiv $ The diffusion coefficient $\sigma$ is uniformly elliptic: there is $\lambda_0>0$ such that for every $(t,x,\zeta) \in [0,T]{\small \times} \mathds R^d {\small \times} \mathds R^d$,  
%%  \begin{equation}
%%   \lambda_0^{-1} \vert \zeta \vert^2  \leq (\sigma\sigma^T(t,x) \zeta, \zeta) \leq \lambda_0 \vert  \zeta \vert^2. 
%%  \end{equation}
%
% 
%%   The function  $\sigma \sigma^T$ is uniformly  $\eta$-H\"{o}lder continuous  in space, uniformly in time and that $b$ is bounded: their is $L>0$ such that   
%%\begin{equation}
%% \sup_{(t,x) \in [0,T]{\small \times} \mathds R^d} \vert b(t,x) \vert + \sup_{t \in [0,T], (x,z) \in \mathbb R^{2d}, x\not=z} \frac{\vert \sigma \sigma^T(t,x) - \sigma \sigma^T(t,z) \vert}{\vert  x - z \vert^{\eta}} \leq L.  
%% \end{equation}

%  The previous theorem shows the usefulness of the distortion mismatch result, which in different contexts may be used to improve  several results involving  the quantization errors in the asymptotic framework.  In our context, it allows us to weaken the assumptions on the functions $g_k$ whereas for the maximal radius problem (see~\cite{PagSag, Jun}), its use is crucial to derive the sharp constant for the asymptotic of the maximal radius sequence of quantizers when considering distributions with radial exponential tails. 
  
 The previous theorem highlights the usefulness of the distortion mismatch result: it allows  to switch from Lipschitz continuous assumptions on the functions $g_k$ into {\em local Lipschitz} assumptions.

\begin{rem} Note  that if we consider Assumption   $(\mathscr{H}_0)$ instead of Assumption  $(\mathscr H_{\rm Liploc})$ in Theorem~\ref{ThmLocalLips}  {\em we still improve the upper bound established in Theorem~3.1  of~\cite{PagPha} since}  this amounts to setting $\theta \equiv 0$ and replacing everywhere  the  ``$[.]_{\rm Liploc}$" coefficients by $[.]_{\rm Lip}$. Then, like for BSDEs, the squared global error appears  as the (weighted) cumulated sum of the squared quantization errors.
%
% we proved  that for every bounded Lipschitz continuous function $f$ on $\mathds R^d$,  
%\textcolor{red}{NON ACTUALISE}
%\begin{equation}
%\vert \Pi_{y,n} f -\hat{\Pi}_{y,n} f \vert  \leq  \frac{2 K_g^{n}}{\phi_n(y) \vee \hat{\phi}_n(y)} \left(\sum_{k=0}^n  B_k^n(f,y) \Vert X_k - \hat X_k \Vert_2^2\right)^{\frac 12}
%\end{equation}
%where $ \phi_{n}(y) = \pi_{y,n} \mbox{\bf 1}$ and  $\hat {\phi}_{n}(y) = \hat{\pi}_{y,n} \mbox{\bf 1}$
%and 
%\begin{equation*}  
% B_k^n(f,y) : = 2 [P]_{\rm Lip}^{2(n-k)}[f]_{\rm Lip}^2  +   2  \Vert f \Vert_{\infty}^2 R_{n,k} +   \Vert f \Vert_{\infty} R_{n,k}^2
% \end{equation*}
% with 
%  $$ 
%  R_{n,k} =  \frac{1}{K_g^2}  \Big[[g_{k+1}^1]^2 + [g_{k}^2]^2    +      \Big(\sum_{m=1}^{n-k}[P]_{\rm Lip}^{m-1}([g_{m+k}^1]_{\rm Lip} + [P]_{\rm Lip}[g_{m+k}^2]_{\rm Lip}) \Big)^2  \Big].  
%  $$
%  \textcolor{red}{ACTUALISE}
%  In fact, in this case, the upper bound~\eqref{EqControlRk} may \textcolor{red}{even} be replaced by 
%   \begin{eqnarray*}    
% \Vert  \Delta_k \Vert_2^2 & \leq &  \Vert  \varphi_{k+1}( X_k,  X_{k+1}, X_{k+1})     - \hat u_{k+1}(\hat X_{k+1})g_{k+1}(\hat X_k, \hat X_{k+1}) \Vert_2^2 \nonumber  \\
% & \leq & 2 \Vert \hat   u_{k+1} \Vert_{\infty}^2 \Big([g_{k+1}^1]^2_{\rm Lip} \Vert  X_k - \hat X_k \Vert_2^2  + [g_{k+1}^2]^2_{\rm Lip} \Vert  X_{k+1} - \hat X_{k+1} \Vert_2^2 \Big).  \label{EqBoundG_k}
%    \end{eqnarray*}
\end{rem}

%\newpage
\small
\bibliographystyle{plain}
\bibliography{NLfilteringbib}
%\end{document}
%****************************
%
%FIN VERSION SYNTHETIQUE
%
%*****************************

\begin{comment}
\newpage
\centerline{\bf Answer to the referee ($SPA-2015-92R3$)}

 First we would like to thank the referee for this new  reviewing of the manuscript. 
 %We ``suspect" that he/she  is the AE or a ``new'' reviewer. If such is the case, we are grateful for taking this task in charge.
 
 We corrected  all typos, missing definitions of notations (including Point 19. about $\underline s$)  and modified all detected awkward formulations  as suggested by the referee.

 \medskip We answer below in more details  to points $10$, $20$, $21$, $27$, $31$, $37$, $38$ and $43$. We temporarily  left the most significant modifications in blue throughout the manuscript.

 \begin{itemize}
 
 \item Point $10$. We did adopt the word ``projection'' though at this stage $q_k$ (now denoted by $\pi_k$) can be {\em any} $\Gamma_k$-valued Borel function. Of course, optimal quantization theory leads to consider in practice a Voronoi projection.
  
 \item Point $20$. We thank the referee for the suggested references. In fact, even in our extended proof posted on ArXiv we do not provide an original proof for this claim but cite an old) reference by Bouchard-Touzi (\cite{BouTou}). We followed your  recommendation and updated claim~$(b)$ in order to present a less stringent assumption on the regularity of the  terminal value $h$ (though it induces a slight increase of that of the driver). We also mentioned the third result which only requires a fractional regularity on $h$ right after the theorem.
 
 \medskip 
 \item Point $21a$. We carefully checked the proof of Theorem~3.2 and came to the conclusion that the formula $K_n(b,\sigma, f)= [h]^2_{\rm Lip} $ is correct. The key inequality is Equation (49) and the lines that follow which are simply a discrete time Gronwall lemma.
 
  \medskip 
 \item Point $21b$.  Our choice to introduce $n_0$ is probably questionable. Our motivation is to provide sharper constants: the constants $K_i(b,\sigma, T,f)$ that come out in the computations are in fact depending on $n$ in a decreasing way through the time step $\Delta_n$. Rather than introducing misleading ``$n$-dependent constants'' $K_i(b,\sigma, T,f,n)$ which would have been decreasing in $n$, we adopted this presentation relying on an exogenous ``starting" discretization step $\Delta_{n_0}$. 
 
 We agree that our presentation was not clear enough  and we modified the presentation of the theorem to take into account the comments by the referee and added a very short comment inside the statement of the theorem to highlight what $n_0$ is. This  is of course reversible.
 
  \medskip 
 \item Point $27$.   In fact there is no error: we first deal in these equations with the second term on the right hand side of the former  inequality. We hopefully made this point clearer, now.
 
   \medskip 
 \item Point $31$. We do not understand either why $\frac 12$ stands there (in fact twice in the paper). It seems to be an error  induced by the fact that the two co-authors do not use the same encoding system \dots  
 Anyway we cancelled this inappropriate coefficient.
    \medskip 
 \item Point $37$-$38$. We propose to maintain the verb ``propagate''  to qualify the property shared by the transitions $P_k(x,dy)$ since what we mean is : if $f$ is $\theta$-locally Lipschitz, then $x\mapsto Pf(x)$ is again $\theta$-Lipchitz (with a uniform control on the $\theta$-Lipschitz coefficient). 

 \medskip 
 \item Point $43$.  In fact the observations, denoted by $Y=(Y_0, \ldots,Y_n)$ or $y=(y_0,\ldots,y_n)$ in the original filtering model no longer appear in the computations. They are dropped in the notation of $g_k$, as well as the reference to $f$ in the notation of $u_k$, etc. The aim is  to alleviate the notations as it was briefly  mentioned prior to Proposition~6.2. In fact, this  is  a quite classical habit in the field of filtering. However we tried to  make this convention more visible and clear. 
 After checking carefully Equation~(75) at the top of page 30, we think it is correct. However, the definition of the function $\varphi_{k+1}$ was misleading and we modified it to avoid confusion. 
 \end{itemize}

We hope this new version will comply with the recommendations of the referee.
\end{document}